\documentclass[twoside]{article}
\usepackage[accepted]{aistats2026}

\usepackage[utf8]{inputenc} 
\usepackage[T1]{fontenc}    
\usepackage{hyperref}       
\usepackage{xurl}            
\usepackage{booktabs}       
\usepackage{mathtools,amssymb,amsmath,amsthm}
\usepackage{amsfonts}       
\usepackage{enumitem}
\usepackage{nicefrac}       
\usepackage{microtype}      
\usepackage{xcolor}         

\usepackage{caption}

\newcommand{\repeatcaption}[2]{%
  \renewcommand{\thefigure}{\ref{#1}}%
  \captionsetup{list=no}%
  \caption{#2 (repeated from page \pageref{#1})}%
  \addtocounter{figure}{-1}}

\usepackage{subcaption}

\usepackage{pgfplots}
\DeclareUnicodeCharacter{2212}{−}
\usepgfplotslibrary{groupplots,dateplot}
\usetikzlibrary{patterns,shapes.arrows}
\pgfplotsset{compat=newest}

\newlength{\figWidth}
\newlength{\figHeight}
\newcommand{\legendDist}{0.35}
\newcommand{\legendEntryFontSize}{\scriptsize}

\usepackage{stmaryrd}

\DeclareMathOperator*{\argmin}{arg\,min}  
\DeclareMathOperator{\tr}{tr}

\renewcommand{\leq}{\leqslant}
\renewcommand{\geq}{\geqslant}
\renewcommand{\le}{\leqslant}
\renewcommand{\ge}{\geqslant}
\renewcommand{\succeq}{\succcurlyeq}
\renewcommand{\preceq}{\preccurlyeq}

\newcommand{\cK}{\mathcal{K}}
\newcommand{\domain}{\mathcal{K}}

\usepackage{algorithm}
\usepackage{algpseudocode}

\theoremstyle{plain}
\newtheorem{Th}{Theorem}[section]

\newtheorem{theorem}[Th]{Theorem}
\newtheorem{lemme}[Th]{Lemma}
\newtheorem{lemma}[Th]{Lemma}

\theoremstyle{definition}
\newtheorem{remark}[Th]{Remark}

\newtheorem*{remark*}{Remark}

\definecolor{crimson2143940}{RGB}{214,39,40}
\definecolor{forestgreen4416044}{RGB}{44,160,44}
\definecolor{mediumpurple148103189}{RGB}{148,103,189}

\begin{document}

\twocolumn[

\aistatstitle{Optimized Projection-Free Algorithms for Online Learning: Construction and Worst-Case Analysis}

\aistatsauthor{ Julien Weibel \And Pierre Gaillard }

\aistatsaddress{ 
  Inria, D.I. École Normale Supérieure, CNRS\\ PSL Research University\\
  \texttt{julien.weibel@inria.fr} 
  \And 
  Université Grenoble Alpes, Inria, CNRS\\ Grenoble INP, LJK\\
  \texttt{pierre.gaillard@inria.fr}
  } 
 
  \aistatsauthor{  Wouter M. Koolen \And Adrien Taylor} 

  \aistatsaddress{ 
  Centrum Wiskunde \& Informatica, Amsterdam\\
  and Twente University, Enschede\\
  \texttt{wmkoolen@cwi.nl}
  \And 
  Inria, D.I. École Normale Supérieure, CNRS\\ PSL Research University\\
  \texttt{adrien.taylor@inria.fr}
  } 
 
\runningauthor{Julien Weibel, Pierre Gaillard, Wouter M. Koolen, Adrien Taylor}

  ]

\begin{abstract}
This work studies and develops projection-free algorithms for online learning with linear optimization oracles (a.k.a.~Frank--Wolfe) for handling the constraint set,
and for convex loss functions. More precisely, this work 
(i)~shows how to exploit semidefinite programming to jointly design and analyze online Frank--Wolfe-type algorithms numerically in a variety of settings,
(ii)~leverages those design techniques to propose an improved (optimized) variant of an online Frank--Wolfe algorithm along with its conceptually simple potential-based proof,
and (iii)~extends this proof to its anytime version, which benefits from a similar $O(T^{3/4})$ regret rate without requiring knowledge of the time horizon $T$ in advance.
We are not aware of other direct regret guarantees for an anytime version of online Frank--Wolfe 
without using the classical doubling trick.

Based on the semidefinite technique, we conclude with strong numerical evidence suggesting that no pure online Frank--Wolfe algorithm within our model class can have a regret guarantee better than $O(T^{3/4})$  without additional assumptions, that the current algorithms do not have optimal constants, and that multiple linear optimization rounds do not generally help to obtain better regret bounds.
\end{abstract}

\abovedisplayskip=5.0pt plus 1.0pt minus 2.5pt 
\belowdisplayskip=5.0pt plus 1.0pt minus 2.5pt
\abovedisplayshortskip=0.0pt plus 1.5pt
\belowdisplayshortskip=3.0pt plus 1.5pt minus 1.5pt

\section{\uppercase{Introduction}}

This work considers the online learning problem, where we sequentially query points $x_1,x_2,\ldots,x_T$ in a domain $\domain$. 
At each time step $t$, we play $x_t$, incur a cost $\ell_t(x_t)$ which we aim to minimize by convention, 
and observe a gradient $g_t=\nabla \ell_t(x_t)$. 
The quality of our guesses is evaluated by comparison with a reference $x_\star\in \domain$. 
In this context, we aim to estimate (or bound) the \emph{cumulative regret} $R_T$ incurred by our choices $x_1,\ldots,x_T$ as compared to a reference point $x_\star\in\domain$:\vspace*{-1pt}
\begin{align*}
R_T(x_1,\ldots,x_T; x_\star)\triangleq & \sum_{t=1}^T\Big\{ \ell_t(x_t)- \ell_t(x_\star)\Big\} \\
 \leq & \sup_{x\in\domain}\Bigg\{\sum_{t=1}^T \ell_t(x_t)- \sum_{t=1}^T \ell_t(x)\Bigg\}.
\end{align*}

In many applications---such as online advertising, sensor networks, or mobile user applications---data is acquired and processed in real time, arriving as a continuous, high-rate flow. This necessitates the adoption of online learning methods, which aim to rapidly integrate large volumes of data as they arrive.
Online learning algorithms and regret bounds form essential frameworks for the theoretical study and optimization of reinforcement learning algorithms~\cite{kaelbling1996reinforcement,wang2022deep}, recommender systems~\cite{bobadilla2013recommender}, or for forecasting time series using expert advice~\cite{cesa2006prediction}. 
These methods find applications in areas like load forecasting~\cite{devaine2013forecasting}, finance and portfolio selection~\cite{li2014online}, generative adversarial networks~\cite{kodali2017convergence}, and, more recently, large language models~\cite{park2024llm}. Online learning has been extensively studied, and we refer to the classical work~\cite{hazanIntroductionOnlineConvex2016} and its references for a comprehensive overview of the topic.

\paragraph*{Assumptions and problem setup.} In this work, we consider the standard situation where the loss functions $\ell_t$ are convex with bounded gradients $\|g_t\|\leq L$ for all $t\geq 1$ for simplicity. Similarly, we assume that the domain $\cK$ is a closed convex bounded set of diameter~$D$. Variations around those assumptions are discussed in Section~\ref{sec:PEP_for_OFW} and in the appendix.

\paragraph{Related Work.} As discussed in classical textbooks (see, e.g.,~\cite{hazanIntroductionOnlineConvex2016,orabona2019modern}), a few \emph{meta} algorithms drive the basis and intuitions behind most online learning schemes, which can often be seen as appropriate approximations to the meta algorithms. Prominent examples include the "follow the leader" (FTL) and "follow the regularized leader" (FTRL) methods, as well as online gradient descent (OGD) \cite{zinkevich2003online}. These algorithms offer favorable regret guarantees of order $O(\sqrt{T})$. However, classical OGD-type methods rely on projections to manage the constraint set $\domain$. In numerous applications, such as matrix completion or recommender systems \cite{hazanIntroductionOnlineConvex2016}, these projections are computationally expensive and constitute a significant efficiency bottleneck. This challenge has led to the development of projection-free algorithms, which replace projections with potentially cheaper oracles, allowing for faster computations. In offline smooth convex optimization over polyhedral sets, the pioneering projection-free algorithm was introduced by Frank--Wolfe~\cite{frank1956algorithm}, that resort to a linear optimization oracle. In the context of online convex optimization, the first such algorithm was proposed by \cite{kalai2005efficient}, but it was limited to linear losses. Subsequently, \cite{hazan2012projection} introduced an online version of the Frank--Wolfe algorithm—referred to as online Frank--Wolfe (OFW)—which guarantees a regret upper bound of order $O(T^{3/4})$ for convex losses. This regret rate is less favorable compared to standard online learning algorithms that permit projections. As a result, considerable effort has been directed toward achieving improved regret guarantees for online projection-free algorithms. Some variants have attained the optimal regret rate of $O(\sqrt{T})$ by using membership oracles, which differ from the linear optimization oracles initially used in Frank--Wolfe and may be less efficient in certain scenarios \cite{levy2019projection, mhammedi2022efficient}. For linear optimization oracles, 
regret rates better than $O(T^{3/4})$
have been achieved under additional assumptions.
\cite{hazan2020faster} introduced a stochastic algorithm with $O(T^{2/3})$  expected regret for smooth functions.
\cite{Xie_Shen_Zhang_Wang_Qian_2020} proposed a stochastic algorithm with $\tilde{O}(\sqrt{T})$ expected regret for smooth  functions but with $O(T)$ linear optimization oracle calls per round.
\cite{Wan_Zhang_2021} proposed a variant of OFW for strongly convex feasible sets with $O(T^{2/3})$ regret for convex losses and $O(\sqrt{T})$ regret for strongly convex losses.

Despite these advancements, it remains an open question whether online Frank--Wolfe variants, relying on linear optimization oracles, are inherently limited to worst-case regret rate of $\Omega(T^{3/4})$ or if improved rates are possible in our setting.

\paragraph{Algorithmic setup.} This work focuses on the specific case of online learning using projection-free algorithms with access to the domain $\domain$ only through a linear optimization oracle.  At each round $t$, we play $x_t\in\domain$, incur some loss $\ell_t(x_t)$, and observe a gradient $g_t = \nabla\ell_t(x_t)$. We introduce an online procedure that encompasses a family of online Frank--Wolfe-type algorithms and unfolds as follows
\begin{align}
& \mathrm{dir}_t =  \sum_{s=1}^t \eta_{t,s}\, g_s 
                + \sum_{s=1}^{t-1} \beta_{t,s}\, (v_s-x_1) \nonumber \\[-.01cm]
& v_t = \argmin_{v\in \domain} \langle \mathrm{dir}_t , v \rangle \label{eq:OFW_general} \\[-.2cm]
& x_{t+1} = x_1 + \sum_{s=1}^t \gamma_{t+1,s}\, (v_s - x_1), \nonumber 
\end{align}
where the algorithm is parameterized by $\{\eta_{t,s}\}_{1\leq s \leq t \leq T}$ (to combine previous gradient information),~$\{\beta_{t,s}\}_{1\leq s < t \leq T}$ (to pick the linearization point $\bar{x}_t$), 
and~$\{\gamma_{t,s}\}_{1\leq s < t \leq T}$ to combine the previous \emph{atoms} $v_t$ (extreme points of $\domain$) of the linear optimization procedure.
In other words, to choose the next query point $x_{t+1}\in\domain$, we form the intermediate objective function $f_t(x)=\langle \sum_{s=1}^t \eta_{t,s}\, g_s,\,x\rangle+\tfrac12\|x-x_1\|^2_2$, approximate it by the linear function $x \mapsto \langle x, \nabla f_t(\bar x_t) \rangle$ in the point $\bar{x}_t = x_1 + \sum_{s=1}^{t-1}\beta_{t,s}\, (v_s-x_1)$ for some sub-probability $\beta_t$, and use the linear optimization oracle to optimize that linear approximation over $\domain$. 
The only constraints on algorithm parameters in~\eqref{eq:OFW_general} are that $\gamma_{t,s}\geq 0$ for all $s<t\leq T$ and $\sum_{s=1}^{t-1} \gamma_{t,s} \leq 1$ for all $t\leq T$
(we do not constrain $\eta_{t,s}$ and $\beta_{t,s}$ 
for wider algorithmic design possibilities).
Note that translations of the domain $\domain$ and initial point $x_1$ leave the search direction $\mathrm{dir}_t$ invariant, while the oracle responses $v_t$ and iterates $x_t$ translate along.

The case \cite[Algorithm 1 \& Theorem 4.4]{hazan2012projection} corresponds to the choices 
$\eta_{t,s} = \frac{D}{2 L T^{3/4}}$,  $\beta_{t,s}=\gamma_{t,s}$, 
$\gamma_{t,t-1}= \frac{1}{\sqrt{t}}$
and $\gamma_{t,s} = \gamma_{t-1,s} (1-\gamma_{t,t-1})$
for $s<t-1$. As for \cite[Algorithm 27 \& Theorem 7.3]{hazanIntroductionOnlineConvex2016}, they correspond to nearly the same choices with $\gamma_{t,t-1}= \min(1,\frac{2}{\sqrt{t}})$.

\paragraph*{Contributions.}

We show how to leverage semidefinite programming to study (exact) worst-case regrets for online learning algorithms.
The method is based on forming a program maximizing the algorithm's regret over the problem instance (loss functions and feasible set), and then reformulating it losslessly as a convex semidefinite program (SDP).
Solving this SDP allows us to obtain (exact) worst-case regret values along with their associated worst-case instances.
In the optimization literature, those programs are commonly referred to as a \emph{performance estimation problem} (PEP, as coined by~\cite{droriPerformanceFirstorderMethods2014}).
Those PEP methods have been extensively studied in classical (offline) optimization settings, but they have not been explored so far for regret minimization with adversarial functions that may change at every time step.

Most works on online learning algorithms report numerical experiments on synthetic stochastic data or real datasets, even though their theoretical results are stated for worst-case settings (see, e.g., \cite{9325943,moondra2021reusing,wan2022online,pmlr-v162-zhang22af}).
Numerical experiments with this semidefinite programming method can bridge this gap 
allowing to investigate tightness of theoretical bounds.
Numerical worst-case regret values can also be used to assess performance of algorithms and identify those worth studying theoretically.
Proof schemes can then be inferred from numerical values of the SDP's dual variables (see Remark~\ref{rem_proofs_from_PEP}).

We exploit and exemplify the approach to study OFW algorithms,
whose analyses remain relatively technical \cite{hazan2012projection,hazanIntroductionOnlineConvex2016} in the literature, and whose known regret bounds of $O(T^{3/4})$ currently remain somewhat surprising. The numerical findings in Section~\ref{s:numerics} strongly suggest that this rate is non-improvable.
As a bonus, Appendix~\ref{sec:PEP_for_OGD_and_FTRL} shows how to apply the same methodological tools to compute tight worst-case regret bounds for OGD and FTRL.

Among the strong points of the approach, it conceptually allows one to optimize the algorithm parameters for better worst-case regret values (through min-max formulation---optimizing worst-case guarantees). However, such problems are naturally non-convex (bilinear matrix inequalities~\cite{toker1995np}). %
This work tackles the OFW design PEP problem by providing a numerically-justified tight relaxation for it, which essentially follows from a few simplifications in the original problem along with an appropriate change of variables. 
This gives rise to a convex SDP that can efficiently be solved numerically.
Most importantly, the feasible points of this new program can still be re-interpreted as  
OFW tunings and proofs. Similar relaxation techniques were, to the best of our knowledge, not used before beyond unconstrained offline gradient-based optimization. 
However, it is well-known that parametric closed-form solutions to SDPs are highly non-trivial in general (see, Remark~\ref{rem_proofs_from_PEP}),
which is unfortunately the case for the design SDP we obtain for OFW.
As a result, the design PEP cannot be directly used to obtain closed-forms for the optimal OFW tuning and its proof, or for a worst-case instance.

To circumvent this issue, we show how to
adapt the SDP techniques to exploit/force structure (i.e., simplicity)
in the resulting proofs, which we demonstrate with a certain potential-based proof strategy for OFW (see Section~\ref{sec:theo_proof}).
Combining this with numerical insights in a constructive approach (see Section~\ref{sec:design_proof_theo}), 
we obtain an optimized regret bound of $1.74 L D T$ for a new variant of OFW
(improving on the previous best known bound of $8 L D T$ from \cite{hazanIntroductionOnlineConvex2016}),
along with its particularly simple and tight tuning for OFW.
We then adapt this proof to an anytime variant tuning of OFW (where $T$ is replaced by $t$ in parameters), which to the best of our knowledge is the first
direct regret guarantee for an anytime variant of OFW .

Finally, we present numerical PEP evaluations of the exact worst-case regret of OFW,
up to moderate horizon $T=100$, for Hazan's tuning~\cite{hazanIntroductionOnlineConvex2016,hazan2012projection}, 
our new tuning, 
and the numerical optimal tuning from the design SDP
(Section~\ref{s:numerics}).
Those numerics support the following claims:
(a)~The design PEP allows to construct a likely minimax optimal OFW algorithm, and observe that its regret rate is $\Theta(T^{3/4})$ (implying a regret rate of $\Omega(T^{3/4})$ for all OFW tunings).
(b)~Our (potential-based) bound for our version of OFW 
is not tight among non-potential-based bounds
(best possible bound is about 2/3 of ours).
Note that our theoretical bound for our OFW remains better than the best possible bound for Hazan’s OFW (and Hazan's bound is also suboptimal).
(c)~The tight numerical bounds for our anytime tuning and our horizon-dependent tuning of OFW are close to each other.
(d)~Allowing for multiple rounds of linear optimization per iteartion in~\eqref{eq:OFW_general} only helps the constant in the regret bounds. 
(e)~We can obtain optimized variations around the same theme, e.g., using $\beta_{t,s}=0$ all the way. This allows obtaining parameter-free algorithms, at the cost of worsened regret rates.

\section{\uppercase{Constructive approach to optimized regret bounds}}\label{sec:PEP_for_OFW}

In this section we present how to leverage semidefinite programming~\cite{vandenberghe1996semidefinite}
to jointly design and analyze OFW-type
algorithms~\eqref{eq:OFW_general} 
numerically in a variety of settings. Section~\ref{s:wc_regret_construction} is dedicated to the computation of regret bounds, and Section~\ref{s:opt_param} focuses on joint parameter optimization. 
 
\subsection{Worst-case regret bounds via semidefinite programming}\label{s:wc_regret_construction}

The core idea for obtaining worst-case regret bounds (and corresponding examples) consists in casting the problem of computing the worst regret for a sequence generated by~\eqref{eq:OFW_general} as an optimization problem, as follows (we omit the $(D,L)$ dependence of $B_T$ for readability):
\begin{multline}\label{eq:PEP}
B_T(\{(\eta_{t,s},\beta_{t,s},\gamma_{t,s})\}_{t,s})\triangleq\\
\begin{aligned}
\sup_{\substack{\cK, \{\ell_t\}_{t\in\llbracket 1,T \rrbracket}\\ 
			x_\star, \{x_t\}_{t\in\llbracket 1,T \rrbracket}\\
			d\in\mathbb{N}}} \,
\!\!\!\!\!\!\!\!\!\!\!\!\!\!\!\! & \,\,\,\,\,\,\,\,\,\,\,\,\,\,\,\,
R_T(x_1,\ldots,x_T; x_\star)\\
\text{s.t. } 
&  \ell_t \text{ is convex and $L$-Lipschitz for }t\in\llbracket 1,T \rrbracket,\\
& \cK \text{ is a non-empty closed convex set of $\mathbb{R}^d$,}\\ 
&\mathrm{Diam}(\cK)\leq D,\\
& \{x_t\}_{t=1,\ldots,T} \text{ is generated by~\eqref{eq:OFW_general}}.\\
\end{aligned}
\end{multline}
This kind of problem seeks the worst-case dimension $d\in\mathbb{N}$, convex domain $\domain$ with bounded diameter, and sequence of Lipschitz convex functions $\{\ell_t\}_{t\in\llbracket 1,T \rrbracket}$ 
that together produce the largest possible regret
when those losses are evaluated at points $\{x_t\}_{t\in\llbracket 1,T \rrbracket}$ compatible with~\eqref{eq:OFW_general}. 
In the optimization literature, \eqref{eq:PEP} is commonly referred to as a \emph{performance estimation problem} (PEP, as coined by~\cite{droriPerformanceFirstorderMethods2014}) which can be reformulated losslessly into a convex SDP~\cite{taylor2017exact}, as briefly outlined below and further detailed in Appendix~\ref{a:pep}.

In a nutshell,~\eqref{eq:PEP} is a priori an infinite-dimensional problem (e.g., it includes functional variables such as the losses $\ell_t$). 
A key idea to reformulate~\eqref{eq:PEP} as a tractable problem consists in sampling the losses $\ell_t$ at the query points $x_t$ and $x_\star$, 
and to treat only the responses $(\nabla \ell_t(x_t), \ell_t(x_t))$ and $(\nabla \ell_t(x_\star), \ell_t(x_\star))$ as variables. 
By appropriately constraining these responses, we can force them to be compatible with some losses $\ell_t$ satisfying the desired assumptions (convexity and Lipschitzness of $\ell_t$). 
With a bit more details, it suffices to require that the samples are compatible through the subgradient inequalities
$\ell_t(x_t) - \ell_t(x_\star) \leq \langle g_t, x_t - x_\star \rangle$ for $t=1,\cdots, T$ and the Lipschitz bounds $\|\nabla \ell_t(x_t)\|\leq L$.
In particular, the linearized loss $f_t(x) = \ell_t(x_t) + \langle \ell_t(x_t), x - x_t \rangle$ satisfies the same subgradient inequality as $\ell_t$ and is also $L$-Lipschitz, 
but it yields a larger regret because $\ell_t(x_t) - \ell_t(x_*) \;\le\; f_t(x_t) - f_t(x_*).$
Moreover, since the algorithm only observes its gradient at $x_t$, the linearized loss $f_t(x)$ is indistinguishable from $\ell_t$ for the algorithm (since $\ell_t(x_t) = f_t(x_t)$ and $\nabla \ell_t(x_t) = \nabla f_t(x_t)$).
Thus, the worst-case regret is achieved by linear losses, and we restrict our attention to those losses without loss of generality.
As for handling the domain $\domain$, one possibility is to sample its indicator function at the query points $v_t$ and impose similar compatibility constraints (see, e.g.,~\cite[Theorem 3.6]{taylor2017exact}). 
Finally, the sampled version of~\eqref{eq:PEP} can be lifted to an SDP via a standard change of variables: all vectors and gradients appearing in~\eqref{eq:PEP} are replaced with their Gram matrix (which, recall, encodes all pairwise inner products between these vectors / gradients). 
The objective, as well as all constraints, then become linear functions of the entries of this Gram matrix. A more complete exposition is provided in Appendix~\ref{a:pep}.

\begin{remark}[Obtaining proofs from~\eqref{eq:PEP}]
\label{rem_proofs_from_PEP}
A key feature of~\eqref{eq:PEP} is that it enables the construction of \emph{algorithm-dependent lower bounds} on the worst-case regret by solving semidefinite programs. 
That is, for given numerical values of $T,L,D$ and the algorithm parameters, one can compute a worst-case example by solving a tractable convex problem. 
In order to obtain \emph{algorithm-dependent upper bounds} on the worst-case regret, a natural procedure consists in formulating the \emph{Lagrange dual} of~\eqref{eq:PEP} (which is also a semidefinite program; see, e.g.,~\cite{vandenberghe1996semidefinite,boyd2004convex}), whose feasible points naturally corresponds to upper bounds on the regret. 
In this context, \emph{finding a proof} consists in finding a feasible point to the dual problem~\cite{goujaud2023fundamental}. 
While algebraic techniques exist for solving such parametric semidefinite programs in closed form, they typically suffer from exponential complexity in the problem size, number of variables/constraints, and number of parameters~\cite{basuAlgorithms,naldi2025solving}. This motivates the search for \emph{simpler/structured} proofs---e.g., potential-based methods~\cite{bansalPotentialFunctionProofsFirstOrder2019}---as exemplified in the proof of Theorem~\ref{thm_upper_bound_OFW}; see Section~\ref{sec:design_proof_theo} for further discussion.
\end{remark}

\subsection{Jointly optimising algorithm parameters and regret bounds}\label{s:opt_param}

A natural path forward is to use~\eqref{eq:PEP} to obtain worst-case optimal algorithms. That is, by solving\begin{align*} 
\min_{\{(\eta_{t,s},\beta_{t,s},\gamma_{t,s})\}_{t,s}} & B_T(\{(\eta_{t,s},\beta_{t,s},\gamma_{t,s})\}_{t,s}) \\[-0.2cm]
\text{ s.t. } & \sum_{s=1}^{t-1}\gamma_{t,s}\leq 1,\, \gamma_{t,s}\geq 0.
\end{align*}
This problem can be formulated as a linear optimization problem with a bilinear matrix inequality constraint (see Appendix~\ref{a:stepsizeopt}), which is unfortunately NP-hard in general~\cite{toker1995np}. For this reason, we propose a slight relaxation of $B_T(\{(\eta_{t,s},\beta_{t,s},\gamma_{t,s})\}_{t,s})$ which corresponds to removing a few constraints from~\eqref{eq:PEP} (that we numerically observed to be inactive). More precisely, this relaxation is obtained through: (i)~we observe that all $x_t$ for $t=2, \cdots, T$ are in the convex hull of $x_1, v_1, \cdots, v_{T-1}$, and thus the domain constraints for $\cK$ are imposed only on vectors $x_1, v_1, \cdots, v_{T-1}, x_\star$; (ii)~we only keep the boundary constraints corresponding to the optimality of $v_t$ compared with $v_{t+1}, \cdots, v_{T-1}, x_\star$:
\begin{multline*}
W_T\bigl(\{\eta_{t,s},\beta_{t,s},\gamma_{t,s}\}_{t,s}\bigr)
\triangleq\\[-0.05cm]
\begin{aligned}
\sup_{\substack{\{g_t\}_{t \in \llbracket 1,T \rrbracket},\, x_\star, d\in\mathbb{N}\\ 
			\{(x_t, v_t, \mathrm{dir}_t)\}_{t \in \llbracket 1,T \rrbracket}
			}} \,
\!\!\!\!\!\!\!\!\! & \,\,\,\,\,\,\,\,\, \quad
\sum_{t=1}^T \langle g_t, x_t - x_\star \rangle \\
\text{s.t. } 
& \{(x_t, \mathrm{dir}_t)\}_{t \in \llbracket 1,T \rrbracket} \text{ compatible with~\eqref{eq:OFW_general}},\\
& \langle -\mathrm{dir}_t, u - v_t \rangle \leq 0 \text{ for all }t \in \llbracket 1,T-1 \rrbracket \\
& \qquad\qquad 	
	\text{ and } u\in\{v_{t+1},\cdots, v_{T-1}, x_\star\}, \\ 
&\mathrm{Diam}(\{ x_1, x_\star, v_1, \cdots, v_{T-1} \})\leq D,\\
&  \Vert g_t \Vert \leq L \text{ for }t=1,\ldots,T,\\
\end{aligned}
\end{multline*}
where the line ``\emph{$\{(x_t, \mathrm{dir}_t)\}_{t \in \llbracket 1,T \rrbracket}$ compatible with~\eqref{eq:OFW_general}}'' means that $x_t$ and $\mathrm{dir}_t$ can be substituted by their expressions in~\eqref{eq:OFW_general}, leading to $W_T\bigl(\{\eta_{t,s},\beta_{t,s},\gamma_{t,s}\}_{t,s}\bigr)\geq B_T\bigl(\{\eta_{t,s},\beta_{t,s},\gamma_{t,s}\}_{t,s}\bigr)$. For this problem, an appropriate change of variables allows us to cast
\begin{equation}\label{eq:jointstepsizeopt}
\begin{aligned}
     \min_{\{(\eta_{t,s},\beta_{t,s},\gamma_{t,s})\}_{t,s}} & W_T(\{(\eta_{t,s},\beta_{t,s},\gamma_{t,s})\}_{t,s}) \\[-0.2cm]
     \text{ s.t. } & \sum_{s=1}^{t-1}\gamma_{t,s}\leq 1,\, \gamma_{t,s}\geq 0
\end{aligned}
\end{equation}
as a convex semidefinite program again (details in Appendix~\ref{a:stepsizeopt}). Numerical results illustrating the performance of the numerically optimized methods are provided in~Figure~\ref{fig:Comparison_tau}; examples of numerically optimized parameters are further provided in Appendix~\ref{a:numerical_stepsizeopt}.

\section{\uppercase{Tuning and regret bound for online Frank--Wolfe}}
\label{sec:theo_proof}

In this section, we analyze the simple OFW scheme (Algorithm~\ref{OFW_alg_new}) with fixed parameters
using PEPs to search for simpler potential-based proofs.
In Section~\ref{sec:proof_pot_theo}, we prove a regret upper bound for our optimized
tuning of OFW.
Section~\ref{sec:theo_anytime} presents the extension to the anytime setting.
Section~\ref{sec:design_proof_theo} explains how the SDP formulations were leveraged to obtain the proof of Section~\ref{sec:proof_pot_theo} and the optimal tuning~\eqref{eq:parameter_choice}.

\vspace{-5pt}
\begin{algorithm}[H]
\caption{Online Frank--Wolfe algo. (fixed $\eta,\sigma$)}
\label{OFW_alg_new}
\begin{algorithmic}[1]
\Require $T\geq 1$,~ $x_1 \in \cK$, $\eta \geq 0$, $\sigma \in (0,1)$
\For{$t=1$ to $T$} 
    \State Play $x_t$, pay cost $\ell_t(x_t)$,  observe $g_t = \nabla\ell_t(x_t)$.
    \State $\mathrm{dir}_t \gets \eta \sum_{s=1}^t g_s + (x_t - x_1)$
    \State $v_t \gets \argmin_{v\in \cK} \langle \mathrm{dir}_t , v \rangle$
    \State $x_{t+1} \gets (1-\sigma) x_t + \sigma v_t$
\EndFor
\end{algorithmic}
\end{algorithm}
\vspace{-15pt}

\subsection{Simple tuning and regret bound}
\label{sec:proof_pot_theo}

Throughout this section, we fix \vspace*{-5pt} 
\begin{equation}
        {\eta=\frac{D}{2 L} \left(\frac{3}{T}\right)^{3/4} \quad \text{ and } \quad \sigma = \min\left(1, \sqrt{\frac{3}{T}}\right) }.
    \label{eq:parameter_choice}
\end{equation} 
We assume the time horizon is at least $T \ge 3$.
Note that tuning~\eqref{eq:parameter_choice} and 
the proof of the following theorem 
were obtained as the optimal solution of a 
design problem using PEPs for potential-based proofs, as explained in Section~\ref{sec:design_proof_theo}.

\begin{theorem}\label{thm_upper_bound_OFW}
Fix $T\geq 3$.
Assume that the cost functions $\ell_t$ are convex and $L$-Lipschitz for all $t\in \llbracket 1, T \rrbracket$,
and that the convex closed domain $\cK$ of feasible points has a diameter bounded by $D$.
Then, for any $x_\star \in\cK$, the following upper bound 
on the regret of the online Frank--Wolfe Algorithm~\ref{OFW_alg_new}, with parameters defined in~\eqref{eq:parameter_choice}, holds:
\begin{align*}
R_T \! & \leq \frac{2 D}{L 3^{3/4} T^{1/4}} \sum_{t=1}^T \Vert g_t \Vert^2
	+ \frac{L}{D 3^{3/4} T^{1/4}} \sum_{t=1}^T \Vert x_t - v_t \Vert^2 \\
& \qquad\qquad\qquad\qquad\qquad\!
	+ \frac{L T^{3/4}}{D 3^{3/4} } \Vert x_\star - x_1 \Vert^2 ,
\end{align*}
and in particular,  $R_T \leq \frac{4}{3^{3/4}} L D T^{3/4} < 1.76 L D T^{3/4}$.
\end{theorem}

Inspired by~\cite{bansalPotentialFunctionProofsFirstOrder2019},
see also~\cite{taylor2019stochastic,karimi2017single},
we use a potential-based proof to prove Theorem~\ref{thm_upper_bound_OFW},
which relies on the next lemma.
Motivated by the fact that Algorithm~\ref{OFW_alg_new} is seen as
an approximation of FTRL (see~\cite[Chapter 7]{hazanIntroductionOnlineConvex2016}), 
the next lemma will consider a potential that 
relates the OFW iterates to the iterates of FTRL, which we display in Algorithm~\ref{FTRL_alg}. To make the connection, it is necessary that OFW and FTRL encounter the same gradients. To ensure that, we formulate the below lemma for linear loss functions.

\begin{algorithm}[H]
\caption{Follow The Regularized Leader (FTRL)}
\label{FTRL_alg}
\begin{algorithmic}[1]
\Require $T\geq 1$,~ $y_1 \in \cK$, $\eta \geq 0$
\For{$t=1$ to $T$} 
    \State Play $y_t$, pay cost $\ell_t(y_t)$,  observe $g_t = \nabla\ell_t(y_t)$.
    \State $y_{t+1} \gets \argmin_{y\in \cK} \eta \langle  \sum_{s=1}^t g_s , y \rangle + \frac{1}{2} \Vert y - x_1 \Vert^2 $
\EndFor
\end{algorithmic}
\end{algorithm}

\begin{lemma}\label{lemma_upper_bound_OFW}
Let $T\geq 3$.
Denote by $\ell_t : x \mapsto \langle g_t, x \rangle$ 
the linear cost function at time $t\in \llbracket 1, T \rrbracket$.
Assume that $\Vert g_t \Vert \leq L$ for all $t\in \llbracket 1, T \rrbracket$
and that the convex closed domain $\cK$ of feasible points has a diameter bounded by~$D$.
Define the sequence of potentials $(\phi_t)_{0\leq t \leq T}$ as:
\begin{align*}
\phi_t \triangleq \sum_{s=1}^t \langle g_s, x_s - y_{t+1} \rangle & - \frac{1}{2\eta} \Vert y_{t+1} - x_1 \Vert^2 \\
	& + \frac{1}{6\eta} \Vert x_{t+1} - y_{t+1}\Vert^2  .
\end{align*}
Then, Algorithm~\ref{OFW_alg_new} and~\ref{FTRL_alg}, run with parameters in~\eqref{eq:parameter_choice} and initialized in $y_1 = x_1 \in \cK$, satisfy for all $  t\in \llbracket1, T\rrbracket$,
\begin{equation*}
   \phi_t - \phi_{t-1} 
\leq \frac{2 D}{L 3^{3/4} T^{1/4}} \Vert g_t\Vert^2 
	+ \frac{L}{ D 3^{3/4}  T^{1/4}} \Vert x_t - v_t \Vert^2 . \end{equation*}
\end{lemma}

Let us remark that the potential used in Lemma~\ref{lemma_upper_bound_OFW}
is inspired by the potential $\psi_t$ used to obtain the optimal upper bound
on the regret of FTRL 
(which gives a reformulation of classical FTRL regret proofs; 
see Appendix~\ref{a:proof_pot_FTRL} and, e.g.,~\cite[Chapters 6--7]{orabona2019modern}), 
where  $\psi_t$ is defined as:
\begin{equation}
\label{eq_def_pot_FTRL}
\psi_t \triangleq \sum_{s=1}^t \langle g_s, y_s - y_{t+1} \rangle 
            - \frac{1}{2\eta} \Vert y_{t+1} - y_1 \Vert^2 . 
\end{equation}

The proof of Lemma~\ref{lemma_upper_bound_OFW} is quite technical
and we defer it to Appendix~\ref{a:proof_lemma_upper_bound}.
In Section~\ref{sec:design_proof_theo}, 
we discuss how we designed the proof of Lemma~\ref{lemma_upper_bound_OFW}
and jointly obtained the optimal choice of 
the OFW parameters~\eqref{eq:parameter_choice}
and of the potential function $\phi_t$ used in Lemma~\ref{lemma_upper_bound_OFW}.

We now prove Theorem~\ref{thm_upper_bound_OFW}
which relies on Lemma~\ref{lemma_upper_bound_OFW}.

\begin{proof}[Proof of Theorem~\ref{thm_upper_bound_OFW}]
First, using convexity of the cost functions $\ell_t$
(recall $g_t = \nabla\ell_t(x_t)$), we get:\vspace*{-4pt} 
\begin{equation*}
R_T \leq \sum_{t=1}^T \langle g_t, x_t - x_\star \rangle. 
\end{equation*}
Then, summing the inequalities from Lemma~\ref{lemma_upper_bound_OFW}
applied to the linearized cost functions $\tilde\ell_t : x \mapsto \langle g_t, x\rangle$,
and observing that $\phi_0 = 0$ as $y_1=x_1$, we get:\begin{equation*}
\phi_T \leq \frac{2 D}{L 3^{3/4} T^{1/4}} \sum_{t=1}^T \Vert g_t \Vert^2
	+ \frac{L}{D 3^{3/4}  T^{1/4}} \sum_{t=1}^T \Vert x_t - v_t \Vert^2 .
\end{equation*}
Throughout the proof, we will repeatedly use the fact that the optimum of a constrained convex optimization problem 
$
x_* \in \arg\min_{x \in \mathcal{K}} f(x)
$
satisfies $\langle \nabla f(x_*), x_* - x \rangle \leq 0$ for all $x \in \mathcal{K}$. Using the optimality from the definition of $y_{T+1}$, we get:\vspace*{-5pt} 
\begin{multline*}
\sum_{t=1}^T \langle g_t, x_t - x_\star \rangle 
	- \frac{1}{2\eta} \Vert x_\star - x_1 \Vert^2\\
\leq \sum_{t=1}^T \langle g_t, x_t - y_{T+1} \rangle 
	- \frac{1}{2\eta} \Vert y_{T+1} - x_1 \Vert^2
\leq \phi_T .
\end{multline*}
Combining those three inequalities (recall $\eta = \frac{ 3^{3/4}D}{2 L T^{3/4}}$), we get:\vspace*{-6pt} 
\begin{align*}
R_T 
\leq \frac{2 D}{L 3^{3/4} T^{1/4}} \sum_{t=1}^T \Vert g_t \Vert^2
	& + \frac{L}{D 3^{3/4} T^{1/4}} \sum_{t=1}^T \Vert x_t - v_t \Vert^2 \\
	& + \frac{L T^{3/4}}{D 3^{3/4}} \Vert x_\star - x_1 \Vert^2.
\end{align*}
Hence, as $\Vert g_t \Vert \leq L$ 
and $\Vert x_t - v_t  \Vert \leq D$ for all $t\in \llbracket 1,T \rrbracket$
and $\Vert x_\star - x_1 \Vert \leq D$,
we get that $R_T \leq \frac{4}{3^{3/4}} L D T^{3/4}$.
This concludes the proof of Theorem~\ref{thm_upper_bound_OFW}.
\end{proof}

\subsection{Extension to an anytime OFW tuning}
\label{sec:theo_anytime}

The proof of Theorem~\ref{thm_upper_bound_OFW} can also
be adapted to an anytime version of the online Frank--Wolfe algorithm~\ref{OFW_alg_new} 
shown in Algorithm~\ref{OFW_alg_new_anytime} 
where the parameters $\eta_t$ and $\sigma_t$ 
for each time round $t$ are obtained by replacing the horizon $T$ by $t$ in the tuning~\eqref{eq:parameter_choice}, that is:
\begin{equation}\label{eq:parameter_choice_anytime}
\eta_t = \frac{D}{2 L} \left( \frac{3}{t} \right)^{3/4}
\quad \text{and} \quad
\sigma_t = \min\left(1, \sqrt{\frac{3}{t}} \right) .
\end{equation}
Note that parameters $\eta_t$ and $\sigma_t$ are anytime:
they depend on the time round $t$ but not on the horizon $T$.

\begin{algorithm}[H]
\caption{Anytime online Frank--Wolfe algorithm} \label{OFW_alg_new_anytime}
\begin{algorithmic}[1]
\Require  $x_1 \in \cK$, $\eta_t \geq 0$ and $\sigma_t \in (0,1)$
	\text{ for $t \geq 2$}
\For{$t=1$ to $T$ (or $\infty$)} 
    \State Play $x_t$, pay cost $\ell_t(x_t)$, observe $g_t = \nabla\ell_t(x_t)$.
    \State $\mathrm{dir}_t \gets \eta_{t+1} \sum_{s=1}^t g_s + (x_t - x_1)$
    \State $v_t \gets \argmin_{v\in \cK} \langle \mathrm{dir}_t , v \rangle$
    \State $x_{t+1} \gets (1-\sigma_{t+1}) x_t + \sigma_{t+1} v_t$
\EndFor
\end{algorithmic}
\end{algorithm}

Adapting the proof of Theorem~\ref{thm_upper_bound_OFW}
to the anytime setting of Algorithm~\ref{OFW_alg_new_anytime},
we get the following theorem, whose detailed proof can be found
in Appendix~\ref{sec:proof_OFW_anytime}.

\begin{theorem}\label{thm_upper_bound_OFW_anytime}
Assume that the cost functions $\ell_t$ are convex and $L$-Lipschitz,
and that the convex closed domain $\cK$ of feasible points has a diameter bounded by $D$.
Then, for any $x_\star \in\cK$, the following upper bound 
on the regret of the anytime online Frank--Wolfe Algorithm~\ref{OFW_alg_new_anytime}, with parameters defined in~\eqref{eq:parameter_choice_anytime}, holds (simultaneously for all $t$):
\begin{equation*}
\forall t \geq 1,\quad 
R_t \leq \frac{5}{3^{3/4}} L D t^{3/4} < 2.20 L D t^{3/4}.
\end{equation*}
\end{theorem}

\subsection{How we got the proof of Theorem~\ref{thm_upper_bound_OFW}}
\label{sec:design_proof_theo}

As explained in Remark~\ref{rem_proofs_from_PEP},
\eqref{eq:PEP} can be used to obtain rigorous regret upper bounds and their proofs for Algorithm~\ref{OFW_alg_new}. However, the algebraic structure of the problem to be solved proved itself quite challenging. In particular, we did not manage to extract a simple solution from~\eqref{eq:PEP}. Therefore, we restricted ourselves to search for optimized algorithms within the set of algorithms with simple structured proofs. A classical template for such proofs is that of relying on the construction of \emph{potential} (or Lyapunov) functions (see, e.g.,~\cite{bansalPotentialFunctionProofsFirstOrder2019} for a nice introduction). We thereby adapted~\eqref{eq:PEP} to help searching for appropriate potential functions.

The idea of a potential-based proof is to define a sequence of potentials
$(\phi_t)_{0\leq t \leq T}$ such that (i)~$R_T \leq \phi_T + A$
and (ii)~we can prove inequalities $\phi_t - \phi_{t-1} \leq B_t$ for all $t$,
where $A$ and $B_t$ are constants which can only depend on $L$, $D$, $t$ and $T$.
Then, we immediately obtain an upper bound on the regret $R_T$ as
$\smash{R_T \leq \phi_0 + A + \sum_{t=1}^T \phi_t - \phi_{t-1}
\leq \phi_0 + A + \sum_{t=1}^T B_t}$.
Under some good choice of potentials $\phi_t$,
each potential difference $\phi_t - \phi_{t-1}$
depends only on a small number of vectors, gradients
and sum of past gradients
(for instance, we will use $x_t$, $x_{t+1}$, $y_t$, $y_{t+1}$, $g_t$ 
and $\smash{G_{t-1} = \sum_{s=1}^{t-1} g_s}$). In other words, (ii)~consists in studying a single iteration of the procedure while (i)~ensures that we can combine those one-iteration analyses to form the global bound.

In such proofs, the first step is therefore to understand what information needs to be summarized in $\phi_t$. 
Here, motivated by the fact that Algorithm~\ref{OFW_alg_new} is seen as
an approximation of FTRL (see~\cite[Chapter 7]{hazanIntroductionOnlineConvex2016}), and by the
potential functions used for upper bounding FTRL's regret (see~\eqref{eq_def_pot_FTRL}, which is a reformulation of classical FTRL regret proofs; see, e.g.,~\cite[Chapters 6--7]{orabona2019modern}),
we tried the family of potentials parametrised by $a$ and $b$ as\vspace*{-3pt}
\begin{equation}
\label{eq_def_general_pot}
\begin{aligned}
\phi_t = 
& \sum_{s=1}^t \langle g_s, x_s - y_{t+1} \rangle
	- \frac{1}{2\eta} \Vert y_{t+1} - x_1 \Vert^2 \\
         & + a\, \Vert x_{t+1} - y_{t+1}\Vert^2
        + b\, \eta\, \langle G_t, x_{t+1} - y_{t+1} \rangle\\
        & + \frac{b}{2} \bigl( \Vert  x_{t+1} - x_1 \Vert^2 
                        - \Vert  y_{t+1} - x_1 \Vert^2 \bigr)
        .
\end{aligned}
\end{equation}
Note that $\phi_t - \phi_{t-1}$ is \emph{autonomous}: it is a function of $t$ only through its dependence in $x_{t}, x_{t+1}, y_t, y_{t+1}, g_t, G_{t-1}, x_1$. 
Because of this fact,
and as $x_{t+1}$, $y_t$ and $y_{t+1}$ are obtained
as autonomous functions of
$x_t$, $g_t$, $G_{t-1}$ and $x_1$ using 
Algorithm~\ref{OFW_alg_new} and FTRL,
it is possible to use the same proof for upper bounding $\phi_t - \phi_{t-1}$
for all iteration times $t$
(then, we get that $B_t$ is also autonomous,
which imposes that $B_t = B$ is a constant that depends only on $L$, $D$ and $T$).
Thus, we can use only $1$ iteration
of upper bounding $\phi_t - \phi_{t-1} \leq B$
for an abstract $t$ instead of doing an upper bound for each $t$
individually.
We also note that $\phi_0 = 0$ and $R_T - \frac{1}{2\eta}\Vert x_\star - x_1\Vert^2 \leq \phi_T$ for $a\geq 0$ and $b\geq 0$
(using optimality of $y_{T+1}$).

To jointly design the proof of Algorithm~\ref{OFW_alg_new}
(with fixed $\eta$ and $\sigma$)
and the potential~\eqref{eq_def_general_pot},
we use the following $1$-iteration variant of \eqref{eq:PEP}:\vspace*{-3pt}
\begin{equation}\label{eq:PEP-potential}
\begin{aligned}
\inf_{ a\geq 0, b\geq 0} \!\!\!\!
\sup_{\substack{ \cK,\, d\in\mathbb{N}\\
    g_t,\, G_{t-1} \\
    x_1,\, x_t,\, v_t,\, x_{t+1} \\ y_t,\, y_{t+1}  }} \,
\!\!\!\! 	\!\!\!\!\!\!\!\! & \,\,\,\,\,\,\,\,
 \phi_t - \phi_{t-1} \\
\text{s.t. } 
& \phi_t - \phi_{t-1} \text{ is generated by \eqref{eq_def_general_pot} from }\\ & \qquad
    \{x_t, x_{t+1}, y_t, y_{t+1}, g_t, G_{t-1}, x_1\} , \\ &  \Vert g_t \Vert \leq L ,\\
& \cK\subset\mathbb{R}^d \text{ is non-empty, closed, convex}, \\ 
&\mathrm{Diam}(\{x_1, x_t, v_t, x_{t+1}, y_t, y_{t+1} \})\leq D,\\
& (x_{t+1}, v_t) \text{ are generated from }\\ & \qquad \{x_1, x_t, g_t, G_{t-1}\} \text{ by Algorithm~\ref{OFW_alg_new}},\\
& y_t \text{ and } y_{t+1} \text{ are generated from }\\ & \qquad \{x_1, g_t, G_{t-1}\} \text{ by FTRL}.\\[-8pt]
\end{aligned}
\end{equation}
Note that \eqref{eq:PEP-potential} can be rewritten as
a semidefinite program whose size does not depend on $T$,
which allows efficient numerical solving even for large values of $T$.

Then, numerically solving \eqref{eq:PEP-potential} for fixed $\eta$ and $\sigma$
gives $b=0$ in~\eqref{eq_def_general_pot} and the values of Lagrange multipliers of the constraints
gives the proof structure of Lemma~\ref{lemma_upper_bound_OFW}
(with literal values for $\eta$, $\sigma$, $a$, 
and for two Lagrange multipliers)
up to the final step of verifying that the remaining terms 
form a sum of squares.
Hence, we are left with finding the optimal $\eta$ and $\sigma$
minimizing the regret upper bound given by~\eqref{eq:PEP-potential}
with the constraint that the remaining terms form a sum of squares,
which is a non-convex problem. We then relax this non-convex problem by keeping
only the leading order terms in the sum of squares constraint,
which we can then solve algebraically, resulting in the choices 
\begin{align*}
    &\eta = \frac{D3^{3/4}}{2LT^{3/4}}, \quad \sigma = \frac{\sqrt{3}}{\sqrt{T}}, \quad \text{and}\quad a = \frac{1}{6\eta},     \end{align*}
thereby concluding the construction of the result of Lemma~\ref{lemma_upper_bound_OFW}.
The details for rewriting~\eqref{eq:PEP-potential} as a semidefinite program,
numerically solving it, and then for obtaining the optimal parameters above
can be found in Appendix~\ref{a:proof_design}.

\section{\uppercase{Numerical results: tight regret bounds, direct stepsize optimization}}\label{s:numerics}

In this section, we leverage the techniques of Section~\ref{sec:PEP_for_OFW} to provide strong numerical evidence on the regret behaviors of different variations of online Frank--Wolfe-type algorithms.

The numerical experiments of this section rely on the CVXPY~\cite{JMLR:v17:15-408} modeling language used in combination with the MOSEK semidefinite solver~\cite{aps2019mosek} for solving~\eqref{eq:jointstepsizeopt}. For computing~\eqref{eq:PEP} we directly implemented the online algorithms within the PEPit software~\cite{goujaudPEPitComputerassistedWorstcase2024}.
In all those numerical experiments, we used $L=D=1$.
Those numerical experiments were performed on a
MacBook Pro 14" with M3 Pro SoC and 36GB of RAM
in a few tens of minutes for the largest time horizon
and in a few minutes for other values
(see Appendix~\ref{a:numerical_computation_times} for more details on the computational times),
note that RAM size was the main limiting factor.
(Our code used for those numerical experiments is available at 
\url{https://github.com/JulienWeibel/Optimized-projection-free-algorithms-for-online-learning-construction-and-worst-case-analysis}.)
Those numerical experiments strongly support the following claims, see Figure~\ref{fig:Comparison_tau}.
In Figure~\ref{fig:Comparison_tau}, we also compare the tight numerical bounds obtained from PEPs to known upper bounds for OFW, referring to their theorems in the caption of Figure~\ref{fig:Comparison_tau}.

\begin{figure*}[!t]
\begin{subfigure}{\textwidth}
    \centering

        \hfill
    \begin{minipage}[t]{0.48\textwidth}
        \vspace*{0pt}
        \centering
\begin{tikzpicture}

\definecolor{crimson2143940}{RGB}{214,39,40}
\definecolor{forestgreen4416044}{RGB}{44,160,44}
\definecolor{mediumpurple148103189}{RGB}{148,103,189}
\definecolor{darkgray176}{RGB}{176,176,176}
\definecolor{darkorange25512714}{RGB}{255,127,14}
\definecolor{lightgray204}{RGB}{204,204,204}
\definecolor{steelblue31119180}{RGB}{31,119,180}

\begin{axis}[
legend cell align={left},
legend style={
  fill opacity=1,
  draw opacity=1,
  text opacity=1,
  at={(0.5,-\legendDist)},
  anchor=north,
  draw=none
},
tick align=outside,
tick pos=left,
x grid style={darkgray176},
xlabel={Time horizon $T$},
xmin=-5, xmax=105,
xtick style={color=black},
y grid style={darkgray176},
ylabel={Worst-case regret},
ymin=-2, ymax=55,
ytick style={color=black},
  width=\figWidth,
  height=\figHeight,
  scale only axis,
 font=\small
]
\addplot [thick, crimson2143940, mark=asterisk, mark size=1, mark options={solid}]
table {%
1 8
2 13.4543426440594
3 18.2360564556382
4 22.6274169979695
5 26.7496121990569
6 30.6692690038211
7 34.4281365652708
8 38.0546276800871
10 44.9873060152279
15 60.9759297785538
20 75.6593287202541
25 89.4427190999916
30 102.548881535096
35 115.117412732146
40 127.243316602728
45 138.995062234592
50 150.424123723456
55 161.570475875318
60 172.46597374228
65 183.136539862808
70 193.603639399487
75 203.885309381655
80 213.996897592455
90 233.760899142071
100 252.98221281347
};
\addlegendentry{{\legendEntryFontSize \cite[Algo.~27]{hazanIntroductionOnlineConvex2016}: bound~\cite[Theorem 7.3]{hazanIntroductionOnlineConvex2016}}}
\addplot [thick, crimson2143940, dashed, mark=asterisk, mark size=1, mark options={solid}]
table {%
1 0.999999978829868
2 1.86602540835967
3 2.86602530022636
4 3.83521194338494
5 4.84566679047157
6 5.77847121423904
7 6.64296540269502
8 7.47558796478797
10 9.06556940344561
15 12.7147146818818
20 16.0829251259494
25 19.2687001005565
30 22.3084311637338
35 25.2313421308435
40 28.0532078096698
45 30.7885720243218
50 33.4461650859791
55 36.0337573951143
60 38.5568065885238
65 41.02483254
70 43.44430215
75 45.8196833
80 48.15440388
90 52.7148001953334
100 57.147182191147
};
\addlegendentry{{\legendEntryFontSize \cite[Algo.~27]{hazanIntroductionOnlineConvex2016}: tight bound $B_T$ from~(\ref{eq:PEP})}}
\addplot [semithick, forestgreen4416044, mark size=3, mark options={solid}]
table {%
1 1.75476535060332
2 2.95115178586752
3 4
4 4.9632259152112
5 5.86741157862262
6 6.72717132202972
7 7.55166264132207
8 8.34711776039087
10 9.8677707265638
15 13.3748060995284
20 16.5955460610261
25 19.6188730425514
30 22.493653007614
35 25.2505058891838
40 27.9102703837895
45 30.4879648892769
50 32.9948800255984
55 35.4397840933123
60 37.829664360127
65 40.17020682258
70 42.466119771115
75 44.7213595499958
80 46.939292628981
90 51.2744407675481
100 55.4905526705042
};
\addlegendentry{{\legendEntryFontSize Algo.~\ref{OFW_alg_new}: bound from Theorem~\ref{thm_upper_bound_OFW}}}
\addplot [semithick, forestgreen4416044, dashed, mark=x, mark size=2, mark options={solid}]
table {%
1 0.99999997622009
2 1.86602536857043
3 2.79508504207169
4 3.41695016277647
5 4.00578768309647
6 4.56827781991907
7 5.10944325015655
8 5.63752492727876
10 6.66119581166219
15 9.02755847964825
20 11.1903480646082
25 13.2176738631319
30 15.1477451800894
35 17.0005288149411
40 18.7895976211105
45 20.5247895195582
50 22.2138706190613
55 23.8627054385189
60 25.4759388445294
65 27.057093360461
70 28.6091703686733
75 30.1349387581324
80 31.6366960286442
85 33.1162504419378
90 34.5752635225472
};
\addlegendentry{{\legendEntryFontSize Algo.~\ref{OFW_alg_new}: tight bound $B_T$ from~(\ref{eq:PEP})}}
\addplot [semithick, mediumpurple148103189, dashed, mark=o, mark size=2, mark options={solid}]
table {%
1 0.999999999986411
2 1.73205077164806
3 2.34209616885784
4 2.90282823623805
5 3.42168738269308
6 3.91697957391197
7 4.38834045416619
8 4.84470501303823
9 5.28476977360727
10 5.71400805695854
11 6.13109156852881
12 6.53978673154998
13 6.93890009089812
14 7.33121252392302
15 7.71568171523579
20 9.55298861900087
25 11.2764639001485
30 12.91523643836
35 14.4863179235929
40 16.0021461745957
45 17.4709535774897
50 18.8995128351095
55 20.2925817455325
60 21.6543713495292
65 22.9879410641388
70 24.2961479561498
75 25.5810970007411
80 26.8448520074568
85 28.0889320029728
90 29.31489412696
95 30.5238801404065
100 31.7171067431816
};
\addlegendentry{{\legendEntryFontSize Optimized algo.~and bound from~(\ref{eq:jointstepsizeopt})}}
\end{axis}

\end{tikzpicture}
    \end{minipage}
    \hfill
    \begin{minipage}[t]{0.48\textwidth}
        \vspace*{0pt}
        \centering
\begin{tikzpicture}

\definecolor{crimson2143940}{RGB}{214,39,40}
\definecolor{darkgray176}{RGB}{176,176,176}
\definecolor{darkturquoise23190207}{RGB}{23,190,207}
\definecolor{forestgreen4416044}{RGB}{44,160,44}
\definecolor{goldenrod18818934}{RGB}{188,189,34}
\definecolor{lightgray204}{RGB}{204,204,204}
\definecolor{steelblue31119180}{RGB}{31,119,180}

\begin{axis}[
legend cell align={left},
legend style={
  fill opacity=1,
  draw opacity=1,
  text opacity=1,
  at={(0.5,-\legendDist)},
  anchor=north,
  draw=none
},
tick align=outside,
tick pos=left,
x grid style={darkgray176},
xlabel={Time horizon $T$},
xmin=-5, xmax=105,
xtick style={color=black},
y grid style={darkgray176},
ymin=-2, ymax=55,
ytick style={color=black},
  width=\figWidth,
  height=\figHeight,
  scale only axis,
 font=\small
]
\addplot [semithick, forestgreen4416044, mark=Mercedes star, mark size=3, mark options={solid}]
table {%
1 2.193456688254154
2 3.688939732334405
3 5.0
4 6.204032394013997
5 7.334264473278278
6 8.408964152537145
7 9.439578301652581
8 10.433897200488582
10 12.334713408204754
15 16.71850762441055
20 20.74443257628261
25 24.523591303189267
30 28.117066259517458
35 31.56313236147981
40 34.88783797973685
45 38.109956111596105
50 41.24360003199805
55 44.299730116640355
60 47.28708045015879
65 50.212758528225045
70 53.08264971389377
75 55.90169943749474
80 58.674115786226224
85 61.40351851407624
90 64.0930509594351
95 66.74546564983403
100 69.3631908381303
};
\addlegendentry{{\legendEntryFontSize Anytime Algo.~\ref{OFW_alg_new_anytime}: bound from Theorem~\ref{thm_upper_bound_OFW_anytime}}}
\addplot [semithick, crimson2143940, dashed, mark=asterisk, mark size=1, mark options={solid}]
table {%
1 0.999999978829868
2 1.86602540835967
3 2.86602530022636
4 3.83521194338494
5 4.84566679047157
6 5.77847121423904
7 6.64296540269502
8 7.47558796478797
10 9.06556940344561
15 12.7147146818818
20 16.0829251259494
25 19.2687001005565
30 22.3084311637338
35 25.2313421308435
40 28.0532078096698
45 30.7885720243218
50 33.4461650859791
55 36.0337573951143
60 38.5568065885238
65 41.02483254
70 43.44430215
75 45.8196833
80 48.15440388
90 52.7148001953334
100 57.147182191147
};
\addlegendentry{{\legendEntryFontSize \cite[Algo.~27]{hazanIntroductionOnlineConvex2016}: tight bound $B_T$ from~(\ref{eq:PEP})}}
\addplot [thick, crimson2143940, dashed, mark=+, mark size=2, mark options={solid}]
table {%
1 0.999999978829868
2 1.86602540820839
3 2.86602528722214
4 3.81941917172879
5 4.7843884768231
6 5.66753177316362
7 6.50107951868022
8 7.31139133654082
9 8.09325316957615
10 8.85139312035817
11 9.58774550759338
12 10.3054843155196
13 11.0069165395407
14 11.6938619340762
15 12.3679052961251
20 15.5818013204948
25 18.5962663625794
30 21.4601077826225
35 24.2075625818537
40 26.8596151684412
45 29.4313051952187
50 31.9318777341248
55 34.3709642044624
60 36.7548638838124
65 39.0892721371677
70 41.3798427025696
75 43.63062719759103
};
\addlegendentry{{\legendEntryFontSize Anytime \cite[Algo.~27]{hazanIntroductionOnlineConvex2016}: tight bound from~(\ref{eq:PEP})}}
\addplot [semithick, forestgreen4416044, dashed, mark=square, mark size=2, mark options={solid}]
table {%
1 0.99999997622009
2 1.86602539808008
3 2.70563299799477
4 3.58153879868095
5 4.36926755755856
6 5.10221940117672
7 5.79270161250749
8 6.4526213576502
10 7.70024853201964
15 10.5307220687907
20 13.0925557276164
25 15.4754138829171
30 17.7332251672391
35 19.895717259638
40 21.9816644090207
45 24.0038401515299
50 25.9713602247288
55 27.8912843690133
60 29.7692345570552
65 31.6097702772639
70 33.4164604156972
75 35.1922403201557
80 36.9396265985578
85 38.6611721882096
90 40.3588001390253
};
\addlegendentry{{\legendEntryFontSize Anytime Algo.~\ref{OFW_alg_new_anytime}: tight bound $B_T$ from~(\ref{eq:PEP})}}
\addplot [semithick, forestgreen4416044, dashed, mark=x, mark size=2, mark options={solid}]
table {%
1 0.99999997622009
2 1.86602536857043
3 2.79508504207169
4 3.41695016277647
5 4.00578768309647
6 4.56827781991907
7 5.10944325015655
8 5.63752492727876
10 6.66119581166219
15 9.02755847964825
20 11.1903480646082
25 13.2176738631319
30 15.1477451800894
35 17.0005288149411
40 18.7895976211105
45 20.5247895195582
50 22.2138706190613
55 23.8627054385189
60 25.4759388445294
65 27.057093360461
70 28.6091703686733
75 30.1349387581324
80 31.6366960286442
85 33.1162504419378
90 34.5752635225472
};
\addlegendentry{{\legendEntryFontSize Algo.~\ref{OFW_alg_new}: tight bound $B_T$ from~(\ref{eq:PEP})}}
\end{axis}

\end{tikzpicture}
    \end{minipage}
    \hfill
    
\medskip

    \hfill
        \begin{minipage}[t]{0.48\textwidth}
        \vspace*{0pt}
        \centering
\begin{tikzpicture}

\definecolor{darkgray176}{RGB}{176,176,176}
\definecolor{lightgray204}{RGB}{204,204,204}
\definecolor{mediumpurple148103189}{RGB}{148,103,189}
\definecolor{orchid227119194}{RGB}{227,119,194}
\definecolor{sienna1408675}{RGB}{140,86,75}

\begin{axis}[
legend cell align={left},
legend style={
  fill opacity=1,
  draw opacity=1,
  text opacity=1,
  at={(0.5,-\legendDist)},
  anchor=north,
  draw=none
},
tick align=outside,
tick pos=left,
x grid style={darkgray176},
xlabel={Time horizon $T$},
xmin=-5, xmax=105,
xtick style={color=black},
y grid style={darkgray176},
ylabel={Worst-case regret},
ymin=-2, ymax=55,
ytick style={color=black},
  width=\figWidth,
  height=\figHeight,
  scale only axis,
 font=\small
]
\addplot [semithick, mediumpurple148103189, dashed, mark=o, mark size=2, mark options={solid}]
table {%
1 0.999999999986411
2 1.73205077164806
3 2.34209616885784
4 2.90282823623805
5 3.42168738269308
6 3.91697957391197
7 4.38834045416619
8 4.84470501303823
9 5.28476977360727
10 5.71400805695854
11 6.13109156852881
12 6.53978673154998
13 6.93890009089812
14 7.33121252392302
15 7.71568171523579
20 9.55298861900087
25 11.2764639001485
30 12.91523643836
35 14.4863179235929
40 16.0021461745957
45 17.4709535774897
50 18.8995128351095
55 20.2925817455325
60 21.6543713495292
65 22.9879410641388
70 24.2961479561498
75 25.5810970007411
80 26.8448520074568
85 28.0889320029728
90 29.31489412696
95 30.5238801404065
100 31.7171067431816
};
\addlegendentry{{\legendEntryFontSize Optimized algo.~(\ref{eq:OFW_multiple}), $r=1$,  from~(\ref{eq:jointstepsizeopt})}}
\addplot [semithick, sienna1408675, dashed, mark=triangle, mark size=2, mark options={solid}]
table {%
1 0.999999999986411
2 1.63299311979155
3 2.15168457507838
4 2.62509003041537
5 3.06278519425511
6 3.4795301018729
7 3.8762878127482
8 4.25983704757672
9 4.62989237992315
10 4.99045867833446
11 5.3410098087897
12 5.68423150190584
13 6.01957981552844
14 6.34900120418311
15 6.6719871075061
20 8.21488294667517
25 9.66223380277867
30 11.0384070754211
35 12.3578609109691
40 13.6309257863288
45 14.8646107403946
50 16.0645089825127
55 17.2346879049347
60 18.378617200216
65 19.4989121739457
70 20.597904053024
75 21.6774336420247
80 22.7391637690755
};
\addlegendentry{{\legendEntryFontSize Optimized algo.~(\ref{eq:OFW_multiple}), $r=2$, from~(\ref{eq:jointstepsizeopt})}}
\addplot [semithick, orchid227119194, dashed, mark=square, mark size=2, mark options={solid}]
table {%
1 0.999999999986411
2 1.5837637868675
3 2.05505753723153
4 2.48625209927649
5 2.88336510543731
6 3.26149038547669
7 3.62091546948163
8 3.96831411446438
9 4.30323394828769
10 4.62950562985113
11 4.94658359011182
12 5.25698225939425
13 5.560188906985
14 5.85799574095499
15 6.14994478087934
20 7.54419607303208
25 8.85176336291036
30 10.0948492043156
35 11.2866164238039
40 12.4364369256464
45 13.5506663055645
50 14.6343709611955
};
\addlegendentry{{\legendEntryFontSize Optimized algo.~(\ref{eq:OFW_multiple}), $r=3$, from~(\ref{eq:jointstepsizeopt})}}
\end{axis}

\end{tikzpicture}
    \end{minipage}
    \hfill
    \begin{minipage}[t]{0.48\textwidth}
        \vspace*{0pt}
        \centering
\begin{tikzpicture}

\definecolor{darkgray176}{RGB}{176,176,176}
\definecolor{gray127}{RGB}{127,127,127}
\definecolor{lightgray204}{RGB}{204,204,204}
\definecolor{mediumpurple148103189}{RGB}{148,103,189}

\begin{axis}[
legend cell align={left},
legend style={
  fill opacity=1,
  draw opacity=1,
  text opacity=1,
  at={(0.5,-\legendDist)},
  anchor=north,
  draw=none
},
tick align=outside,
tick pos=left,
x grid style={darkgray176},
xlabel={Time horizon $T$},
xmin=-5, xmax=105,
xtick style={color=black},
y grid style={darkgray176},
ymin=-2, ymax=55,
ytick style={color=black},
  width=\figWidth,
  height=\figHeight,
  scale only axis,
 font=\small
]
\addplot [semithick, gray127, dashed, mark=star, mark size=2, mark options={solid}]
table {%
1 0.999999999988875
2 1.73205079427809
3 2.34209622334741
4 2.90670363021072
5 3.44220483725166
6 3.95167756792504
7 4.44511777330512
8 4.92357067946405
9 5.39080426794702
10 5.84791808629133
11 6.29615347280525
12 6.73716873663269
13 7.17116723126338
14 7.59935835351766
15 8.02193620499087
20 10.0681475303238
25 12.0284923064602
30 13.9253024928537
35 15.7722150044208
40 17.5783873084459
45 19.3503333022518
50 21.0929670020047
55 22.8100363786092
60 24.5043424164102
65 26.1782977655244
70 27.8339857818303
75 29.4730527040755
80 31.0969519615577
85 32.7069906330521
90 34.3032966796525
95 35.8879962752747
100 37.4610375681626
};
\addlegendentry{{\legendEntryFontSize Optimized algo.~from~(\ref{eq:jointstepsizeopt}) with $\beta_{t,s}=0$}}
\addplot [semithick, mediumpurple148103189, dashed, mark=o, mark size=2, mark options={solid}]
table {%
1 0.999999999986411
2 1.73205077164806
3 2.34209616885784
4 2.90282823623805
5 3.42168738269308
6 3.91697957391197
7 4.38834045416619
8 4.84470501303823
9 5.28476977360727
10 5.71400805695854
11 6.13109156852881
12 6.53978673154998
13 6.93890009089812
14 7.33121252392302
15 7.71568171523579
20 9.55298861900087
25 11.2764639001485
30 12.91523643836
35 14.4863179235929
40 16.0021461745957
45 17.4709535774897
50 18.8995128351095
55 20.2925817455325
60 21.6543713495292
65 22.9879410641388
70 24.2961479561498
75 25.5810970007411
80 26.8448520074568
85 28.0889320029728
90 29.31489412696
95 30.5238801404065
100 31.7171067431816
};
\addlegendentry{{\legendEntryFontSize Optimized algo.~from~(\ref{eq:jointstepsizeopt})}}
\end{axis}

\end{tikzpicture}
    \end{minipage}
    \hfill
    \end{subfigure}

    \caption{
        (Top left) Comparison of known upper bounds (respectively from~\cite[Theorem 7.3]{hazanIntroductionOnlineConvex2016} and Theorem~\ref{thm_upper_bound_OFW}) against tight numerical bounds
        (worst-case regrets) obtained from~\eqref{eq:PEP}, for~\textcolor{crimson2143940}{\cite[Algorithm 27]{hazanIntroductionOnlineConvex2016}} and \textcolor{forestgreen4416044}{Algorithm~\ref{OFW_alg_new} (parameters from Theorem~\ref{thm_upper_bound_OFW})}. 
        (Top right) Tight numerical regret bounds for \textcolor{crimson2143940}{\cite[Algorithm 27]{hazanIntroductionOnlineConvex2016}} and \textcolor{forestgreen4416044}{Algorithm~\ref{OFW_alg_new} (parameters from Theorem~\ref{thm_upper_bound_OFW})} against their anytime versions.
        (Bottom left) Tight numerical bounds for optimized online Frank--Wolfe with respectively $r\in\{1, 2,3\}$ linear optimization steps per time round
        (where \eqref{eq:jointstepsizeopt_multiple} is a variant of~\eqref{eq:jointstepsizeopt}
        with~\eqref{eq:OFW_general} replaced by~\eqref{eq:OFW_multiple}, which we detail in
        Appendix~B.5).
        (Bottom right) Tight numerical regret bounds for optimized online Frank--Wolfe with and without regularization (i.e.,~\eqref{eq:jointstepsizeopt} with and without $\beta_{t,s}=0$).
    }
    \label{fig:Comparison_tau}
\end{figure*}

\paragraph{Comparison of known upper bounds to tight numerical bounds. (Top left of Figure~\ref{fig:Comparison_tau}.)}
Algorithm~\ref{OFW_alg_new}, instantiated with parameters in Equation~\eqref{eq:parameter_choice}, enjoys a regret guarantee that improves upon the classical bound $8LD T$ from~\cite[Theorem 7.3]{hazanIntroductionOnlineConvex2016} (corresponding to Algorithm 27 therein). 
This improvement also holds when comparing the tight bounds
of both algorithms, with a gain of approximately 34\% (i.e., a factor of about $0.66$).
Although the upper bound provided in Theorem~\ref{thm_upper_bound_OFW} is not tight, it remains within a constant factor of approximately $1.5$ of the numerically computed tight worst-case bound using~\eqref{eq:PEP}. 
Note that the difference between those two bounds is due to Theorem~\ref{thm_upper_bound_OFW} being optimal among simple potential-based proofs as described in Section~\ref{sec:design_proof_theo}, while the tight numerical bound is optimal among all possible proofs (including non-potential-based ones).
Moreover, Algorithm~\ref{OFW_alg_new} (with the parameters from Theorem~\ref{thm_upper_bound_OFW}) is near-optimal among online Frank--Wolfe algorithms of the form~\eqref{eq:OFW_general}, up to a constant multiplicative factor of roughly $1.18$.
Note that the difference between those two tight numerical bounds is due to them being optimal for different classes of algorithms:
tuning~\eqref{eq:parameter_choice} is optimal among the simpler variants of OFW defined by Algorithm~\ref{OFW_alg_new}, while \eqref{eq:jointstepsizeopt} gives optimal tuning among the wider class of algorithms defined by~\eqref{eq:OFW_general}, which includes Algorithm~\ref{OFW_alg_new} as a special case.

\paragraph{Anytime variants. (Top right of Figure~\ref{fig:Comparison_tau}.)} The anytime algorithmic variants presented in Figure~\ref{fig:Comparison_tau} correspond to respectively \cite[Algorithm 27]{hazanIntroductionOnlineConvex2016} and Algorithm~\ref{OFW_alg_new} (with parameter choices from Theorem~\ref{thm_upper_bound_OFW}) where the time-horizon $T$ dependence of algorithm parameters is replaced by the current time $t$.
We observe that the anytime variant of
Algorithm~\ref{OFW_alg_new} (that is, Algorithm~\ref{OFW_alg_new_anytime})
has a guarantee close to that of the original Algorithm~\ref{OFW_alg_new}
within a constant multiplicative factor of about $1.17$.
The anytime variant of \cite[Algorithm 27]{hazanIntroductionOnlineConvex2016}
has a better guarantee than the original \cite[Algorithm 27]{hazanIntroductionOnlineConvex2016},
but still worse than the anytime variant of
Algorithm~\ref{OFW_alg_new}
(with a constant multiplicative factor of about $1.24$).
Although the upper bound provided in Theorem~\ref{thm_upper_bound_OFW_anytime} is not tight, it remains within a constant factor of approximately $1.6$ of the numerically computed tight worst-case bound.

\paragraph{Multiple linear optimization rounds per iteration. (Bottom left of Figure~\ref{fig:Comparison_tau}.)} A natural extension of~\eqref{eq:OFW_general} consists in performing a fixed number ($r>1$) of linear optimization steps per iteration by defining $r$ search directions $\text{dir}_{t,1},\ldots,\text{dir}_{t,r}$ (sequentially) per time step, together with the corresponding atoms $v_{t,1},\ldots,v_{t,k}$:
\begin{equation}\label{eq:OFW_multiple}
\begin{aligned}
& \text{dir}_{t,k} =  \sum_{s=1}^t \eta_{t,k,s}\, g_s 
    						+ \sum_{s=1}^{t-1}\sum_{j=1}^r \beta_{t,k,s,j}\, (v_{s,j}-x_1) 
\\ & \qquad\qquad\qquad\qquad\
    						+ \sum_{j=1}^{k-1} \beta_{t,k,t,j}\, (v_{t,j}-x_1) \\[.2cm]
& v_{t,k} = \argmin_{v\in \domain} \langle \text{dir}_{t,k} , v \rangle
\end{aligned}
\end{equation}
for choosing the next query point $x_{t+1} = x_1 +  \sum_{s=1}^t\sum_{k=1}^r \gamma_{t+1,s,k}\, (v_{s,k} - x_1)$.
    				We observe that multiple rounds (fixed in advance and not a function of the time horizon) of linear optimization oracles do not help improve the regret rates of online Frank--Wolfe algorithms. In all cases, the regret scales as $T^{3/4}$.
(For a clearer view of this rate, log–log plot variants of Figure~\ref{fig:Comparison_tau} are provided in Appendix~\ref{a:log_log_plots}.)
Hence, this gives a strong numerical conjecture that all variants of OFW
within our model class have a regret rate of $\Omega(T^{3/4})$.
The numerical regret bounds for optimized online Frank--Wolfe algorithms
with fixed number $r>1$ of linear optimization oracle calls per time round
were obtained in a similar way to the method outlined in Section~\ref{s:opt_param},
see Appendix~\ref{a:numerical_multiple_stepsizeopt} for details.

\paragraph{Unregularized online Frank--Wolfe. (Bottom right of Figure~\ref{fig:Comparison_tau}.)} A desirable feature of setting $\beta_{t,s}=0$ in~\eqref{eq:OFW_general} is that the remaining parameters $\eta_{t,s}$ and $\gamma_{t,s}$ are naturally dimension-independent, that is, they cannot depend on $L$ and $D$ in a meaningful way (as there is no other external quantities, there is no way to have $\eta_{t,s}$ and $\gamma_{t,s}$ being dimension-independent while depending on $L$ and $D$). Unfortunately, this nice feature is counter-balanced by an apparent worse regret rate.
More precisely, the optimized Frank--Wolfe algorithms with fixed $\beta_{t,s}=0$ appear to have their worst-case regrets scaling as $O(T^\alpha)$ for $\alpha \approx 7/8$.

\section{\uppercase{Conclusion}}\label{s:ccl}

In this work, we have studied and developed projection-free algorithms 
for online learning that rely on linear optimization oracles (a.k.a. Frank--Wolfe)
for handling the constraint set, and for convex loss functions.
More precisely, the contributions of this work are
(i)~to formulate the problem of analyzing and designing online Frank--Wolfe-type algorithms as semidefinite problems (SDPs) that can be solved numerically in a variety of settings,
(ii)~how to use those design methods to propose an improved (optimized) variant of an online Frank--Wolfe algorithm (Algorithm~\ref{OFW_alg_new}),
along with its conceptually simple potential-based proof,
and (iii)~its anytime version, which benefits from a similar $O(T^{3/4})$ regret rate without requiring knowledge of the time horizon $T$ in advance.
This SDP methodology provides a constructive and principled approach
to regret bounds and their corresponding worst-case instances.
Algorithms with optimal regret guarantees can then be designed
by jointly optimising the algorithm parameters and the regret bound.
We then leveraged those techniques to perform 
rigorous numerical experiments strongly supporting 
(a)~near-optimality claims of the proposed parameter choices,
and (b)~that all OFW-type algorithms have a regret rate of $\Omega(T^{3/4})$.
Finally, we explained how the presented SDP approach can be used to obtain conceptually simple proofs with optimal regret bounds, taking as an example
our potential-based proof from part~(ii).

The findings of our numerical experiments
motivate the following future work opportunities:
(1)~Can we find a tighter analysis of the OFW algorithm?
(2)~Can we find a closed-form for the optimal OFW-type methods?
(3)~Can we find a closed-form solution and the corresponding regret bound for the optimal unregularized OFW-type method?

\acknowledgments{
This work took place in the context of the associate team 4TUNE within \href{https://project.inria.fr/inriacwi/4tune/}{CWI-Inria international lab}. 
Julien Weibel and Adrien Taylor are supported by the European Union (ERC grant CASPER 101162889). 
Views and opinions expressed are however those of the author(s) only and do not necessarily reflect those of the European Union or the European Research Council. Neither the European Union nor the granting authority can be held responsible for them.
The French government also partly funded this work under the management of Agence Nationale de la Recherche as part of the ``France 2030'' program, reference ANR-23-IACL-0008 (PR[AI]RIE-PSAI)
and reference ANR-23-IACL-0006 (MIAI Cluster).

The author would like to thank Shuvomoy Das Gupta and anonymous referees
for helpful comments
which helped improved the presentation of this article.
}

\bibliographystyle{apalike}

\section*{Checklist}

\begin{enumerate}

  \item For all models and algorithms presented, check if you include:
  \begin{enumerate}
    \item A clear description of the mathematical setting, assumptions, algorithm, and/or model. 
    
    [\textbf{Yes}]
    Mathematical settings and assumptions are represented in the introduction section and restated in theorems. Algorithms are presented in dedicated display environments.
    \item An analysis of the properties and complexity (time, space, sample size) of any algorithm. 
    
    [\textbf{Not Applicable}]
    We only consider algorithms from an OFW class that all perform
    one linear optimization step per time round (same time complexity),
    and that all have access to exactly one gradient per time step 
    (same sample size complexity).
    The class contains algorithms with complete memory, but the algorithms
    studied in the theorems all work with constant space.
    \item (Optional) Anonymized source code, with specification of all dependencies, including external libraries. 
    
    [\textbf{No}]
    Open-source and user-friendly code relying on standard external software packages will be released upon acceptance of this work at the conference.
  \end{enumerate}

  \item For any theoretical claim, check if you include:
  \begin{enumerate}
    \item Statements of the full set of assumptions of all theoretical results.
    
    [\textbf{Yes}]
    All the theoretical results of this work can be found in Section~\ref{sec:theo_proof}.
    Those theoretical results include their full set of assumptions.
    \item Complete proofs of all theoretical results. 
    
    [\textbf{Yes}]
			Complete and correct proofs of all theoretical results can be found in Section~\ref{sec:theo_proof}, except for the more technical proofs
			which can be found in Appendix~\ref{a:proof_lemma_upper_bound}
			(with comments in Section~\ref{sec:theo_proof} pointing
		to Appendix~\ref{a:proof_lemma_upper_bound}).
    \item Clear explanations of any assumptions. 
    
    [\textbf{Yes}]     
    The assumptions used---which are standard in the literature---are presented and explained in the introduction section.
  \end{enumerate}

  \item For all figures and tables that present empirical results, check if you include:
  \begin{enumerate}
    \item The code, data, and instructions needed to reproduce the main experimental results (either in the supplemental material or as a URL). 
    
    [\textbf{Yes}]
        The techniques used for reproducing all the numerical results presented in this work are explained in Section~\ref{sec:PEP_for_OFW}---and all omitted details are provided in Appendix~\ref{a:SDPs}. Furthermore, the Python code will be released on an open repository and will be easily executed with standard software suites.
    \item All the training details (e.g., data splits, hyperparameters, how they were chosen). 
    
    [\textbf{Yes}]
    All numerical results of this work were obtained via classical convex optimization packages~\cite{aps2019mosek,JMLR:v17:15-408}, but other ones could have been used instead (deterministic convex optimization).
    \item A clear definition of the specific measure or statistics and error bars (e.g., with respect to the random seed after running experiments multiple times). 
    
    [\textbf{Not Applicable}]
    The numerical experiments in this work are deterministic
    			(they do not involve any probability distributions or dataset),
			thus the notions of statistical significance and error bars are not relevant here.    \item A description of the computing infrastructure used. (e.g., type of GPUs, internal cluster, or cloud provider). 
    
    [\textbf{Yes}]
    Numerical experiments are cheap and can be run on any modern laptop. We included a comment with the model of modern laptop
    used to run those numerical experiments, 
    an upper bound on computation time,
    and a remark about RAM size being a limiting factor for running
    larger experiments.
  \end{enumerate}

  \item If you are using existing assets (e.g., code, data, models) or curating/releasing new assets, check if you include:
  \begin{enumerate}
    \item Citations of the creator If your work uses existing assets. 
    [\textbf{Yes}]
    \item The license information of the assets, if applicable. [\textbf{Not Applicable}]
    \item New assets either in the supplemental material or as a URL, if applicable. [\textbf{Not Applicable}]
    \item Information about consent from data providers/curators. [\textbf{Not Applicable}]
    \item Discussion of sensible content if applicable, e.g., personally identifiable information or offensive content. [\textbf{Not Applicable}]
  \end{enumerate}
      The classical convex optimization packages~\cite{aps2019mosek,JMLR:v17:15-408} used to run the numerical experiments of this work are properly cited in the paper.
    No other asset is used in this work.
    Our paper does not introduce any new asset. 

  \item If you used crowdsourcing or conducted research with human subjects, check if you include:
  \begin{enumerate}
    \item The full text of instructions given to participants and screenshots. [\textbf{Not Applicable}]
    \item Descriptions of potential participant risks, with links to Institutional Review Board (IRB) approvals if applicable. [\textbf{Not Applicable}]
    \item The estimated hourly wage paid to participants and the total amount spent on participant compensation. [\textbf{Not Applicable}]
  \end{enumerate}
	We did not use crowdsourcing nor conduct research with human subjects. 

\end{enumerate}

\clearpage
\appendix
\thispagestyle{empty}

\onecolumn
\aistatstitle{Supplementary Materials}

\appendix

\tableofcontents

\section{Omitted proofs}

\subsection{Detailed proof of Lemma~\ref{lemma_upper_bound_OFW}}
\label{a:proof_lemma_upper_bound}

We first discuss the main difference between our proof scheme for upper bounding the regret of OFW
and that of Hazan~\cite[Theorem 7.3]{hazanIntroductionOnlineConvex2016}.
Hazan's proof scheme is based on proving by induction that 
$F_{t+1}(x_t) - F_{t+1}(y_{t+1}) \leq 2 D \sigma_t$ where for all $t$,
the function $F_{t} : x \mapsto \eta \sum_{s=1}^{t-1} \langle g_s, x \rangle + \frac{1}{2} \Vert x-x_1 \Vert^2$
is the intermediate objective function of the FTRL iterate $y_t$.
Using the strong convexity of the functions $F_{t}$, this bound can be transformed into a bound
on the norm distance between the OFW iterates $x_t$ and the anticipative\footnote{
There is a typo in the proof of Hazan~\cite[Theorem 7.3]{hazanIntroductionOnlineConvex2016}.
The proof is written for bounding $F_{t+1}(x_t) - F_{t+1}(y_{t})$
(indeed it seems more natural to compare $x_t$ to $y_t$ as OFW can be seen as an approximation of FTRL),
but after correction it indeed needs to be $F_{t+1}(x_t) - F_{t+1}(y_{t+1})$.
} FTRL iterates $y_{t+1}$,
which then gives a regret upper bound for OFW involving the regret upper bound of FTRL
and the norm distance errors $\Vert x_t - y_{t+1}\Vert^2$.
On the other hand, our proof scheme for OFW is potential-based and stores information on the closeness of the OFW iterates $x_{t+1}$ and the FTRL iterates $y_{t+1}$ through the norm term 
$\frac{1}{6 \eta} \Vert x_{t+1} - y_{t+1} \Vert^2$ in the potential $\phi_t$
rather than through the FTRL intermediate functions with $F_{t+1}(x_t) - F_{t+1}(y_{t+1})$ in Hazan's proof scheme.
A benefit of our more direct approach is that when upper bounding the potential increase $\phi_t - \phi_{t-1}$,
we get that either the regret or the norm term $\Vert x_t - y_{t+1} \Vert^2$
can be “large” but not both simultaneously,
whereas with Hazan's proof scheme both could be “large” simultaneously.
This allows our proof scheme to give tighter regret upper bounds for OFW
than Hazan's proof scheme.
Another improvement of our proof scheme is to work with linearized cost functions,
which allows the OFW and FTRL iterates to use the same gradients
(in Hazan's proof scheme, the mismatch between gradients for OFW and FTRL creates an extra error term).

\begin{proof}[Proof of Lemma~\ref{lemma_upper_bound_OFW}]
Let $t\in \llbracket 1, T\rrbracket$ be fixed. 
We are going to prove the upper bound on $\phi_t - \phi_{t-1}$.
Denote by $G_{t-1} = \sum_{s=1}^{t-1} g_s$ the sum of past gradients,
so that $G_t = G_{t-1} + g_t$.
Note that we have:
\begin{multline*}
\phi_t - \phi_{t-1}
=  \langle g_t, x_t - y_{t+1} \rangle + \langle G_{t-1}, y_t - y_{t+1} \rangle \\
    + \frac{1}{6\eta} ( \Vert x_{t+1} - y_{t+1}\Vert^2 - \Vert x_{t} - y_{t}\Vert^2 )
    + \frac{1}{2\eta} ( \Vert y_{t} - x_1 \Vert^2 - \Vert y_{t+1} - x_1 \Vert^2 ) .
\end{multline*}

Using the optimality in the definition of $y_t$,
we have that 
$\langle G_{t-1}, y_t - y_{t+1} \rangle 
	\leq \frac{1}{\eta} \langle y_t - x_1, y_{t+1} - y_t \rangle$,
which gives us:
\begin{equation}
	\label{eq_upper_bound_pot_OFW_1}
\phi_t - \phi_{t-1}
\leq \langle g_t, x_t - y_{t+1} \rangle
        + \frac{1}{6\eta} ( \Vert x_{t+1} - y_{t+1}\Vert^2 - \Vert x_{t} - y_{t}\Vert^2 )
        - \frac{1}{2\eta} \Vert y_t - y_{t+1} \Vert^2 .
\end{equation}

From the optimality in the definitions of $v_t$ and $y_{t+1}$, we have:
\begin{align*}
\langle G_t, v_t - y_{t+1} \rangle + \frac{1}{\eta} \langle x_t - x_1, v_t - y_{t+1} \rangle 
& \leq 0 \\
\text{and }\quad
\langle G_t, y_{t+1} - v_t \rangle + \frac{1}{\eta} \langle y_{t+1} - x_1, y_{t+1} -v_t \rangle 
& \leq 0,
\end{align*}
which imply that $\langle x_t - y_{t+1}, v_t - y_{t+1} \rangle \leq 0$.
Thus, since $x_{t+1} = (1-\sigma)x_t+\sigma v_t$, we get:
\begin{align*}
\Vert x_{t+1} - y_{t+1}\Vert^2 
& = (1-\sigma)^2 \Vert x_{t} - y_{t+1}\Vert^2       
                                + 2\sigma(1-\sigma) \langle x_t - y_{t+1}, v_t - y_{t+1} \rangle
                                + \sigma^2 \Vert v_{t} - y_{t+1}\Vert^2 
\\ & \leq (1-\sigma)^2 \Vert x_{t} - y_{t+1}\Vert^2       
                                - 2\sigma^2 \langle x_t - y_{t+1}, v_t - y_{t+1} \rangle
                                + \sigma^2 \Vert v_{t} - y_{t+1}\Vert^2 
\\ & = (1-2\sigma) \Vert x_{t} - y_{t+1}\Vert^2 
		+ \sigma^2 \Vert x_t - v_t \Vert^2 .
\end{align*}
Combining this inequality with Equation~\eqref{eq_upper_bound_pot_OFW_1}, we get:
\begin{equation}
	\label{eq_upper_bound_pot_OFW_2}
\phi_t - \phi_{t-1}
\leq \frac{\sigma^2}{6 \eta}  \Vert x_t - v_{t} \Vert^2 
        + \langle g_t, x_t - y_{t+1} \rangle - \frac{1}{2\eta} \Vert y_t - y_{t+1} \Vert^2
         + \frac{1}{6\eta} ( (1-2 \sigma) \Vert x_{t} - y_{t+1}\Vert^2 
                                - \Vert x_{t} - y_{t}\Vert^2 ) .
\end{equation}

Let $\lambda \geq 0$ whose value will be chosen later.
From the optimality in the definitions of $y_t$ and $y_{t+1}$, we have:
\begin{align*}
\langle G_{t-1}, y_t - y_{t+1} \rangle + \frac{1}{\eta} \langle y_t - x_1, y_t - y_{t+1} \rangle 
& \leq 0 \\
\text{and }\quad
\langle G_t, y_{t+1} - y_t \rangle + \frac{1}{\eta} \langle y_{t+1} - x_1, y_{t+1} -y_t \rangle 
& \leq 0,
\end{align*}
which imply that $\langle g_t, y_{t+1} - y_t \rangle 
\leq - \frac{1}{\eta} \Vert y_t - y_{t+1}\Vert^2$.
From this inequality, we obtain:
\begin{align*}
\langle g_t, x_t - y_{t+1} \rangle
& = \langle g_t, x_t - y_{t+1} + \lambda(y_t - y_{t+1}) \rangle 
        + \lambda \langle g_t, y_{t+1} - y_t \rangle 
\\ & \leq \langle g_t, x_t - y_{t+1} + \lambda(y_t - y_{t+1}) \rangle 
        - \frac{\lambda}{\eta} \Vert y_t - y_{t+1}\Vert^2
\\ & \leq \lambda^g \Vert g_t\Vert^2
        + \frac{1}{4 \lambda^g}  \Vert x_t - y_{t+1} + \lambda(y_t - y_{t+1}) \Vert^2 
        - \frac{\lambda}{\eta} \Vert y_t - y_{t+1} \Vert^2 \,,
\end{align*}
for any $\lambda^g >0$ to be chosen latter. 
Combining this inequality with Equation~\eqref{eq_upper_bound_pot_OFW_2}, we get:
\begin{multline*}
	\label{eq_upper_bound_pot_OFW_3}
\phi_t - \phi_{t-1}
\leq \frac{\sigma^2}{ 6\eta}  \Vert x_t - v_{t} \Vert^2 
	+ \lambda^g \Vert g_t\Vert^2
        + \frac{1}{4 \lambda^g} \Vert x_t - y_{t+1} + \lambda(y_t - y_{t+1}) \Vert^2 
	\\ - \frac{1+2\lambda}{2\eta} \Vert y_t - y_{t+1} \Vert^2
        + \frac{1}{6\eta} ( (1-2 \sigma) \Vert x_{t} - y_{t+1}\Vert^2 
                                - \Vert x_{t} - y_{t}\Vert^2 ) .
\end{multline*}

Hence, to complete the proof for the upper bound on $\phi_t - \phi_{t-1}$,
it suffices to show that there exists $\lambda \geq 0$ such that the following expression is non-negative, 
which can be done by rewriting it as a sum of squares:
\begin{align*}
 Q(\lambda) & \triangleq
\frac{1+2\lambda}{2\eta} \Vert y_t - y_{t+1} \Vert^2 
- \frac{1}{4 \lambda^g} \Vert x_t - y_{t+1} + \lambda(y_t - y_{t+1}) \Vert^2      
        \\ & \qquad - \frac{1}{6\eta} ( (1-2 \sigma) \Vert x_{t} - y_{t+1}\Vert^2 
                                - \Vert x_{t} - y_{t}\Vert^2 ) .
\\ & = \frac{2+3\lambda}{3\eta} \Vert y_t - y_{t+1} \Vert^2      
 + \frac{\sigma}{3\eta} \Vert x_{t} - y_{t+1}\Vert^2  
 + \frac{1}{3\eta}  \langle x_{t} - y_{t+1}, y_{t+1} - y_{t} \rangle 
\\ & \qquad - \frac{1}{4 \lambda^g}  \Vert x_t - y_{t+1} + \lambda(y_t - y_{t+1}) \Vert^2
\end{align*}
Indeed, substituting our choices of $\eta = \frac{D 3^{3/4}}{2 L T^{3/4}}$, $\sigma =  \min(1, {\frac{\sqrt{3}}{\sqrt{T}}})$ and $\lambda^g = \frac{\sigma^2 D^2}{3 \eta L^2}$, whose derivations are detailed in Section~\ref{sec:theo_proof}, and defining $\tilde T = T/3$ for conciseness, we obtain:
\begin{align*}
 \frac{D}{L} Q(\lambda) %
& = \frac{4+6\lambda}{3} \tilde T^{3/4} \Vert y_t - y_{t+1} \Vert^2 
+ \frac{2 \tilde T^{1/4}}{3} \Vert x_{t} - y_{t+1}\Vert^2 
+ \frac{2 \tilde T^{3/4}}{3}  \langle x_{t} - y_{t+1}, y_{t+1} - y_{t} \rangle 
\\ & \qquad - \frac{3 \tilde T^{1/4}}{8} \Vert x_t - y_{t+1} + \lambda(y_t - y_{t+1}) \Vert^2   
\\ & = \Bigl(\frac{4+6\lambda}{3} \tilde T^{3/4} 
	- \frac{3 \tilde T^{1/4}}{8} \lambda^2 \Bigr) \Vert y_t - y_{t+1} \Vert^2 
+ \frac{7 \tilde T^{1/4}}{24} \Vert x_{t} - y_{t+1}\Vert^2 
\\ & \qquad
+ 2\times \Bigl( \frac{ \tilde T^{3/4}}{3} + \frac{3 \tilde T^{1/4}}{8} \lambda \Bigr)  \langle x_{t} - y_{t+1}, y_{t+1} - y_{t} \rangle .
\end{align*}

Thus, using a Schur complement argument, 
$Q(\lambda)$ is a sum of squares if and only if:
$$ \Bigl( \frac{1}{3} \tilde T^{3/4} + \frac{3}{8} \tilde T^{1/4} \lambda \Bigr)^2
        \leq \frac{7}{24} \tilde T^{1/4} 
       \Bigl( \frac{4+6\lambda}{3} \tilde T^{3/4} - \frac{3}{8} \tilde T^{1/4} \lambda^2 \Bigr) ,$$
which is equivalent to:
$$ \frac{\sqrt{\tilde T}}{4} \lambda^2 - \frac{\tilde T}{3} \lambda + \Bigl( \frac{\tilde T^{3/2}}{9} - \frac{7 \tilde T}{18} \Bigr) \leq 0 .$$
Hence, there exists a non-negative $\lambda$ solving this second order equation if and only if $\Delta \geq 0$, where:
$$ \Delta\triangleq \frac{\tilde T^2}{9} - 4 \frac{\sqrt{\tilde T}}{4}
	\Bigl( \frac{\tilde T^{3/2}}{9} - \frac{7\tilde T}{18} \Bigr)
= \frac{7 \tilde T^{3/2}}{18} \geq 0 .$$
As a consequence, 
there exists a choice of $\lambda\geq 0$ such that $Q(\lambda)$ is a sum of squares,
which concludes the proof of the upper bound on $\phi_t - \phi_{t-1}$,
and thus also concludes the proof of Lemma~\ref{lemma_upper_bound_OFW}.
\end{proof}

\subsection{Potential-based proof for FTRL}\label{a:proof_pot_FTRL}

We present here the potential-based proof of the optimal upper bound
on the regret of FTRL (see Algorithm~\ref{FTRL_alg})
using the potential $\psi_t$ defined in \eqref{eq_def_pot_FTRL}.
This potential-based proof is fundamentally
a reformulation of classical FTRL regret proofs
(see , e.g.,~\cite[Chapters 6--7]{orabona2019modern}).
As the proof for the single (horizon-dependent) parameter $\eta$ version of FTRL
immediately generalizes to the round-dependent parameters $\{\eta_t\}_{t\in\llbracket 2, T\rrbracket}$ version of FTRL,
we prove the result for the latter version.

We first define the anytime version of FTRL with round-dependent parameters,
where the values of $\eta_t$ can be different between the rounds.

\begin{algorithm}[H]
\caption{Follow The Regularized Leader (FTRL)}
\label{FTRL_alg_anytime}
\begin{algorithmic}[1]
\Require $T\geq 1$,~ $y_1 \in \cK$, $\eta_t \geq 0 \text{ for } t\in\llbracket 2, T\rrbracket$
\For{$t=1$ to $T$} 
    \State Play $y_t$, pay cost $\ell_t(y_t)$, and observe $g_t = \nabla\ell_t(y_t)$.
    \State $y_{t+1} \gets \argmin_{y\in \cK} \eta_{t+1} \langle  \sum_{s=1}^t g_s , y \rangle + \frac{1}{2} \Vert y - x_1 \Vert^2 $
\EndFor
\end{algorithmic}
\end{algorithm}

We define the potential $\psi_t$ for the anytime version of FTRL
(which is the anytime version of the definition in \eqref{eq_def_pot_FTRL}):
\begin{equation}\label{eq_def_pot_FTRL_anytime}
\psi_t = \sum_{s=1}^t \langle g_s, y_s - y_{t+1} \rangle - \frac{1}{2 \eta_{t+1}} \Vert y_{t+1} - y_1 \Vert^2 .
\end{equation}

We can now state and prove the regret upper bound for the anytime version of FTRL
with non-increasing parameters $\{\eta_t\}_{t\in\llbracket 2,T\rrbracket}$.
This gives the optimal regret upper bound in the single-parameter case
(i.e., $\eta_t  =\eta$ for all $t$).

\begin{lemme}
Let $T\geq 1$.
Assume that the cost functions $\ell_t$ are convex and $L$-Lipschitz for all $t\in \llbracket 1, T \rrbracket$,
and that the convex closed domain $\cK$ of feasible points has a diameter bounded by $D$.
Assume that the FTRL parameters $\{ \eta_t \}_{t\in \llbracket 2,T \rrbracket}$ are non-increasing in $t$,
and set $\eta_{T+1} = \eta_T$.
Then, for any $y_\star \in\cK$, the following upper bound 
on the regret of the FTRL Algorithm~\ref{FTRL_alg_anytime} holds:
\begin{equation*}
R_T  \triangleq  \sum_{t=1}^T \ell_t(y_t) - \ell_t(y_\star)
\leq \frac{1}{2} \sum_{t=1}^T \eta_t \Vert g_t \Vert^2
	+ \frac{1}{2\eta_{T}} \Vert y_\star - y_1 \Vert^2.
\end{equation*}
In particular, 
\begin{enumerate}[label=(\roman*)]
\item for $\eta = \frac{D}{L\sqrt{T}}$
we get that $R_T \leq L D \sqrt{T}$,
\item\label{regret_FTRL:item2} for $\eta_t = \frac{D}{L\sqrt{t}}$
we get that $R_T \leq \frac{3}{2} L D \sqrt{T}$,
\item\label{regret_FTRL:item3} for $\eta_t = \frac{D}{L\sqrt{2 t}}$
we get that $R_T \leq \sqrt{2} L D \sqrt{T}$.
\end{enumerate}
\end{lemme}

Note that we can impose $\eta_T = \eta_{T+1}$ as $\eta_{T+1}$ is only used for computing $y_{T+1}$
which comes after the horizon $T$ and is thus not played and only used as a virtual point in the analysis.
In particular, as $\eta_{T+1}$ are $y_{T+1}$ is a virtual value only used in the analysis, this does not
change the anytime property of the tunings \ref{regret_FTRL:item2} and \ref{regret_FTRL:item3}. 
This “trick” allows us to obtain a slightly tighter and nicer-looking regret upper bound:
indeed, for $\eta_t = \Theta(\frac{1}{\sqrt{t}})$, this allows us to save an extra factor of $O(\frac{1}{\sqrt{T}})$
and does not change the leading order term in $\Theta(\sqrt{T})$.

\begin{proof}
We present a potential-based proof based on the potential $\psi_t$
defined in \eqref{eq_def_pot_FTRL_anytime}.

We start by upper bounding $\psi_t - \psi_{t-1}$
for fixed $t\in \llbracket 1, T \rrbracket$.
Using the optimality in the definition of $y_t$ gives:
\begin{equation*}
\langle G_{t-1}, y_t - y_{t+1} \rangle 
+ \frac{1}{\eta_t} \langle y_t - y_1, y_t - y_{t+1} \rangle \leq 0 .
\end{equation*}
Then, using this equation 
and the inequality 
$\langle u, v \rangle \leq \frac{\eta_t}{2} \Vert u \Vert^2
                            + \frac{1}{2\eta_t} \Vert v \Vert^2$, 
we get the upper bound:
\begin{align*}
\psi_{t} - \psi_{t-1}
& = \langle G_t, y_t - y_{t+1} \rangle 
+ \frac{1}{2\eta_t}  \Vert y_t - y_1 \Vert^2 
                            - \frac{1}{2\eta_{t+1}} \Vert y_{t+1} - y_1 \Vert^2 
\\ & \leq \langle g_t, y_t - y_{t+1} \rangle 
+ \frac{1}{2\eta_t} \left( \Vert y_t - y_1 \Vert^2 
                        + 2 \langle y_t - y_1, y_{t+1} - y_t \rangle
                            - \Vert y_{t+1} - y_1 \Vert^2 \right)
\\ & \qquad \qquad \qquad\qquad
		- \left( \frac{1}{2\eta_{t+1}} - \frac{1}{2\eta_t} \right) \Vert y_{t+1} - y_1 \Vert^2
\\ & = \langle g_t, y_t - y_{t+1} \rangle 
- \frac{1}{2\eta_t} \Vert y_t - y_{t+1} \Vert^2 
- \left( \frac{1}{2\eta_{t+1}} - \frac{1}{2\eta_t} \right) \Vert y_{t+1} - y_1 \Vert^2
\\ & \leq \frac{\eta_t}{2} \Vert g_t \Vert^2 .
\end{align*}

Now, combining those potential differences, the sum telescopes
(remark that $\psi_0 = 0$) and we get:
\begin{equation*}
\psi_T 
= \sum_{t=1}^T \psi_t - \psi_{t-1}
\leq \frac{1}{2} \sum_{t=1}^T \eta_t \Vert g_t \Vert^2 .
\end{equation*}
Using the optimality in the definition of $y_{t+1}$ gives:
\begin{equation*}
R_T - \frac{1}{2\eta_{T+1}} \Vert y_\star - y_1 \Vert^2
\leq \psi_T ,
\end{equation*}
which concludes the proof of the main upper bound of $R_T$ (recall that $\eta_T = \eta_{T+1}$).
The second statement follows directly by using the Lipschitz and
diameter bounds.
\end{proof}

\subsection{Detailed proof of Theorem~\ref{thm_upper_bound_OFW_anytime}}
\label{sec:proof_OFW_anytime}

The proof of Theorem~\ref{thm_upper_bound_OFW_anytime}
is an adaptation of that of Theorem~\ref{thm_upper_bound_OFW}.
Thus, to prove Theorem~\ref{thm_upper_bound_OFW_anytime}, we start
by proving the following lemma upper bounding the potential increase $\phi_t - \phi_{t-1}$, which is an adaptation of Lemma~\ref{lemma_upper_bound_OFW}.
\begin{lemma}\label{lemma_upper_bound_OFW_anytime}
Denote by $\ell_t : x \mapsto \langle g_t, x \rangle$ 
the linear cost function at time $t\geq 1$. Assume that $\Vert g_t \Vert \leq L$ for all $t\geq 1$
and that the convex closed domain $\cK$ of feasible points has a diameter bounded by $D$.
Define the sequence of potentials $(\phi_t)_{t\geq 0}$ as:
\begin{equation*}
\phi_t \triangleq \sum_{s=1}^t \langle g_s, x_s - y_{t+1} \rangle 
	 - \frac{1}{2\eta_{t+1}} \Vert y_{t+1} - x_1 \Vert^2
	+ \frac{1}{6\eta_{t+1}} \Vert x_{t+1} - y_{t+1}\Vert^2 .
\end{equation*}
Then, Algorithm~\ref{OFW_alg_new_anytime} and~\ref{FTRL_alg_anytime}, run with parameters in~\eqref{eq:parameter_choice_anytime} and initialized in $y_1 = x_1 \in \cK$, satisfy,
\begin{equation*}
   \phi_t - \phi_{t-1} 
\leq \frac{2 D}{L 3^{3/4} t^{1/4}} \Vert g_t\Vert^2 
	+ \frac{L}{ D 3^{3/4} (t+1)^{1/4}} \Vert x_t - v_t \Vert^2 \,, \qquad \forall  t\geq2\,.
\end{equation*}
\end{lemma}

The proof of Lemma~\ref{lemma_upper_bound_OFW_anytime}
is an adaptation of that of Lemma~\ref{lemma_upper_bound_OFW}
where the difference between $\eta_{t+1}$ and $\eta_t$ creates extra error terms.

\begin{proof}[Proof of Lemma~\ref{lemma_upper_bound_OFW_anytime}]
Let $t\geq 2$ be fixed. 
We are going to prove the upper bound on $\phi_t - \phi_{t-1}$.
Denote by $G_{t-1} = \sum_{s=1}^{t-1} g_s$ the sum of past gradients,
so that $G_t = G_{t-1} + g_t$.
Note that we have:
\begin{multline*}
\phi_t - \phi_{t-1}
=  \langle g_t, x_t - y_{t+1} \rangle + \langle G_{t-1}, y_t - y_{t+1} \rangle \\
    + \frac{1}{6\eta_{t+1}} \Vert x_{t+1} - y_{t+1}\Vert^2 - \frac{1}{6\eta_{t}} \Vert x_{t} - y_{t}\Vert^2 
    + \frac{1}{2\eta_{t}} \Vert y_{t} - x_1 \Vert^2 - \frac{1}{2\eta_{t+1}} \Vert y_{t+1} - x_1 \Vert^2 .
\end{multline*}

Using the optimality in the definition of $y_t$,
we have that 
$\langle G_{t-1}, y_t - y_{t+1} \rangle 
	\leq \frac{1}{\eta_{t}} \langle y_t - x_1, y_{t+1} - y_t \rangle$,
which gives us:
\begin{multline}
	\label{eq_upper_bound_pot_OFW_1_anytime}
\phi_t - \phi_{t-1}
\leq \langle g_t, x_t - y_{t+1} \rangle
        - \frac{1}{2\eta_{t}} \Vert y_t - y_{t+1} \Vert^2 
        - \frac{1}{2} \left( \frac{1}{\eta_{t+1}} - \frac{1}{\eta_{t}} \right) \Vert y_{t+1} - x_1 \Vert^2 \\
       + \frac{1}{6\eta_{t+1}} \Vert x_{t+1} - y_{t+1}\Vert^2 - \frac{1}{6\eta_{t}} \Vert x_{t} - y_{t}\Vert^2 .
\end{multline}

From the optimality in the definitions of $v_t$ and $y_{t+1}$, we have:
\begin{align*}
\langle G_t, v_t - y_{t+1} \rangle + \frac{1}{\eta_{t+1}} \langle x_t - x_1, v_t - y_{t+1} \rangle 
& \leq 0 \\
\text{and }\quad
\langle G_t, y_{t+1} - v_t \rangle + \frac{1}{\eta_{t+1}} \langle y_{t+1} - x_1, y_{t+1} -v_t \rangle 
& \leq 0,
\end{align*}
which imply that $\langle x_t - y_{t+1}, v_t - y_{t+1} \rangle \leq 0$.
Thus, since $x_{t+1} = (1-\sigma_{t+1})x_t + \sigma_{t+1} v_t$, we get:
\begin{align*}
\Vert x_{t+1} - y_{t+1}\Vert^2 
& = (1-\sigma_{t+1})^2 \Vert x_{t} - y_{t+1}\Vert^2       
                                + 2\sigma_{t+1}(1-\sigma_{t+1}) \langle x_t - y_{t+1}, v_t - y_{t+1} \rangle
                                + \sigma_{t+1}^2 \Vert v_{t} - y_{t+1}\Vert^2 
\\ & \leq (1-\sigma_{t+1})^2 \Vert x_{t} - y_{t+1}\Vert^2       
                                - 2\sigma_{t+1}^2 \langle x_t - y_{t+1}, v_t - y_{t+1} \rangle
                                + \sigma_{t+1}^2 \Vert v_{t} - y_{t+1}\Vert^2 
\\ & = (1-2\sigma_{t+1}) \Vert x_{t} - y_{t+1}\Vert^2 
		+ \sigma_{t+1}^2 \Vert x_t - v_t \Vert^2 .
\end{align*}
Combining this inequality with Equation~\eqref{eq_upper_bound_pot_OFW_1_anytime}, we get:
\begin{multline}
	\label{eq_upper_bound_pot_OFW_2_anytime}
\phi_t - \phi_{t-1}
\leq \frac{\sigma_{t+1}^2}{6 \eta_{t+1}}  \Vert x_t - v_{t} \Vert^2 
        + \langle g_t, x_t - y_{t+1} \rangle - \frac{1}{2\eta_{t}} \Vert y_t - y_{t+1} \Vert^2
        - \frac{1}{2} \left( \frac{1}{\eta_{t+1}} - \frac{1}{\eta_{t}} \right) \Vert y_{t+1} - x_1 \Vert^2 \\
       	+ \frac{1}{6\eta_{t+1}} (1-2 \sigma_{t+1}) \Vert x_{t} - y_{t+1}\Vert^2 
                                - \frac{1}{6\eta_{t}} \Vert x_{t} - y_{t}\Vert^2 .
\end{multline}

Let $\lambda \geq 0$ whose value will be chosen later.
From the optimality in the definitions of $y_t$ and $y_{t+1}$, we have:
\begin{align*}
\langle G_{t-1}, y_t - y_{t+1} \rangle + \frac{1}{\eta_{t}} \langle y_t - x_1, y_t - y_{t+1} \rangle 
& \leq 0 \\
\text{and }\quad
\langle G_t, y_{t+1} - y_t \rangle + \frac{1}{\eta_{t+1}} \langle y_{t+1} - x_1, y_{t+1} -y_t \rangle 
& \leq 0,
\end{align*}
which imply that $\langle g_t, y_{t+1} - y_t \rangle 
\leq - \frac{1}{\eta_{t}} \Vert y_t - y_{t+1}\Vert^2 
	- ( \frac{1}{\eta_{t+1}} - \frac{1}{\eta_{t}}) \langle y_{t+1} - x_1, y_{t+1} -y_t \rangle $.
From this inequality, we obtain:
\begin{align*}
\langle g_t, x_t - y_{t+1} \rangle
& = \langle g_t, x_t - y_{t+1} + \lambda(y_t - y_{t+1}) \rangle 
        + \lambda \langle g_t, y_{t+1} - y_t \rangle 
\\ & \leq \langle g_t, x_t - y_{t+1} + \lambda(y_t - y_{t+1}) \rangle 
        - \frac{\lambda}{\eta_{t}} \Vert y_t - y_{t+1}\Vert^2
	\\ & \qquad 
	- \lambda \left( \frac{1}{\eta_{t+1}} - \frac{1}{\eta_{t}} \right) \langle y_{t+1} - x_1, y_{t+1} -y_t \rangle 
\\ & \leq \lambda^g \Vert g_t\Vert^2
        + \frac{1}{4 \lambda^g}  \Vert x_t - y_{t+1} + \lambda(y_t - y_{t+1}) \Vert^2 
        - \frac{\lambda}{\eta_{t}} \Vert y_t - y_{t+1} \Vert^2 \,
        \\ & \qquad 
        - \lambda \left( \frac{1}{\eta_{t+1}} - \frac{1}{\eta_{t}} \right) \langle y_{t+1} - x_1, y_{t+1} -y_t \rangle ,
\end{align*}
for any $\lambda^g >0$ to be chosen latter. 
Combining this inequality with Equation~\eqref{eq_upper_bound_pot_OFW_2_anytime}, we get:
\begin{align*}
	\label{eq_upper_bound_pot_OFW_3}
\phi_t - \phi_{t-1}
\leq \ & \frac{\sigma_{t+1}^2}{ 6\eta_{t+1}}  \Vert x_t - v_{t} \Vert^2 
	+ \lambda^g \Vert g_t\Vert^2
        + \frac{1}{4 \lambda^g} \Vert x_t - y_{t+1} + \lambda(y_t - y_{t+1}) \Vert^2 
	\\ &
	- \frac{1}{2} \left( \frac{1}{\eta_{t+1}} - \frac{1}{\eta_{t}} \right) \Vert y_{t+1} - x_1 \Vert^2 
	- \lambda \left( \frac{1}{\eta_{t+1}} - \frac{1}{\eta_{t}} \right) \langle y_{t+1} - x_1, y_{t+1} -y_t \rangle 
	\\ &
	- \frac{1+2\lambda}{2\eta_{t}} \Vert y_t - y_{t+1} \Vert^2
       	+ \frac{1}{6\eta_{t+1}} (1-2 \sigma_{t+1}) \Vert x_{t} - y_{t+1}\Vert^2 
                                - \frac{1}{6\eta_{t}} \Vert x_{t} - y_{t}\Vert^2 .
\end{align*}
Recombining terms around the scalar product term, this upper bound can be rewritten as:
\begin{align*}
	\phi_t - \phi_{t-1}
\leq \ & \frac{\sigma_{t+1}^2}{ 6\eta_{t+1}}  \Vert x_t - v_{t} \Vert^2 
	+ \lambda^g \Vert g_t\Vert^2
	- \frac{1}{2} \left( \frac{1}{\eta_{t+1}} - \frac{1}{\eta_{t}} \right) 
			\Vert y_{t+1} - x_1 + \lambda (y_{t+1} - y_{t}) \Vert^2 
	\\ &
	+ \frac{1}{4 \lambda^g} \Vert x_t - y_{t+1} + \lambda(y_t - y_{t+1}) \Vert^2 
	+ \frac{\lambda^2}{2} \left( \frac{1}{\eta_{t+1}} - \frac{1}{\eta_{t}} \right) \Vert y_t - y_{t+1} \Vert^2
	\\ &
	- \frac{1+2\lambda}{2\eta_{t}} \Vert y_t - y_{t+1} \Vert^2
       	+ \frac{1}{6\eta_{t+1}} (1-2 \sigma_{t+1}) \Vert x_{t} - y_{t+1}\Vert^2 
                                - \frac{1}{6\eta_{t}} \Vert x_{t} - y_{t}\Vert^2 .
\end{align*}

Using the definition of $\eta_t = \frac{D 3^{3/4}}{2 L t^{3/4}}$
and that $(1+x)^{3/4} \leq 1 + \frac{3}{4} x$ for $x\geq 0$, we get that
\begin{equation*}
0 \leq
\left( \frac{1}{\eta_{t+1}} - \frac{1}{\eta_{t}} \right) 
\leq \frac{3}{4 t \eta_t} .
\end{equation*}
Combining those two inequalities, we get:
\begin{align*}
	\phi_t - \phi_{t-1}
\leq \ & \frac{\sigma_{t+1}^2}{ 6\eta_{t+1}}  \Vert x_t - v_{t} \Vert^2 
	+ \lambda^g \Vert g_t\Vert^2
	\\ &
	+ \frac{1}{4 \lambda^g} \Vert x_t - y_{t+1} + \lambda(y_t - y_{t+1}) \Vert^2 
	- \left( \frac{1+2\lambda}{2\eta_t} - \frac{3\lambda^2}{8t \eta_{t}} \right) \Vert y_t - y_{t+1} \Vert^2
	\\ &
       	+ \frac{1}{6\eta_{t+1}} (1-2 \sigma_{t+1}) \Vert x_{t} - y_{t+1}\Vert^2 
                                - \frac{1}{6\eta_{t}} \Vert x_{t} - y_{t}\Vert^2 .
\end{align*}

Hence, to complete the proof for the upper bound on $\phi_t - \phi_{t-1}$,
it suffices to show that there exists $\lambda >0$ such that the following expression is non-negative:
\begin{align*}
 Q(\lambda) & \triangleq
\left( \frac{1+2\lambda}{2\eta_t} - \frac{3\lambda^2}{8t \eta_{t}} \right) \Vert y_t - y_{t+1} \Vert^2 
- \frac{1}{4 \lambda^g} \Vert x_t - y_{t+1} + \lambda(y_t - y_{t+1}) \Vert^2      
        \\ & \qquad -   \frac{1}{6\eta_{t+1}} (1-2 \sigma_{t+1}) \Vert x_{t} - y_{t+1}\Vert^2 
                                + \frac{1}{6\eta_{t}} \Vert x_{t} - y_{t}\Vert^2 
\\ & = \left( \frac{2+3\lambda}{3\eta_t} - \frac{3\lambda^2}{8t \eta_{t}} \right) \Vert y_t - y_{t+1} \Vert^2       
 + \frac{1}{3\eta_t}  \langle x_{t} - y_{t+1}, y_{t+1} - y_{t} \rangle 
 \\ & \qquad  + \left(\frac{\sigma_{t+1}}{3\eta_{t+1}} - \frac{1}{6} \left( \frac{1}{\eta_{t+1}} - \frac{1}{\eta_{t}} \right) \right) \Vert x_{t} - y_{t+1}\Vert^2 
- \frac{1}{4 \lambda^g}  \Vert x_t - y_{t+1} + \lambda(y_t - y_{t+1}) \Vert^2
\\ & \geq \tilde Q(\lambda) ,
\end{align*}
where
\begin{align*}
 \tilde Q(\lambda) & \triangleq
 \left( \frac{2+3\lambda}{3\eta_t} - \frac{3\lambda^2}{8t \eta_{t}} \right) \Vert y_t - y_{t+1} \Vert^2       
 + \frac{1}{3\eta_t}  \langle x_{t} - y_{t+1}, y_{t+1} - y_{t} \rangle 
 \\ & \qquad  + \left(\frac{\sigma_{t+1}}{3\eta_{t+1}} - \frac{1}{8 t \eta_{t}} \right) \Vert x_{t} - y_{t+1}\Vert^2 
- \frac{1}{4 \lambda^g}  \Vert x_t - y_{t+1} + \lambda(y_t - y_{t+1}) \Vert^2 
\\ & =
\left(  \frac{2+3\lambda}{3\eta_t} - \frac{3\lambda^2 t^{-1}}{8\eta_{t}} - \frac{\lambda^2}{4\lambda^g} \right) 
		\Vert y_t - y_{t+1} \Vert^2       
 + 2 \times \left( \frac{1}{6\eta_t} + \frac{\lambda}{4 \lambda^g} \right) \langle x_{t} - y_{t+1}, y_{t+1} - y_{t} \rangle 
 \\ & \qquad  + \left(\frac{\sigma_{t+1}}{3\eta_{t+1}} - \frac{1}{8 t \eta_{t}} - \frac{1}{4 \lambda^g} \right) 
 				\Vert x_{t} - y_{t+1}\Vert^2 .
\end{align*}
Hence, to complete the proof for the upper bound on $\phi_t - \phi_{t-1}$,
it suffices to show that there exists $\lambda \geq 0$ such that $\tilde Q(\lambda)$ is non-negative, which can be done by rewriting it as a sum of squares.
Using a Schur complement argument, we get that $\tilde Q(\lambda)$ is a sum of square if and only if:
\begin{equation}
\frac{\sigma_{t+1}}{3\eta_{t+1}} - \frac{1}{4 \lambda^g} - \frac{1}{8 t \eta_{t}} 
\geq 0
\end{equation}
and
\begin{equation}
\left( \frac{1}{6\eta_t} + \frac{\lambda}{4 \lambda^g} \right)^2
\leq \left(\frac{\sigma_{t+1}}{3\eta_{t+1}}  - \frac{1}{4 \lambda^g} - \frac{1}{8 t \eta_{t}} \right) 
\left(  \frac{2+3\lambda}{3\eta_t} - \frac{\lambda^2}{4\lambda^g} -  \frac{3\lambda^2}{8 t \eta_{t}} \right) .
\end{equation}
We now substitute our choices of 
$\eta_t = \frac{D 3^{3/4}}{2 L t^{3/4}}$, $\sigma_t =  \min(1, {\frac{\sqrt{3}}{\sqrt{t}}})$ and $\lambda^g = \frac{\sigma_t^2 D^2}{3 \eta_t L^2}$, whose derivations for the fixed parameter setting (i.e., non-anytime setting) are detailed in Section~\ref{sec:theo_proof}
(note that those parameter choices are still leading-term optimal in the anytime setting).
In particular, note that, as $t\geq 2$, we have that
$\frac{\sigma_{t+1}}{\eta_{t+1}} \geq \frac{2L}{D}\left(\frac{t}{3}\right)^{1/4}$.
Thus, defining $\tilde t = t/3$ for conciseness, we obtain for $t\geq 1$ that 
$\tilde Q(\lambda) \geq 0$ is a sum of square if:
\begin{equation}
\frac{7 \tilde t^{1/2}}{24} 
\geq \frac{1}{12 \tilde t^{1/4}} 
\end{equation}
and
\begin{equation}
\left( \frac{\tilde t^{3/4}}{3} + \frac{3 \tilde t^{1/4} \lambda}{8} \right)^2
\leq \left(\frac{7 \tilde t^{1/4}}{24} - \frac{1}{12 \tilde t^{1/4}}  \right) 
\left(  \frac{4 \tilde t^{3/4}}{3} + 2 \tilde t^{3/4}\lambda - \frac{3 \tilde t^{1/4} \lambda^2}{8} - \frac{\lambda^2}{4 \tilde t^{1/4}} \right) .
\end{equation}
The first line is always satisfied for $t\geq 1$.
The second line is equivalent to:
\begin{equation}\label{eq_pol_eq_anytime_OFW_lambda}
\left( \frac{\sqrt{\tilde t}}{4} + \frac{1}{24} - \frac{1}{48 \sqrt{\tilde t}} \right) \lambda^2
+ \left( - \frac{\tilde t}{3} + \frac{\sqrt{\tilde t}}{6} \right) \lambda
+ \left( \frac{\tilde t^{3/2}}{9} - \frac{7 \tilde t}{18} + \frac{\sqrt{\tilde t}}{9} \right)
\leq 0.
\end{equation}
Note that 
$- \frac{\tilde t}{3} + \frac{\sqrt{\tilde t}}{6} \leq 0$ is equivalent to $t \geq 3/4$,
which is always true.
Also note that $\frac{\sqrt{\tilde t}}{4} - \frac{10}{96} + \frac{1}{48 \sqrt{\tilde t}} \geq 0$
is true for any $t\geq 1$.
Hence, there exists a non-negative $\lambda$ that solves \eqref{eq_pol_eq_anytime_OFW_lambda}
if and only if $\Delta(\tilde t) \geq 0$ where:
\begin{equation}
\Delta(\tilde t) \triangleq 
\left( - \frac{\tilde t}{3} + \frac{\sqrt{\tilde t}}{6} \right)^2
- 4 \left( \frac{\sqrt{\tilde t}}{4} + \frac{1}{24} - \frac{1}{48 \sqrt{\tilde t}} \right)
\left( \frac{\tilde t^{3/2}}{9} - \frac{7 \tilde t}{18} + \frac{\sqrt{\tilde t}}{9} \right) .
\end{equation}
Indeed, we have:
\begin{equation}
\Delta(\tilde t) = \frac{7 \tilde t^{3/2}}{27} - \frac{\tilde t}{108} - \frac{11\sqrt{\tilde t}}{216} + \frac{1}{108}
\geq \Delta(1/3) > 0.
\end{equation}
As a consequence, 
there exists a choice of $\lambda\geq 0$ such that $\tilde Q(\lambda)$ is a sum of squares,
which concludes the proof of the upper bound on $\phi_t - \phi_{t-1}$ for $t\geq 2$,
and thus also concludes the proof of Lemma~\ref{lemma_upper_bound_OFW_anytime}.
\end{proof}

\begin{proof}[Proof of Theorem~\ref{thm_upper_bound_OFW_anytime}]
Let $t\geq 1$ be fixed. We prove the regret bound on $R_t$ in the theorem for this value of $t$.
Through this proof, we will denote by $s$ some time index ranging from $1$ to $t$.
First, using convexity of the cost functions $\ell_s$
(recall $g_s = \nabla\ell_s(x_s)$), we get:
\begin{equation*}
R_t \leq \sum_{s=1}^t \langle g_s, x_s - x_\star \rangle. 
\end{equation*}
We upper bound $\phi_s - \phi_{s-1}$ by applying 
Lemma~\ref{lemma_upper_bound_OFW_anytime} for $s\in\llbracket 2, t\rrbracket$
to the linearized cost functions $\tilde\ell_s : x \mapsto \langle g_s, x\rangle$,
and then summing those inequalities, we get
\begin{equation*}
\phi_t - \phi_1 \leq \frac{2 D}{L 3^{3/4}} \sum_{s=2}^t \frac{1}{s^{1/4}}\Vert g_s \Vert^2
	+ \frac{L}{D 3^{3/4}} \sum_{s=2}^t \frac{1}{(s+1)^{3/4}} \Vert x_s - v_s \Vert^2 .
\end{equation*}
Throughout the proof, we will repeatedly use the fact that the optimum of a constrained convex optimization problem 
$
x_* \in \arg\min_{x \in \mathcal{K}} f(x)
$
satisfies $\langle \nabla f(x_*), x_* - x \rangle \leq 0$ for all $x \in \mathcal{K}$. Using the optimality from the definition of $y_{t+1}$, we get:
\begin{equation*}
\sum_{s=1}^t \langle g_s, x_s - x_\star \rangle 
	- \frac{1}{2\eta_{t+1}} \Vert x_\star - x_1 \Vert^2
\leq \sum_{s=1}^t \langle g_s, x_s - y_{t+1} \rangle 
	- \frac{1}{2\eta_{t+1}} \Vert y_{t+1} - x_1 \Vert^2
\leq \phi_t .
\end{equation*}
Combining those four inequalities (recall $\eta_{t+1} = \frac{ 3^{3/4}D}{2 L (t+1)^{3/4}}$), we get:
\begin{equation*}
R_t
\leq  \frac{2 D}{L 3^{3/4}} \sum_{s=2}^t \frac{1}{s^{1/4}} \Vert g_s \Vert^2
	+ \frac{L}{D 3^{3/4}} \sum_{s=2}^t \frac{1}{(s+1)^{1/4}} \Vert x_s - v_s \Vert^2
	+ \frac{L (t+1)^{3/4}}{D 3^{3/4}} \Vert x_\star - x_1 \Vert^2
	+ \phi_1.
\end{equation*}
Using Cauchy-Schwarz inequality, and the bounds on the gradients and on the diameter,
we get $\phi_1 \leq (1+\frac{1}{6 \eta_2})LD \leq 1.12 LD$.
Moreover, as $\Vert g_s \Vert \leq L$ 
and $\Vert x_s - v_s  \Vert \leq D$ for all $s\in \llbracket 1,t \rrbracket$
and $\Vert x_\star - x_1 \Vert \leq D$,
we get:
\begin{equation}\label{eq:bound_anytime_OFW_with_sums}
R_t 
\leq  \frac{2 L D}{3^{3/4}} \sum_{s=2}^t \frac{1}{s^{1/4}} 
	+ \frac{L D}{ 3^{3/4}} \sum_{s=2}^t \frac{1}{(s+1)^{1/4}} 
	+ \frac{L D (t+1)^{3/4}}{3^{3/4}} 
	+ 1.12LD .
\end{equation} 
Using series-integral comparison, we get for any integers $1\leq a \leq b$:
\begin{equation*}
\sum_{t=a}^b \frac{1}{t^{1/4}} \leq \int_{a-1}^b u^{-1/4}\, \mathrm{d}u = \frac{4}{3} ( b^{3/4} - (a-1)^{3/4} ) .
\end{equation*}
Combining those inequalities, and using that $(1+x)^{3/4} \leq 1 + \frac{3}{4} x$ for $x\geq0$, we get:
\begin{align*}
R_t
& \leq \frac{2 L D}{3^{3/4}} \frac{4}{3}(t^{3/4} - 1)
	+ \frac{L D}{ 3^{3/4}} \frac{4}{3}((t+1)^{3/4} - 2^{3/4})
	+ \frac{L D (t+1)^{3/4}}{3^{3/4}} 
	+ 1.12LD \\
& \leq \frac{5}{3^{3/4}} LD t^{3/4}
	+ \left( 1.12 - \frac{8}{3} \frac{1}{3^{3/4}} - \frac{4}{3} \left(\frac{2}{3}\right)^{3/4} \right) LD
	+ \frac{L D}{ 3^{3/4}} \frac{1}{t^{1/4}} \\
& \leq \frac{5}{3^{3/4}} LD t^{3/4}
	-1.03 LD
	+ \frac{L D}{ 3^{3/4}} \frac{1}{t^{1/4}} \\
& \leq \frac{5}{3^{3/4}} LD t^{3/4} ,
\end{align*} 
where in the last inequality we use that $t\geq 1$.
This being true for all values of $t\geq 1$,
this concludes the proof of Theorem~\ref{thm_upper_bound_OFW_anytime}.
\end{proof}

\section{Detailed semidefinite formulations}\label{a:SDPs}
\subsection{Tractable formulation of~\texorpdfstring{\eqref{eq:PEP}}{(4)}}\label{a:pep}

As explained in Section~\ref{s:opt_param}, \eqref{eq:PEP} is a priori
an infinite-dimensional problem, as it includes functional variables
which are the losses $\ell_t$ and the indicator function of the feasible set $\cK$.
However, \eqref{eq:PEP} can be reformulated as finite dimensional 
linear semidefinite program.
In this section, we detail how to obtain this reformulation.

The first step consists in reformulating~\eqref{eq:PEP} as a finite-dimensional problem 
by sampling the losses $\ell_t$ at the query points $x_t$ and $x_\star$, 
and to treat only the responses $(\nabla \ell_t(x_t), \ell_t(x_t))$ and $(\nabla \ell_t(x_\star), \ell_t(x_\star))$ as variables. 
By appropriately constraining these responses, we can force them to be compatible with some losses $\ell_t$ satisfying the desired assumptions (convexity and Lipschitzness of $\ell_t$). 
Using an interpolation / extension theorem \cite[Theorem 3.3 and Equation (7)]{taylor2017exact},
it suffices to require that the samples are compatible through the subgradient inequalities
$\ell_t(x_t) - \ell_t(x_\star) \leq \langle \nabla \ell_t(x_t), x_t - x_\star \rangle$  
and $\ell_t(x_\star) - \ell_t(x_t) \leq \langle \nabla \ell_t(x_\star), x_\star - x_t \rangle$
and the Lipschitz bounds $\|\nabla \ell_t(x_t)\|\leq L$ for $t=1,\cdots, T$.
Indeed, as the subgradients $\nabla \ell_t(x_\star)$ are never used in~\eqref{eq:PEP}
their values do not matter, and thus we can impose 
$\nabla \ell_t(x_\star) = \nabla \ell_t(x_t)$ without changing the value of the problem
(that is, we get a tight reformulation of~\eqref{eq:PEP} and not just an upper bound).
This allows us to assume without loss of generality that the cost functions $\ell_t$
are all linear.
This corresponds to upper bound $\ell_t(x_t) - \ell_t(x_\star)$
by $\langle g_t, x_t - x_\star \rangle$ (recall $g_t=\nabla \ell_t(x_t)$) 
in the objective of~\eqref{eq:PEP}, and replace the cost functions variables $\ell_t$
by the gradient variables $g_t$.
Note that we get a tight reformulation of~\eqref{eq:PEP} and not just an upper bound
as any point (with linear cost functions) in this new reformulation is still feasible 
with the same objective value in the initial problem~\eqref{eq:PEP}.

As for handling the convex domain $\domain$, one possible approach is to sample its indicator function at the query points $v_t$ and $x_t$ for $t=1,\dots,T$ and $x_\star$ 
and impose similar compatibility constraints.
Indeed, denoting by $\iota_{\cK}$ the indicator function of the feasible set $\cK$,
and using \cite[Theorem 3.6]{taylor2017exact}, those compatibility constraints are:
for all $u, v\in \{ x_1,\dots,x_T, v_1, \dots, v_{T-1}, x_\star \}$, 
and $g \in \nabla \iota_{\cK} (u)$ a subgradient of $\iota_{\cK}$ at $u$,
we have $\langle g, v - u \rangle \leq 0$.
As the points $\{ (x_t, v_t, \mathrm{dir}_t) \}_{t=1,\dots,T}$ are generated by
\eqref{eq:OFW_general},
we get that $-\mathrm{dir}_t$ is a subgradient of $\iota_{\cK}$ at $v_t$ for $t=1,\dots,T-1$.
However, we have no special choice for the subgradients of $\iota_{\cK}$
at $x_1, \dots, x_T, x_\star$, and we will choose their subgradients to be $0$
(which leads to trivial constraints that we remove from the problem).
Those two function sampling steps give us the following finite-dimensional reformulation of~\eqref{eq:PEP}:
\begin{equation*}
\begin{aligned}
B_T(\{(\eta_{t,s},\beta_{t,s},\gamma_{t,s})\}_{t,s}) =
\sup_{\substack{\{g_t\}_{t=1,\ldots,T},\, x_\star\\ \{(x_t, v_t, \mathrm{dir}_t)\}_{t=1,\ldots,T}\\d\in\mathbb{N}}} \,
& \sum_{t=1}^T \langle g_t, x_t - x_\star \rangle \\
\text{subject to: } 
&  \Vert g_t \Vert \leq L \text{ for }t=1,\ldots,T,\\
& \langle - \mathrm{dit}_t, u - v_t \rangle \leq 0
    \text{ for $t=1,\dots,T-1$ } \\
& \qquad\qquad
    \text{and } u\in\{x_1,\dots,x_{T},v_1,\dots,v_{T-1},x_\star\}
    , \\
&\mathrm{Diam}(\{x_1,\dots,x_{T},v_1,\dots,v_{T-1},x_\star\})\leq D,\\
& \{(x_t,v_t,\mathrm{dir}_t)\}_{t=1,\ldots,T} \text{ is generated by~\eqref{eq:OFW_general}}.\\
\end{aligned}
\end{equation*}

As $x_2, \cdots, x_T$ are convex combination of $x_1, v_1, \cdots, v_{T-1}$,
we can omit them in the boundary and diameter constraints,
leading to the following simpler finite-dimensional reformulation of~\eqref{eq:PEP}:
\begin{equation}\label{eq:PEP_finite_dim}
\begin{aligned}
B_T(\{(\eta_{t,s},\beta_{t,s},\gamma_{t,s})\}_{t,s}) =
\!\!\!\!
\sup_{\substack{\{g_t\}_{t=1,\ldots,T},\, x_\star\\ \{(x_t, v_t, \mathrm{dir}_t)\}_{t=1,\ldots,T}\\d\in\mathbb{N}}} \,
&  \sum_{t=1}^T \langle g_t, x_t - x_\star \rangle \\
\text{subject to: } 
& \{(x_t, \mathrm{dir}_t)\}_{t=1,\ldots,T} \text{ compatible with~\eqref{eq:OFW_general}},\\
& \langle -\mathrm{dir}_t, u - v_t \rangle \leq 0 \text{ for all }t=1,\ldots,T-1 \\
& \qquad\qquad\quad 	
	\text{ and } u\in\{x_1, v_{1},\cdots, v_{T-1}, x_\star\}, \\ 
&\mathrm{Diam}(\{ x_1, v_1, \cdots, v_{T-1}, x_\star \})\leq D,\\
&  \Vert g_t \Vert \leq L \text{ for }t=1,\ldots,T. \\
\end{aligned}
\end{equation}

\providecommand{\transp}{\mathrm{T}}
\providecommand{\barVect}[1]{\bar{#1}}
\providecommand{\R}{\mathbb{R}}

Finally, this sampled version~\eqref{eq:PEP_finite_dim} of~\eqref{eq:PEP} can be lifted to a semidefinite program via a standard change of variables: 
all vectors and gradients appearing in~\eqref{eq:PEP_finite_dim} are replaced with their Gram matrix (which, recall, encodes all pairwise inner products between these vectors / gradients). 
As the problem is invariant by translation of the vectors $x_t$, $v_t$ and $x_\star$
and of the feasible set $\cK$, without loss of generality,
we can assume that $x_1 = 0$.
Let $G$ be the Gram matrix of the gradients $g_1,\cdots,g_T$
and the vectors $v_1,\cdots,v_{T-1}$ and $x_\star$, that is,
$G = P^\transp P$ with
$P = [ \, g_1 \, | \, \cdots \, | \, g_T \, | \, v_1 \, | \, \cdots \, | \, v_{T-1} \, | \, x_\star \, ]$.

Let $e_1,\dots,e_{2T}$ be the standard basis vectors of $\R^{2 T}$.
Define $\barVect{g}_t = e_t$ for $t=1,\dots,T$,
$\barVect{v}_t = e_{T+t}$ for $t=1,\dots,T-1$,
$\barVect{x}_\star = e_{2T}$ and $\barVect{x}_1 = 0$.
Thus, we get that $\barVect{g}_s^\transp G \barVect{g}_t = \langle g_s, g_t \rangle$
for all $s,t = 1,\dots,T$;
and similarly with the other vectors $\barVect{v}_t$
for $t=1,\dots,T-1$, $\barVect{x}_\star$ and $\barVect{x}_1$.
For two vectors $u$ and $v$ in $\R^{2T}$, we define
$u \odot v = ( u v^\transp + v u^\transp ) / 2$ their symmetric outer product.
We denote by $\tr$ the trace operator for square matrices.

The objective, as well as all constraints, then become linear functions of the entries of the Gram matrix $G$.
Indeed, note that for all $u,v\in \{ x_1, v_1, \cdots, v_{T-1}, x_\star \}$ and $t \in \llbracket 1, T \rrbracket$, we have:
\begin{align*}
    & \|g_t\|^2 = \barVect{g}_t^\transp G \bar g_t = \tr( \barVect{g}_t \odot \barVect{g}_t) \, G), \\     & \|u-v\|^2 =  \tr( ((\barVect{u}-\barVect{v}) \odot (\barVect{u}-\barVect{v})) \, G), \\
    & \langle g_t, x_t \rangle = \left\langle g_t, \sum_{s=1}^{t-1} \gamma_{t,s}\, v_s \right\rangle = \sum_{s=1}^{t-1} \tr((\barVect{g}_t \odot \barVect{v}_s)\, G), \\     & \langle g_t, x_* \rangle = \tr((\bar g_t \odot \bar x_\star)\, G)  \,.
\end{align*}
This gives us the following linear (in $G$) semidefinite program reformulation of~\eqref{eq:PEP_finite_dim}:
\begin{equation}\label{eq:PEP_SDP}
\begin{aligned}
B_T(\{(\eta_{t,s},\beta_{t,s},\gamma_{t,s})\}_{t,s}) = \sup_{\substack{ G \succeq 0 }} \,
&  \sum_{t=1}^T \sum_{s=1}^{t-1} \gamma_{t,s} \tr( (\barVect{g}_t \odot \barVect{v}_s)  \, G)
			-  \sum_{t=1}^T \tr( (\barVect{g}_t \odot \barVect{x}_\star) \, G) \\
\text{subject to: } 
& \barVect{\mathrm{dir}}_t = \sum_{s=1}^t \eta_{t,s}\, \barVect{g}_s 
						+ \sum_{s=1}^{t-1} \beta_{t,s}\, \barVect{v}_s
	\text{ for } t=1,\dots,T-1, \\
& \tr( (\barVect{\mathrm{dir}}_t \odot (\barVect{v}_t - u) ) \, G) \leq 0
	\text{ for all }t=1,\ldots,T-1 \\
& \qquad\qquad\quad 	
	\text{ and } u\in\{\barVect{x}_1, \barVect{v}_{1},\cdots, \barVect{v}_{T-1}, \barVect{x}_\star\}, \\ 
& \tr( ((u-v) \odot (u-v)) \, G)  \leq D^2 \\
& \qquad\qquad\quad 	
\text{ for } u,\, v \in 
	\{ \barVect{x}_1, \barVect{v}_1, \cdots, \barVect{v}_{T-1}, \barVect{x}_\star \}, \\
&  \tr( (\barVect{g}_t \odot \barVect{g}_t) \, G)  \leq L^2 \text{ for }t=1,\ldots,T. \\
\end{aligned}
\end{equation}
Note that the variable $d\in\mathbb{N}$ for the dimension of the feasible set $\cK$
(and of all the gradients and vectors)
that appears in~\eqref{eq:PEP} and~\eqref{eq:PEP_finite_dim}
induces a constraint imposing that the rank of the matrix $G$ is at most $d$.
However, this constraint disappears when taking the supremum over all values of $d\in\mathbb{N}$.

\providecommand{\lambdaLip}{\lambda^{\mathrm{Lip}}}
\providecommand{\lambdaDiam}{\lambda^{\mathrm{Diam}}}
\providecommand{\lambdaBound}{\lambda^{\mathrm{Brd}}}

As explained in Remark~\ref{rem_proofs_from_PEP},
for given numerical values of $T,L,D$ and the algorithm parameters,
solving the tractable convex problem~\eqref{eq:PEP_SDP}
allows us to get worst-case examples
giving \emph{algorithm-dependent lower bounds} on the worst-case regret.
In order to obtain \emph{algorithm-dependent upper bounds} on the worst-case regret, a natural procedure consists in formulating the \emph{Lagrange dual} of~\eqref{eq:PEP} (which is also a semidefinite program; see, e.g.,~\cite{vandenberghe1996semidefinite,boyd2004convex}), whose feasible points naturally corresponds to upper bounds on the regret. 
In this context, \emph{finding a proof} consists in finding a feasible point to the dual problem~\cite{goujaud2023fundamental}. 
In particular, it is useful to reformulate~\eqref{eq:PEP_SDP} as its Lagrange dual
which is the following semidefinite problem:
\begin{equation}\label{eq:PEP_SDP_dual}
\begin{aligned}
&
\begin{aligned}
B_T(\{(\eta_{t,s},\beta_{t,s},\gamma_{t,s})\}_{t,s}) = \inf_{\substack{ \lambdaLip\geq 0 \\ \lambdaDiam\geq 0 \\ \lambdaBound \geq 0 }} \,
&  \sum_{t=1}^T \lambdaLip_t\, L^2 
+ \frac{1}{2} \sum_{u,\, v \in 
	\{ \barVect{x}_1, \barVect{v}_1, \cdots, \barVect{v}_{T-1}, \barVect{x}_\star \}}
		\lambdaDiam_{\{u,v\}}\, D^2
\\
\text{subject to: } 
& S(\eta,\beta,\gamma; \lambda) \succeq 0, \\
\end{aligned}\\
&
\begin{aligned}
 \text{ where }  S(\eta,\beta,\gamma; \lambda) = 
	& \ \frac{1}{2} \sum_{u,\, v \in 
	\{ \barVect{x}_1, \barVect{v}_1, \cdots, \barVect{v}_{T-1}, \barVect{x}_\star \}}
		\lambdaDiam_{\{u,v\}}\,  ((u-v) \odot (u-v)) \\
&  + \sum_{t=1}^T \lambdaLip_t\, (\barVect{g}_t \odot \barVect{g}_t)  
+ \sum_{t=1}^{T-1} \sum_{u\in\{\barVect{x}_1, \barVect{v}_{1},\cdots, \barVect{v}_{T-1}, \barVect{x}_\star \}} \!\!\!\!\!\!
	\lambdaBound_{{\barVect{v}_t}, u}  \,
		(\barVect{\mathrm{dir}}_t \odot (\barVect{v}_t - u) ) \\
&  - \sum_{t=1}^T \sum_{s=1}^{t-1} \gamma_{t,s} \, (\barVect{g}_t \odot \barVect{v}_s)
			+  \sum_{t=1}^T (\barVect{g}_t \odot \barVect{x}_\star), \\
\end{aligned}\\
&
\text{ and }
\lambda = (\lambdaLip, \lambdaDiam, \lambdaBound) \text{ and }
\barVect{\mathrm{dir}}_t = \sum_{s=1}^t \eta_{t,s}\, \barVect{g}_s 
						+ \sum_{s=1}^{t-1} \beta_{t,s}\, \barVect{v}_s
	\text{ for } t=1,\dots,T-1 .
\end{aligned}
\end{equation}

\subsection{Joint stepsize optimization; semidefinite formulation of~\eqref{eq:jointstepsizeopt}}\label{a:stepsizeopt}

As mentioned in Section~\ref{s:opt_param},
a natural path forward is to use~\eqref{eq:PEP} to obtain worst-case optimal algorithms. That is, by solving
\[ \min_{\{(\eta_{t,s},\beta_{t,s},\gamma_{t,s})\}_{t,s}} \left\{B_T(\{(\eta_{t,s},\beta_{t,s},\gamma_{t,s})\}_{t,s}) \text{ s.t. } \sum_{s=1}^{t-1}\gamma_{t,s}\leq 1,\, \gamma_{t,s}\geq 0\right\}.\]
Using \eqref{eq:PEP_SDP_dual}, this problem can be reformulated as
a linear optimization problem with a bilinear matrix inequality constraint, which is unfortunately NP-hard in general~\cite{toker1995np}.
The bilinearity in the matrix inequality $S(\eta,\beta,\gamma; \lambda) \succeq 0$
is due to terms 
$\lambdaBound_{{\barVect{v}_t}, u}  \, 
	(\barVect{\mathrm{dir}}_t \odot (\barVect{v}_t - u) )$
appearing in the definition of $S(\eta,\beta,\gamma; \lambda)$
which are bilinear in $(\eta,\beta)$ and $\lambdaBound$.
A classical approach (see~\cite{droriPerformanceFirstorderMethods2014}) 
to circumvent bilinearity in the matrix inequalities
consists in using convex relaxation of the problem,
which works well in our case as we explained in Section~\ref{s:opt_param}
and we detail in this section.
Note that another possible approach (see \cite{dasguptaBranchandboundPerformanceEstimation2024})
consists in adapting a branch-and-bound algorithm to compute the best possible regret guarantee by:
(i)~dividing the search space into regions, 
(ii)~computing upper and lower bounds on the best possible regret guarantee 
for each region via convex relaxations of the problem,
(iii)~discarding regions whose lower bound is larger than the best (across all regions) current upper bound as those regions cannot contain the optimal point / value,
and (iv)~repeating steps (i)--(iii) with the remaining regions until convergence.
However, this branch-and-bound approach is numerically more costly
and is not necessary when the direct convex relaxation method works.

For this reason, we propose a slight relaxation of $B_T(\{(\eta_{t,s},\beta_{t,s},\gamma_{t,s})\}_{t,s})$ which corresponds to removing a few constraints from~\eqref{eq:PEP} (numerically observed to be inactive). More precisely, this relaxation is obtained through: (i)~we observe that all $x_t$ for $t=2, \cdots, T$ are in the convex hull of $x_1, v_1, \cdots, v_{T-1}$, and thus the domain constraints for $\cK$ are imposed only on vectors $x_1, v_1, \cdots, v_{T-1}, x_\star$; (ii)~we keep only the boundary constraints corresponding to the optimality of $v_t$ compared with $v_{t+1}, \cdots, v_{T-1}, x_\star$:
this leads to the definition of $W_T\bigl(\{\eta_{t,s},\beta_{t,s},\gamma_{t,s}\}_{t,s}\bigr)$
in Section~\ref{s:opt_param} with 
$W_T\bigl(\{\eta_{t,s},\beta_{t,s},\gamma_{t,s}\}_{t,s}\bigr)\geq B_T\bigl(\{\eta_{t,s},\beta_{t,s},\gamma_{t,s}\}_{t,s}\bigr)$. 
In terms of the semidefinite program reformulation~\eqref{eq:PEP_SDP_dual}
of~\eqref{eq:PEP}, this relaxation corresponds to impose
$\lambdaBound_{\barVect{v}_t, u} = 0$ for all $t=1,\dots,T-1$ and
$u\in \{ \barVect{v}_{t+1}, \cdots, \barVect{v}_{T-1}, \barVect{x}_\star \}$,
giving us:
\begin{equation}\label{eq:PEP_relaxed_SDP_dual}
\begin{aligned}
&
\begin{aligned}
W_T(\{(\eta_{t,s},\beta_{t,s},\gamma_{t,s})\}_{t,s}) = \inf_{\substack{ \lambdaLip\geq 0 \\ \lambdaDiam\geq 0 \\ \lambdaBound \geq 0 }} \,
&  \sum_{t=1}^T \lambdaLip_t\, L^2 
+ \frac{1}{2} \sum_{u,\, v \in 
	\{ \barVect{x}_1, \barVect{v}_1, \cdots, \barVect{v}_{T-1}, \barVect{x}_\star \}}
		\lambdaDiam_{\{u,v\}}\, D^2
\\
\text{subject to: } 
& S(\eta,\beta,\gamma; \lambda) \succeq 0, \\
\end{aligned}\\
&
\begin{aligned}
 \text{ where }  S(\eta,\beta,\gamma; \lambda) = 
	& \ \frac{1}{2} \sum_{u,\, v \in 
	\{ \barVect{x}_1, \barVect{v}_1, \cdots, \barVect{v}_{T-1}, \barVect{x}_\star \}}
		\lambdaDiam_{\{u,v\}}\,  ((u-v) \odot (u-v)) \\
&  + \sum_{t=1}^T \lambdaLip_t\, (\barVect{g}_t \odot \barVect{g}_t)  
+ \sum_{t=1}^{T-1} \sum_{u\in\{\barVect{v}_{t+1},\cdots, \barVect{v}_{T-1}, \barVect{x}_\star \}} \!\!\!\!\!\!
	\lambdaBound_{{\barVect{v}_t}, u}  \,
		(\barVect{\mathrm{dir}}_t \odot (\barVect{v}_t - u) ) \\
&  - \sum_{t=1}^T \sum_{s=1}^{t-1} \gamma_{t,s} \, (\barVect{g}_t \odot \barVect{v}_s)
			+  \sum_{t=1}^T (\barVect{g}_t \odot \barVect{x}_\star), \\
\end{aligned}\\
&
\text{ and }
\lambda = (\lambdaLip, \lambdaDiam, \lambdaBound) \text{ and }
\barVect{\mathrm{dir}}_t = \sum_{s=1}^t \eta_{t,s}\, \barVect{g}_s 
						+ \sum_{s=1}^{t-1} \beta_{t,s}\, \barVect{v}_s
	\text{ for } t=1,\dots,T-1 .
\end{aligned}
\end{equation}

Using this problem, the joint minimization problem
\begin{equation}\tag{\ref{eq:jointstepsizeopt}}
     \min_{\{(\eta_{t,s},\beta_{t,s},\gamma_{t,s})\}_{t,s}} \left\{W_T(\{(\eta_{t,s},\beta_{t,s},\gamma_{t,s})\}_{t,s}) \text{ s.t. } \sum_{s=1}^{t-1}\gamma_{t,s}\leq 1,\, \gamma_{t,s}\geq 0\right\}
\end{equation}
is still \emph{a priori} a linear optimization problem with a bilinear matrix inequality constraint
(which recall are unfortunately NP-hard in general~\cite{toker1995np})
due to the presence in $S(\eta,\beta,\gamma; \lambda)$
of the bilinear (in $(\eta,\beta)$ and $\lambdaBound$) term
(where abusing notations, we write $\barVect{v}_T = \barVect{x}_\star$):
\begin{multline*}
\sum_{t=1}^{T-1} \sum_{u\in\{\barVect{v}_{t+1},\cdots, \barVect{v}_{T-1}, \barVect{x}_\star \}} \!\!\!\!\!\!
	\lambdaBound_{{\barVect{v}_t}, u}  \,
		(\barVect{\mathrm{dir}}_t \odot (\barVect{v}_t - u) ) \\
\begin{aligned}
= & 
\sum_{t=1}^{T-1} \sum_{s=1}^t
\sum_{j=t+1}^T
	\lambdaBound_{{\barVect{v}_t}, \barVect{v}_j}  \,
		\eta_{t,s} \, (\barVect{g}_s \odot (\barVect{v}_t - \barVect{v}_j) ) 
+ \sum_{t=1}^{T-1} \sum_{s=1}^{t-1}
\sum_{j=t+1}^T
	\lambdaBound_{{\barVect{v}_t}, \barVect{v}_j}  \,
		\beta_{t,s} \, (\barVect{v}_s \odot (\barVect{v}_t - \barVect{v}_j) ) .
\end{aligned}
\end{multline*}
However, we remark that:
\begin{align*}
\sum_{t=1}^{T-1} \sum_{s=1}^t \sum_{j=t+1}^T
	\lambdaBound_{{\barVect{v}_t}, \barVect{v}_j}  \,
		\eta_{t,s} \, (\barVect{g}_s \odot (\barVect{v}_t - \barVect{v}_j) ) = & 
\sum_{t=1}^{T-1} \sum_{s=1}^t \left( \sum_{j=t+1}^T
	\lambdaBound_{{\barVect{v}_t}, \barVect{v}_j}  \,
		\eta_{t,s} \right) (\barVect{g}_s \odot \barVect{v}_t )
- \sum_{j=1}^T \sum_{s=1}^j \sum_{t=s}^{j-1}
	\lambdaBound_{{\barVect{v}_t}, \barVect{v}_j}  \,
		\eta_{t,s} \, (\barVect{g}_s \odot \barVect{v}_j)		
\\
= & 
\sum_{t=1}^{T-1} \sum_{s=1}^t \left( \sum_{j=t+1}^T
	\lambdaBound_{{\barVect{v}_t}, \barVect{v}_j}  \,
		\eta_{t,s} 
- \sum_{j=s}^{t-1} \lambdaBound_{{\barVect{v}_j}, \barVect{v}_t}  \,
		\eta_{j,s} \right) (\barVect{g}_s \odot \barVect{v}_t)	,	
\end{align*}
and:
\begin{align*}
\sum_{t=1}^{T-1} \sum_{s=1}^{t-1} \sum_{j=t+1}^T
	\lambdaBound_{{\barVect{v}_t}, \barVect{v}_j}  \,
		\beta_{t,s} \, (\barVect{v}_s \odot (\barVect{v}_t - \barVect{v}_j) ) = & 
\sum_{t=1}^{T-1} \sum_{s=1}^{t-1} \left( \sum_{j=t+1}^T
	\lambdaBound_{{\barVect{v}_t}, \barVect{v}_j}  \,
		\beta_{t,s} \right) (\barVect{v}_s \odot \barVect{v}_t )
- \sum_{j=1}^T \sum_{s=1}^{j-1} \sum_{t=s+1}^{j-1}
	\lambdaBound_{{\barVect{v}_t}, \barVect{v}_j}  \,
		\beta_{t,s} \, (\barVect{v}_s \odot \barVect{v}_j)		
\\
= & 
\sum_{t=1}^{T-1} \sum_{s=1}^{t-1} \left( \sum_{j=t+1}^T
	\lambdaBound_{{\barVect{v}_t}, \barVect{v}_j}  \,
		\beta_{t,s} 
- 	\sum_{j=s+1}^{t-1} \lambdaBound_{{\barVect{v}_j}, \barVect{v}_t}  \,
		\beta_{j,s} \right) (\barVect{v}_s \odot \barVect{v}_t)		.
\end{align*}

This motivates the following change of variables in \eqref{eq:PEP_relaxed_SDP_dual}
which allows us to recast~\eqref{eq:jointstepsizeopt}
as a linear convex semidefinite program again:
$\eta$, $\beta$ and $\lambdaBound$ are replaced by:
\begin{equation}\label{eq:SDP_opt_param_change_vars}
\begin{aligned}
& B_{t,s} 
= \eta_{t,s} \sum_{j=t+1}^{T} \lambdaBound_{\barVect{v}_t, \barVect{v}_j}
		- \sum_{j=s}^{t-1} \eta_{j,s}\,  \lambdaBound_{\barVect{v}_j, \barVect{v}_t}
& \qquad \forall 1 \leq s \leq t \leq T, 
\\
& C_{t,s} 
= \beta_{t,s} \sum_{j=t+1}^{T} \lambdaBound_{\barVect{v}_t, \barVect{v}_j}
		- \sum_{j=s+1}^{t-1} \beta_{j,s}\,  \lambdaBound_{\barVect{v}_j, \barVect{v}_t}
& \qquad \forall 1 \leq s < t \leq T.
\end{aligned}
\end{equation}
Note that the other variables, that is, $\gamma$, $\lambdaLip$ and $\lambdaDiam$,
are left unchanged. 
Also note that from the definition of $B_{t,s}$ and $C_{t,s}$, we get that
they must satisfy the constraints
$\sum_{s=t}^T B_{t,s} = 0$ and $\sum_{s=t+1}^T B_{t,s} = 0$
for all $t=1,\dots,T$.
Thus, \eqref{eq:jointstepsizeopt} can be reformulated as the following
linear convex semidefinite program:
\begin{equation}\label{eq:PEP_relaxed_joint_opt}
\begin{aligned}
&
\begin{aligned}
\inf_{\{(B_{t,s},C_{t,s},\gamma_{t,s})\}_{t,s}}
\inf_{\substack{ \lambdaLip\geq 0 \\ \lambdaDiam\geq 0 \\ \lambdaBound \geq 0 }} \,
&  \sum_{t=1}^T \lambdaLip_t\, L^2 
+ \frac{1}{2} \sum_{u,\, v \in 
	\{ \barVect{x}_1, \barVect{v}_1, \cdots, \barVect{v}_{T-1}, \barVect{x}_\star \}}
		\lambdaDiam_{\{u,v\}}\, D^2
\\
\text{subject to: } 
& S(B,C,\gamma; \lambda) \succeq 0, \\
& \sum_{t=s}^T B_{t,s} = 0 \text{ and }
	\sum_{t=s+1}^T C_{t,s} = 0 \text{ for } s=1,\ldots,T \\
\end{aligned}\\
&
\begin{aligned}
 \text{ where }  S(B,C,\gamma; \lambda) = 
	& \ \frac{1}{2} \sum_{u,\, v \in 
	\{ \barVect{x}_1, \barVect{v}_1, \cdots, \barVect{v}_{T-1}, \barVect{x}_\star \}}
		 \lambdaDiam_{\{u,v\}}\,  ((u-v) \odot (u-v)) \\
& + \sum_{t=1}^T \lambdaLip_t\, (\barVect{g}_t \odot \barVect{g}_t)   + \sum_{t=1}^{T} \sum_{s=1}^t B_{t,s} \, (\barVect{g}_s \odot \barVect{v}_t)
+ \sum_{t=1}^{T} \sum_{s=1}^{t-1} C_{t,s} \, (\barVect{v}_s \odot \barVect{v}_t) \\
&  - \sum_{t=1}^T \sum_{s=1}^{t-1} \gamma_{t,s} \, (\barVect{g}_t \odot \barVect{v}_s)
			+  \sum_{t=1}^T (\barVect{g}_t \odot \barVect{x}_\star) , \\
\end{aligned}\\
&
\text{ and }
\lambda = (\lambdaLip, \lambdaDiam, \lambdaBound) .
\end{aligned}
\end{equation}

Note that there are $T^2$ of the variables $\{ B_{t,s}, C_{t,s} \}_{t,s}$,
while there was $(T-1)^2 + T(T-1)/2$ of the variables $\eta$, $\beta$ and $\lambdaBound$;
this is because problem~\eqref{eq:PEP_relaxed_SDP_dual} was overparametrized.
Thus, when inverting the change of variables,
for given values of $\{ B_{t,s}, C_{t,s} \}_{t,s}$,
there exist several solutions for $\eta$, $\beta$ and $\lambdaBound$
solving~\eqref{eq:SDP_opt_param_change_vars}.
We propose one solution for inverting~\eqref{eq:SDP_opt_param_change_vars} 
giving simple algorithms:
(i)~we choose $\eta_{t,s} = 1$ for all $1 \leq s \leq t \leq T$,
(ii)~we solve for $\lambdaBound$ using the first line of~\eqref{eq:SDP_opt_param_change_vars}
substituting in the value of $\eta$,
(iii)~we solve for $\beta$ using the second line of~\eqref{eq:SDP_opt_param_change_vars}
substituting in the value of $\lambdaBound$;
that is:
\begin{equation}\label{eq:param_solve_change_vars}
\begin{aligned}
& \eta_{t,s} = 1 & \qquad \forall 1 \leq s \leq t \leq T, \\
& \lambdaBound_{\barVect{v}_s, \barVect{v}_t} = B_{t,s+1} - B_{t,s}
	& \qquad \forall 1 \leq s < t \leq T, \\
& \beta_{t,s} 
= \frac{1}{B_{t,t}} \left(
	C_{t,s} 
	+ \sum_{j=s+1}^{t-1} \beta_{j,s}\,  \lambdaBound_{\barVect{v}_j, \barVect{v}_t}
\right)
& \qquad \forall 1 \leq s < t \leq T.
\end{aligned}
\end{equation}

\subsection{Design and optimization of the proof of Theorem~\ref{thm_upper_bound_OFW}}
\label{a:proof_design}

As we explained in Section~\ref{sec:design_proof_theo},
we used variants of~\eqref{eq:PEP} to design the potential-based proof
of Section~\ref{sec:theo_proof} and to jointly optimize the algorithm and potential parameters for this proof.
We outlined our general method in Section~\ref{sec:design_proof_theo},
and here we present the details for 
rewriting~\eqref{eq:PEP-potential} as a semidefinite program,
numerically solving it, and then for obtaining the optimal parameters
given at the end of Section~\ref{sec:design_proof_theo}.

We start by restating the $1$-iteration inner maximization problem from~\eqref{eq:PEP-potential}
which upper bounds the potential increase 
for an abstract $t$ and for given pontential parameters $a$ and $b$:
\begin{equation}\label{eq:PEP-potential_inner}
\begin{aligned}
B_t(\eta,\sigma,a,b) \triangleq
\sup_{\substack{ \cK,\, d\in\mathbb{N}\\
    g_t,\, G_{t-1} \\
    x_1,\, x_t,\, v_t,\, x_{t+1} \\ y_t,\, y_{t+1}  }} \,
& \phi_t - \phi_{t-1} \\
\text{subject to: } 
& \phi_t - \phi_{t-1} \text{ is generated from }
    \{x_t, x_{t+1}, y_t, y_{t+1}, g_t, G_{t-1}, x_1\} \text{  by \eqref{eq_def_general_pot}}, \\
&  \Vert g_t \Vert \leq L ,\\
& \cK \text{ is a non-empty closed convex set of $\mathbb{R}^d$,}\\ 
&\mathrm{Diam}(\{x_1, x_t, v_t, x_{t+1}, y_t, y_{t+1} \})\leq D,\\
& (x_{t+1}, v_t) \text{ are generated from } \{x_1, x_t, g_t, G_{t-1}\} \text{ by Algorithm~\ref{OFW_alg_new}},\\
& y_t \text{ and } y_{t+1} \text{ are generated from } \{x_1, g_t, G_{t-1}\} \text{ by FTRL}.\\
\end{aligned}
\end{equation}
In particular, we have that~\eqref{eq:PEP-potential} is equal to
the minimum of $B_t(\eta,\sigma,a,b)$ over $a\geq 0$ and $b\geq 0$.

Note that $x_{t+1}$ can be removed from the boundary conditions
as it is combination of $v_t$ and $x_t$.
Substituting in the definition of $\phi_t-\phi_{t-1}$
and the optimality condition of $v_t$, $y_t$ and $y_{t+1}$,
we can get rid of $\cK$, and we obtain the following self-containing reformulation
of~\eqref{eq:PEP-potential_inner}:
\begin{equation}\label{eq:PEP-potential_inner_2}
\begin{aligned}
B_t(\eta,\sigma,a,b) = 
\sup_{\substack{ d\in\mathbb{N}\\
    g_t,\, G_{t-1} \\
    x_1,\, x_t,\, v_t,\, x_{t+1} \\ y_t,\, y_{t+1}  }} \,
& 
\left\{
\begin{aligned} &
\langle g_t, x_t - y_{t+1} \rangle + \langle G_{t-1}, y_t - y_{t+1} \rangle \\
&	+ a ( \Vert x_{t+1} - y_{t+1}\Vert^2 - \Vert x_{t} - y_{t}\Vert^2) \\
&	+  b\, \eta(  \langle G_{t-1}+g_t, x_{t+1} - y_{t+1} \rangle 
				- \langle G_{t-1}, x_{t} - y_{t} ) \rangle  \\
&	+ \frac{b}{2} \bigl( \Vert x_{t+1} - x_1 \Vert^2 - \Vert y_{t+1} - x_1 \Vert^2 \bigr. \\
				& \qquad\qquad \bigl.
				- \Vert x_{t} - x_1 \Vert^2 + \Vert y_{t} - x_1 \Vert^2  \bigr) \\
&	+ \frac{1}{2\eta} ( \Vert y_{t} - x_1 \Vert^2 - \Vert y_{t+1} - x_1 \Vert^2 )
\end{aligned}\right\}
\\
\text{subject to: } 
&  \Vert g_t \Vert \leq L ,\\
&\mathrm{Diam}(\{x_1, x_t, v_t, y_t, y_{t+1} \})\leq D,\\
& x_{t+1} = \sigma v_t + (1-\sigma) x_t, \\
& \langle \eta \, G_{t-1} + (y_t - x_1), y_t - u \rangle \leq 0
	\text{ for } u \in \{ x_t, v_t, y_{t+1} \},\\ 
& \langle \eta \, (G_{t-1} + g_t) + (y_{t+1} - x_1), y_{t+1} - u \rangle \leq 0
	\text{ for } u \in \{ x_t, v_t, y_t \},\\
& \langle \eta \, (G_{t-1} + g_t) + (x_t - x_1), v_t - u \rangle \leq 0
	\text{ for } u \in \{ x_t, y_t, y_{t+1} \}.\\
\end{aligned}
\end{equation}

Without loss of generality, we assume that $x_1 = 0$.
Then, using the same method as detailed in Appendix~\ref{a:pep},
we can reformulate this problem as a linear convex semidefinite program.
To this end, we do a change of variables in~\eqref{eq:PEP-potential_inner_2},
replacing the gradients $G_{t-1}$ and $g_t$ and the vectors $x_t$, $v_t$, $y_t$
and $y_{t+1}$ by their Gram matrix $H$.
Let $e_1, \ldots, e_6$ be the standard basis vectors of $\R^6$.
Define $\barVect{g}_t = e_1$, $\barVect{G}_{t-1} = e_2$, $\barVect{x}_t = e_3$,
$\barVect{v}_t = e_4$, $\barVect{y}_t = e_5$ and $\barVect{y}_{t+1} = e_6$.
Thus, we have $\barVect{g}_t^{\transp} H \barVect{x}_t = \langle g_t, x_t \rangle$,
and similarly with the others gradients / vectors.
Define the matrix $A_{\mathrm{obj}}(\eta,\sigma,a,b)$, 
which is linear in $(a,b)$ but not in $(\eta,\sigma,a,b)$, as:
\begin{equation}
\begin{aligned} 
A_{\mathrm{obj}}(\eta,\sigma,a,b)
\ \triangleq & \
\barVect{g}_t \odot (\barVect{x}_t - \barVect{y}_{t+1}) 
	+ \barVect{G}_{t-1} \odot (\barVect{y}_t - \barVect{y}_{t+1}) \\
&	+ a ( (\barVect{x}_{t+1} - \barVect{y}_{t+1}) \odot (\barVect{x}_{t+1} - \barVect{y}_{t+1})
		-  (\barVect{x}_{t} - \barVect{y}_{t}) \odot (\barVect{x}_{t} - \barVect{y}_{t}) ) \\
&	+  b\, \eta(  
	( \barVect{G}_{t-1}+\barVect{g}_t)\odot (\barVect{x}_{t+1} - \barVect{y}_{t+1})
		- \barVect{G}_{t-1} \odot (\barVect{x}_{t} - \barVect{y}_{t}) ) \\
&	+ \frac{b}{2} ( \barVect{x}_{t+1} \odot \barVect{x}_{t+1} 
					- \barVect{y}_{t+1} \odot \barVect{y}_{t+1}
				- \barVect{x}_{t} \odot \barVect{x}_{t} 
					+ \barVect{y}_{t} \odot \barVect{y}_{t}  ) \\
&	+ \frac{1}{2\eta} ( \barVect{y}_{t} \odot \barVect{y}_{t} 
					-\barVect{y}_{t+1} \odot \barVect{y}_{t+1} ) ,
\end{aligned}
\end{equation}
where $\barVect{x}_{t+1} = \sigma \barVect{v}_t + (1-\sigma) \barVect{x}_t$.
Also define:
\begin{align*}
& \mathrm{dir}_{\barVect{v}_t} 
	= \eta \, \barVect{G}_{t-1} + \eta \, \barVect{g}_t + \barVect{x}_{t}, \\
& \mathrm{dir}_{\barVect{y}_t} 
	= \eta \, \barVect{G}_{t-1}  + \barVect{y}_{t}, \\
& \mathrm{dir}_{\barVect{y}_{t+1}} 
	= \eta \, \barVect{G}_{t-1} + \eta \, \barVect{g}_t + \barVect{y}_{t+1}.
\end{align*}
With this change of variables, \eqref{eq:PEP-potential_inner_2}
can be reformulated as the following linear convex semidefinite program:
\begin{equation}\label{eq:PEP-potential_inner_3}
\begin{aligned}
B_t(\eta,\sigma,a,b) =
\sup_{H \succeq 0} \,
& 
\tr( A_{\mathrm{obj}} \, H) 
\\
\text{subject to: } 
&  \tr( (\barVect{g}_t \odot \barVect{g}_t)\, H) \leq L^2 ,\\
& \tr( (u \odot v)\, H) \leq D^2
	\text{ for } u,v \in \{  \barVect{x}_1, \barVect{x}_t, \barVect{v}_t, \barVect{x}_{t+1}, \barVect{y}_t, \barVect{y}_{t+1} \} , \\
& \tr\bigl( (\mathrm{dir}_{u} \odot (u - v)) \, H \bigr) \leq 0
	\text{ for } u \in \{ \barVect{v}_t, \barVect{y}_t, \barVect{y}_{t+1} \} \\
& \qquad\qquad\qquad\qquad\qquad\qquad\!\!
	\text{ and } v \in \{ \barVect{x}_1, \barVect{x}_t, \barVect{v}_t, \barVect{y}_t, \barVect{y}_{t+1} \} \setminus \{u\}.\\ 
\end{aligned}
\end{equation}

We then reformulate~\eqref{eq:PEP-potential_inner_3} as its Lagrange dual, giving us:
\begin{equation}\label{eq:PEP-potential_inner_4}
\begin{aligned}
&
\begin{aligned}
B_t(\eta,\sigma,a,b) =
\inf_{\substack{ \lambdaLip\geq 0 \\ \lambdaDiam\geq 0 \\ \lambdaBound \geq 0 }} \,
&  \lambdaLip\, L^2 
+ \frac{1}{2} \sum_{u,\, v \in 
	\{  \barVect{x}_1, \barVect{x}_t, \barVect{v}_t, \barVect{y}_t, \barVect{y}_{t+1} \}}
		\lambdaDiam_{\{u,v\}}\, D^2
\\
\text{subject to: } 
& S(\eta,\sigma,a,b; \lambda) \succeq 0, \\
\end{aligned}\\
&
\begin{aligned}
 \text{ where }  S(\eta,\sigma,a,b; \lambda) = 
	& \ \frac{1}{2} \sum_{u,\, v \in 
	\{  \barVect{x}_1, \barVect{x}_t, \barVect{v}_t, \barVect{y}_t, \barVect{y}_{t+1} \}}
		\lambdaDiam_{\{u,v\}}\,  ((u-v) \odot (u-v)) \\
&  
+ \sum_{u\in\{  \barVect{v}_t, \barVect{y}_t, \barVect{y}_{t+1} \}} 
	\sum_{v\in\{  \barVect{x}_1, \barVect{x}_t, \barVect{v}_t, \barVect{y}_t, \barVect{y}_{t+1} \} \setminus \{u\} } \!\!\!\!\!\!
	\lambdaBound_{u, v}  \,
		(\barVect{\mathrm{dir}}_u \odot (u - v) ) \\
& + \lambdaLip\, (\barVect{g}_t \odot \barVect{g}_t)   - A_{\mathrm{obj}}(\eta,\sigma,a,b) , \\
\end{aligned}\\
&
\text{ and }
\lambda = (\lambdaLip, \lambdaDiam, \lambdaBound) . \end{aligned}
\end{equation}

\providecommand{\lambdaP}{\tilde\lambda}

Now, jointly minimizing \eqref{eq:PEP-potential_inner_4} in $(a,b)$ and $\lambda$
for fixed $(\eta,\sigma)$, we get a reformulation of~\eqref{eq:PEP-potential}
as the following linear convex semidefinite program:
\begin{equation}\label{eq:PEP-potential_SDP}
\begin{aligned}
&
\begin{aligned}
\inf_{a\geq 0, b\geq 0}
\inf_{\substack{ \lambdaLip\geq 0 \\ \lambdaDiam\geq 0 \\ \lambdaBound \geq 0 }} \,
&  \lambdaLip\, L^2 
+ \frac{1}{2} \sum_{u,\, v \in 
	\{  \barVect{x}_1, \barVect{x}_t, \barVect{v}_t, \barVect{y}_t, \barVect{y}_{t+1} \}}
		\lambdaDiam_{\{u,v\}}\, D^2
\\
\text{subject to: } 
& S(\eta,\sigma,a,b; \lambda) \succeq 0, \\
\end{aligned}\\
&
\begin{aligned}
 \text{ where }  S(\eta,\sigma,a,b; \lambda) = 
	& \ \frac{1}{2} \sum_{u,\, v \in 
	\{  \barVect{x}_1, \barVect{x}_t, \barVect{v}_t, \barVect{y}_t, \barVect{y}_{t+1} \}}
		\lambdaDiam_{\{u,v\}}\,  ((u-v) \odot (u-v)) \\
&  
+ \sum_{u\in\{  \barVect{v}_t, \barVect{y}_t, \barVect{y}_{t+1} \}} 
	\sum_{v\in\{  \barVect{x}_1, \barVect{x}_t, \barVect{v}_t, \barVect{y}_t, \barVect{y}_{t+1} \} \setminus \{u\} } \!\!\!\!\!\!
	\lambdaBound_{u, v}  \,
		(\barVect{\mathrm{dir}}_u \odot (u - v) ) \\
& + \lambdaLip\, (\barVect{g}_t \odot \barVect{g}_t)   - A_{\mathrm{obj}}(\eta,\sigma,a,b) , \\
\end{aligned}\\
&
\text{ and }
\lambda = (\lambdaLip, \lambdaDiam, \lambdaBound) . \end{aligned}
\end{equation}
Note that the size of~\eqref{eq:PEP-potential_SDP} does not depend on $T$,
which allows efficient numerical solving even for large values of $T$.
Numerically solving~\eqref{eq:PEP-potential_SDP} gives us that for
jointly optimal $(a,b)$ and $\lambda$ we have $b=0$,
$\lambdaDiam_{\{u,v\}} = 0$ except for
$\lambdaDiam_{\{\barVect{x}_t, \barVect{v}_t\}} = a \sigma^2$,
$\lambdaBound_{u,v} = 0$ except for
$\lambdaBound_{\barVect{v}_t,\barVect{y}_{t+1}}
= \lambdaBound_{\barVect{y}_{t+1},\barVect{v}_t} = 2 a \sigma$,
$\lambdaBound_{\barVect{y}_{t+1},\barVect{y}_{t}}$
and $\lambdaBound_{\barVect{y}_t,\barVect{y}_{t+1}} = 1 + \lambdaBound_{\barVect{y}_{t+1},\barVect{y}_{t}}$.
Following our notations from the proof of Lemma~\ref{lemma_upper_bound_OFW}
(see Appendix~\ref{a:proof_lemma_upper_bound}),
we write $\lambda^g = \lambdaLip$ and $\lambdaP =  \lambdaBound_{\barVect{y}_{t+1},\barVect{y}_{t}}$.
Note that the Cholesky decomposition of $S = S(\eta,\sigma,a,b;\lambda)$
encodes for the sum of squares $\tr(S \, H) \geq 0$ that must be added during the proof,
where some of those squares can be interpreted as inequalities.
The Cholesky decomposition of $S$ has 2 eigenvectors:
the first eigenvector 
$\sqrt{\lambda^g} \, \barVect{g}_t + \frac{1}{2\sqrt{\lambda^g}} \, (\barVect{x}_t - \barVect{y}_{t+1} + \lambdaP(\barVect{y}_t - \barVect{y}_{t+1}))$
encodes for the inequality
\[ \langle g_t, x_t - y_{t+1} + \lambdaP(y_t - y_{t+1}) \rangle
\leq \lambda^g \Vert g_t \Vert^2 + \frac{1}{4\lambda^g} \Vert x_t - y_{t+1} + \lambdaP(y_t - y_{t+1}) \Vert^2,  \]
while the second eigenvector indicates that the upper bound we get from the proof
will contain minus a sum of squares, which is composed of the vectors $x_t$, $y_t$ and $y_{t+1}$.
The values of all those Lagrange multipliers in $\lambda$ and the value of $S$
gives us what inequalities and in which quantity to use in the proof of Lemma~\ref{lemma_upper_bound_OFW}.
Thus, from all of this we can deduce the proof up to showing that $Q(\lambdaP)$ is a sum of squares, where:
\begin{align*}
 Q(\lambdaP) & \triangleq
\frac{1+2\lambdaP}{2\eta} \Vert y_t - y_{t+1} \Vert^2 
- \frac{1}{4 \lambda^g} \Vert x_t - y_{t+1} + \lambdaP(y_t - y_{t+1}) \Vert^2      
        \\ & \qquad - a ( (1-2 \sigma) \Vert x_{t} - y_{t+1}\Vert^2 
                                - \Vert x_{t} - y_{t}\Vert^2 ) .
\end{align*}

Hence, all that is left to do is finding optimal $\eta$, $\sigma$, $a$ $\lambda^g$
such that there exists $\tilde\lambda \geq 0$ for which $Q(\tilde\lambda)$ is a sum of squares
and that minimizes the upper bound given by the proof on the regret of OFW $R_T$,
which is:
\begin{equation*}
\frac{1}{2\eta}\Vert x_\star - x_1 \Vert + \sum_{t=1}^T B_t(\eta,\sigma,a,0)
\leq \frac{1}{2\eta} D^2 + \lambda^g L^2 T + a \sigma^2 D^2 T .
\end{equation*}
Then, we rewrite $Q(\lambdaP)$ as:
\begin{align*}
 Q(\lambdaP) & =
\left( \frac{1+2\lambdaP}{2\eta} + a \right) \Vert y_t - y_{t+1} \Vert^2
+ 2 a \sigma \Vert x_{t} - y_{t+1}\Vert^2 
        + 2 a \langle x_{t} - y_{t+1}, y_{t+1} - y_t \rangle
\\ & \qquad
- \frac{1}{4 \lambda^g} \Vert x_t - y_{t+1} + \lambdaP(y_t - y_{t+1}) \Vert^2      
\\ & = \left( \frac{1+2\lambdaP}{2\eta} + a - \frac{\lambdaP^2}{4\lambda^g} \right) \Vert y_t - y_{t+1} \Vert^2
+ \left( 2 a \sigma - \frac{1}{4\lambda^g} \right) \Vert x_{t} - y_{t+1}\Vert^2 
\\ & \qquad
        + \left( 2 a + \frac{\lambdaP}{2\lambda^g} \right) \langle x_{t} - y_{t+1}, y_{t+1} - y_t \rangle.
\end{align*}
Thus, we get that $ Q(\lambdaP)$ is a sum of squares if and only if:
\begin{equation*}
\left\{
\begin{aligned}
& \left( 2 a \sigma - \frac{1}{4\lambda^g} \right) \geq 0 \\
& \left( a + \frac{\lambdaP}{4\lambda^g} \right)^2
\leq \left( 2 a \sigma - \frac{1}{4\lambda^g} \right) 
	\left( \frac{1+2\lambdaP}{2\eta} + a - \frac{\lambdaP^2}{4\lambda^g} \right) ,
\end{aligned}\right.
\end{equation*}
where the second line is equivalent to:
\begin{equation}\label{eq:SoS_pol_cond_1}
\Bigl( \frac{a\sigma}{2 \lambda^g} \Bigr) \lambdaP^2 
- \Bigl( \frac{1}{\eta} \bigl(2a\sigma - \frac{1}{4\lambda^g} \bigr) - \frac{a}{2 \lambda^g} \Bigr) \lambdaP 
+ \Bigl( (1-2\sigma)a^2 + \frac{a}{4\lambda^g} 
                -\frac{1}{2 \eta} \bigl(2a\sigma - \frac{1}{4\lambda^g} \bigr) \Bigr) \leq 0 .
\end{equation}
An elementary second order polynomial study indicates
that this polynomial equation has no non-negative solution $\lambdaP \geq 0$
when the linear coefficient is positive, that is, we get the necessary condition:
\begin{equation}\label{eq:SoS_cond_2}
\Bigl( \frac{1}{\eta} \bigl(2a\sigma - \frac{1}{4\lambda^g} \bigr) - \frac{a}{2 \lambda^g} \Bigr) \geq 0 .
\end{equation}
When~\eqref{eq:SoS_cond_2} holds,
we can upper bound the constant coefficient of~\eqref{eq:SoS_pol_cond_1}
by its leading order term $a^2$,
giving us the (slightly) more conservative but easier no analyze condition:
\begin{equation}\label{eq:SoS_pol_cond_2}
\Bigl( \frac{a\sigma}{2 \lambda^g} \Bigr) \lambdaP^2 
- \Bigl( \frac{1}{\eta} \bigl(2a\sigma - \frac{1}{4\lambda^g} \bigr) - \frac{a}{2 \lambda^g} \Bigr) \lambdaP 
+ a^2 \leq 0 .
\end{equation}
Moreover, when~\eqref{eq:SoS_cond_2} holds, we have that~\eqref{eq:SoS_pol_cond_2} has a non-negative solution $\lambdaP\geq0$
if and only if:
\begin{equation*}
\Bigl( \frac{1}{\eta} \bigl(2a\sigma - \frac{1}{4\lambda^g} \bigr) - \frac{a}{2 \lambda^g} \Bigr)^2
- 4 \Bigl( \frac{a\sigma}{2 \lambda^g} \Bigr) a^2 
\geq 0 .
\end{equation*}
Hence, combining all of the above, we get that $Q(\lambdaP)$ is a sum of squares
whenever:
\begin{equation}\label{eq:SoS_final_cond}
\left\{
\begin{aligned}
& \left( 2 a \sigma - \frac{1}{4\lambda^g} \right) \geq 0 \\
& \Bigl( \frac{1}{\eta} \bigl(2a\sigma - \frac{1}{4\lambda^g} \bigr) - \frac{a}{2 \lambda^g} \Bigr) \geq 0 \\
& \Bigl( \frac{1}{\eta} \bigl(2a\sigma - \frac{1}{4\lambda^g} \bigr) - \frac{a}{2 \lambda^g} \Bigr)^2
- 4 \Bigl( \frac{a\sigma}{2 \lambda^g} \Bigr) a^2 
\geq 0.
\end{aligned}
\right.
\end{equation}
Hence, to find the values of $\eta$, $\sigma$, $a$ $\lambda^g$
giving the optimal upper bound, we must solve the problem:
\begin{equation*}
\begin{aligned}
\inf_{\eta,\sigma,a,\lambda^g\geq 0} & \quad
\frac{1}{2\eta} D^2 + \lambda^g L^2 T + a \sigma^2 D^2 T
\\
\text{subject to: } 
& (\eta,\sigma,a,\lambda^g) \text{ satisfy \eqref{eq:SoS_final_cond}}, 
\text{ and } \sigma \leq 1 .
\end{aligned}
\end{equation*}
Solving this non-convex problem algebraically via Lagrange relaxation
(note that only the last constraint in~\eqref{eq:SoS_final_cond}
is saturated at the optimum)
we get that the optimal values are:
\[
    \eta = \frac{D3^{3/4}}{2LT^{3/4}}, \quad \sigma = \frac{\sqrt{3}}{\sqrt{T}}, \quad a = \frac{1}{6\eta},\quad \text{and},\quad  \lambda^g = 2 a \sigma^2 \frac{D^2}{L^2} = \frac{D^2}{\eta T L^2}\, .
\]
Plugging those values into $Q(\lambdaP)$ then allows us to conclude the proof
of Lemma~\ref{lemma_upper_bound_OFW}.
This concludes the design of the optimal proof of Lemma~\ref{lemma_upper_bound_OFW}.

\subsection{Examples of computation times for computing PEPs~\eqref{eq:PEP} and~\eqref{eq:jointstepsizeopt}}\label{a:numerical_computation_times}

\begin{table}[h]
\centering
\begin{tabular}{|c|c|c|c|c|c|c|c|c|c|c|}
\hline
\textbf{T}  & 10 & 15 & 20 & 25 & 30 & 35 & 40 & 45 & 50 & 55  \\ \hline
\textbf{time (s)}  &  0.27 & 0.74 & 1.62 & 3.35 & 6.69 & 14.17 & 22.97 & 40.76 & 62.63 & 94.35  \\ \hline
\hline
\textbf{n}  & 60 & 65 & 70 & 75 & 80 & 85 & 90 & 95 &  &  \\ \hline
\textbf{time (s)}   & 140.88 & 221.85 & 313.78 & 533.60 & 773.32 & 1202.68 & 1603.33 & 2126.25 &  & \\ \hline
\textbf{time (min)}   & 2.35 & 3.70 & 5.23 & 8.89 & 12.89 & 20.04 & 26.72 & 35.44 &  & \\ \hline
\end{tabular}
\caption{Computation time for numerically solving~\eqref{eq:PEP} for different values of $T$ with parameters fixed as in~\eqref{eq:parameter_choice}
(with the appropriate transformation of parameters to go from Algorithm~\ref{OFW_alg_new} to the definition of~\eqref{eq:OFW_general} used in~\eqref{eq:PEP}).}
\label{tab:computation_time_PEP_OFW_fixed_parameters}
\end{table}

\begin{table}[h]
\centering
\begin{tabular}{|c|c|c|c|c|c|c|c|c|c|c|}
\hline
\textbf{T}  & 10 & 15 & 20 & 25 & 30 & 35 & 40 & 45 & 50 & 55  \\ \hline
\textbf{time (s)}  & 0.15 & 0.28 & 0.51 & 0.88 & 1.65 & 2.97 & 4.23 & 6.28 & 11.30 & 15.82  \\ \hline
\hline
\textbf{T}  & 60 & 65 & 70 & 75 & 80 & 85 & 90 & 95 & 100 &  \\ \hline
\textbf{time (s)}  & 20.94 & 31.53 & 42.12 & 55.64 & 63.04 & 98.36 & 109.48 & 
146.80 & 
194.84 & \\ \hline
\end{tabular}
\caption{Computation time for numerically solving~\eqref{eq:jointstepsizeopt} for different values of $T$.}
\label{tab:computation_time_design_PEP_OFW}
\end{table}

In this section, we provide some examples of computation times for numerically solving the PEPs~\eqref{eq:PEP} and~\eqref{eq:jointstepsizeopt} for different values of horizon time $T$.
Table~\ref{tab:computation_time_PEP_OFW_fixed_parameters} presents the computation times for solving the PEP~\eqref{eq:PEP} with parameters fixed as in~\eqref{eq:parameter_choice}
(with the appropriate transformation of parameters to go from Algorithm~\ref{OFW_alg_new} to the definition of~\eqref{eq:OFW_general} used in~\eqref{eq:PEP}).
Table~\ref{tab:computation_time_design_PEP_OFW} presents the computation times for solving the design PEP~\eqref{eq:jointstepsizeopt} which 
optimizes the parameters of the algorithm, thus giving us the optimal worst-case regret guarantee for any OFW-type algorithm of the form~\eqref{eq:OFW_general}.
Note that the computation times for solving~\eqref{eq:jointstepsizeopt} are significantly smaller than those for solving~\eqref{eq:PEP} with fixed parameters,
which appears to be due to the fact that, as the design PEP~\eqref{eq:jointstepsizeopt} jointly optimizes the algorithm parameters and its regret bound proof,
it can find more structured proofs that allow for more efficient numerical solving of the underlying semidefinite program.

\subsection{Examples of numerically optimized stepsize patterns}\label{a:numerical_stepsizeopt}

The following list provides numerical examples of optimal parameter values for the online Frank--Wolfe-type algorithm~\eqref{eq:OFW_general}
for $T= 2,\ldots,6$, together with their optimal worst-case regret guarantees.
(Note that for $T=1$ the algorithm is necessarily trivial as we can only play the uninformed choice~$x_1$.)
Those values were obtained by numerically solving the relaxed linear convex semidefinite program~\eqref{eq:PEP_relaxed_joint_opt}
with $L=D=1$.
Note that we chose the parameters given by~\eqref{eq:param_solve_change_vars} and thus we do not report the values
of $\eta_{t,s} = 1$ for $1\leq s \leq t \leq T-1$.
We present the values of the parameters $\{ \gamma_{t,s} \}_{1\leq s < t \leq T}$
and $\{ \beta_{t,s} \}_{1\leq s < t \leq T-1}$
as square $T\times T$ and $(T-1)\times(T-1)$ matrices, respectively,
where the values for out of range indices are left blank intentionally.

\begin{itemize}
\item
For $T=2$, we have $R_T \leq B_2 \leq 1.7321$ and:
\begin{equation*}
[\gamma_{t,s}] = 
\begin{bmatrix}
~ & ~ \\
0.5 & ~~
\end{bmatrix},
\quad
[\beta_{t,s}] = 
\begin{bmatrix}
~~  
\end{bmatrix} .
\end{equation*}

\item
For $T=3$, we have $R_T \leq B_3 \leq 2.3421$ and:
\begin{equation*}
[\gamma_{t,s}] = 
\begin{bmatrix}
~ & ~ & ~ \\
0.5 & ~ & ~ \\
0.3118 & 0.3764 & ~~~~~ \\
\end{bmatrix},
\quad
[\beta_{t,s}] = 
\begin{bmatrix}
~ & ~  \\
-0.1099 & ~~~~~ 
\end{bmatrix} .
\end{equation*}

\item
For $T=4$, we have $R_T \leq B_4 \leq 2.9029$ and:
\begin{equation*}
[\gamma_{t,s}] = 
\begin{bmatrix}
~ & ~ & ~ & ~ \\
0.5 & ~ & ~ & ~ \\
0.3133 & 0.3734 & ~ & ~ \\
0.1843 & 0.2197 & 0.4116 & ~~~~~ \\
\end{bmatrix},
\quad
[\beta_{t,s}] = 
\begin{bmatrix}
~ & ~  & ~  \\
0.1961 & ~  & ~  \\
-0.4249 & 0.1465  & ~~~~~ 
\end{bmatrix} .
\end{equation*}

\item
For $T=5$, we have $R_T \leq B_5 \leq 3.4217$ and:
\begin{align*}
& [\gamma_{t,s}] = 
\begin{bmatrix}
~ & ~ & ~ & ~ & ~ \\
0.5 & ~ & ~ & ~ & ~ \\
0.3067 & 0.3866 & ~ & ~ & ~ \\
0.2124 & 0.2677 & 0.3075 & ~ & ~ \\
0.1201 & 0.1514 & 0.1739 & 0.4345 & ~~~~~ \\
\end{bmatrix}, \\
& [\beta_{t,s}] = 
\begin{bmatrix}
~ & ~  & ~ & ~  \\
0.5649 & ~  & ~ & ~  \\
-0.2595 & 0.212  & ~ & ~  \\
-0.6282 & -0.253  & 0.3473 & ~~~~~ 
\end{bmatrix} .
\end{align*}

\item
For $T=6$, we have $R_T \leq B_6 \leq 3.917$ and:
\begin{align*}
& [\gamma_{t,s}] = 
\begin{bmatrix}
~ & ~ & ~ & ~ & ~ & ~ \\
0.5 & ~ & ~ & ~ & ~ & ~ \\
0.3068 & 0.3863 & ~ & ~ & ~ & ~ \\
0.2101 & 0.2646 & 0.3151 & ~ & ~ & ~ \\
0.1406 & 0.177 & 0.2108 & 0.3309 & ~ & ~ \\
0.0784 & 0.0985 & 0.1174 & 0.1842 & 0.4432 & ~~~~~ \\
\end{bmatrix}, \\
& [\beta_{t,s}] = 
\begin{bmatrix}
~ & ~  & ~ & ~ & ~  \\
0.5856 & ~  & ~ & ~ & ~  \\
-0.0481 & 0.4808  & ~ & ~ & ~  \\
-0.5053 & -0.0949  & 0.3922 & ~ & ~  \\
-0.7675 & -0.4251  & -0.001 & 0.3876 & ~~~~~ 
\end{bmatrix} .
\end{align*}

\end{itemize}

Note that the structure of the values of $\{ \gamma_{t,s} \}_{1\leq s < t \leq T}$
implies that we get the update rule $x_{t+1} = \sigma_t v_t + (1-\sigma_t) x_t$
with $\sigma_t = \gamma_{t+1,t}$ in~\eqref{eq:OFW_general}.

\subsection{Joint stepsize optimization with multiple linear optimization rounds per iteration}\label{a:numerical_multiple_stepsizeopt}

As explained in Section~\ref{s:numerics}, the joint stepsize optimization
method for online Frank--Wolfe-type algorithms explained in Section~\ref{s:opt_param}
and in Appendix~\ref{a:stepsizeopt}
can be easily adapted to the variant algorithm 
with multiple linear optimization rounds per iteration defined in~\eqref{eq:OFW_multiple}.
In this section, we explain the details on how to make this adaptation.

We first define the worst-case regret for~\eqref{eq:OFW_multiple}
with given parameters by simply replacing~\eqref{eq:OFW_general}
by~\eqref{eq:OFW_multiple} in~\eqref{eq:PEP}:
\begin{equation*}\label{eq:PEP_multiple}
\begin{aligned}
\tilde B_T(\{(\eta_{t,k,s},\beta_{t,k,s,j},\gamma_{t,s,j})\}_{t,k,s,j}) \triangleq
\sup_{\substack{\cK, \{\ell_t\}_{t=1,\ldots,T}\\ x_\star, \{x_t\}_{t=1,\ldots,T}\\d\in\mathbb{N}}} \,
&R_T(x_1,\ldots,x_T; x_\star)\\
\text{subject to: } 
&  \ell_t \text{ is convex and $L$-Lipschitz for }t=1,\ldots,T,\\
& \cK \text{ is a non-empty closed convex set of $\mathbb{R}^d$,}\\ 
&\mathrm{Diam}(\cK)\leq D,\\
& \{x_t\}_{t=1,\ldots,T} \text{ is generated by~\eqref{eq:OFW_multiple}}.\\
\end{aligned}
\end{equation*}
To adapt the joint minimization problem~\eqref{eq:jointstepsizeopt} to~\eqref{eq:OFW_multiple},
we need to replace $W_T\bigl(\{\eta_{t,s},\beta_{t,s},\gamma_{t,s}\}_{t,s}\bigr)$
from Section~\ref{s:numerics} by some relaxation of the problem above.
More precisely, this relaxation is obtained through: 
(i)~we observe that all $x_t$ for $t=2, \cdots, T$ are in the convex hull of $x_1, v_{1,1}, \cdots, v_{1,r}, \cdots, v_{T-1,r}$, and thus the domain constraints for $\cK$ are imposed only on vectors $x_1, v_{1,1}, \cdots, v_{T-1,r}, x_\star$; 
(ii)~we only keep the boundary constraints  corresponding to the optimality of $v_{t,k}$ compared with $v_{t,k+1},\ldots, v_{t,r}, v_{t+1,1},\ldots,v_{T-1,r},x_\star$:
\begin{multline}
\label{eq:PEP_relaxed_multiple}
\tilde W_T(\{(\eta_{t,k,s},\beta_{t,k,s,j},\gamma_{t,s,j})\}_{t,k,s,j}) 
\triangleq \\
\begin{aligned}
\sup_{\substack{
	\{g_t\}_{t=1,\ldots,T},\, x_\star\\ 
	\{(x_t, v_{t,k}, \mathrm{dir}_{t,k})\}_{t=1,\ldots,T, k=1,\ldots,r}\\
	d\in\mathbb{N}}} \,
&\sum_{t=1}^T \langle g_t, x_t - x_\star \rangle \\
\text{subject to: } 
& \{(x_t, \mathrm{dir}_{t,1}, \dots, \mathrm{dir}_{t,r})\}_{t=1,\ldots,T} \text{ compatible with~\eqref{eq:OFW_multiple}},\\
& \langle -\mathrm{dir}_{t,k}, u - v_{t,k} \rangle \leq 0 \text{ for all }t=1,\ldots,T-1 
	\text{ and } k=1,\ldots,r \\
& \qquad\qquad\quad 	
	\text{ and } u\in\{v_{t,k+1},,\cdots, v_{t,r}, v_{t+1,1}, \cdots, v_{T-1,r}, x_\star\}, \\ 
&\mathrm{Diam}(\{ x_1, v_{1,1}, \cdots, v_{1,r}, \cdots, v_{T-1,r}, x_\star \})\leq D,\\
&  \Vert g_t \Vert \leq L \text{ for }t=1,\ldots,T,\\
\end{aligned}
\end{multline}
where the line ``\emph{$\{(x_t, \mathrm{dir}_{t,1}, \dots, \mathrm{dir}_{t,r})\}_{t=1,\ldots,T} \text{ compatible with~\eqref{eq:OFW_multiple}}$}'' 
means that $x_t$ and $\mathrm{dir}_{t,k}$ can be substituted by their expressions in~\eqref{eq:OFW_multiple},
leading to $\tilde W_T\bigl(\{(\eta_{t,k,s},\beta_{t,k,s,j},\gamma_{t,s,j})\}_{t,k,s,j}\bigr)
\geq \tilde B_T\bigl(\{(\eta_{t,k,s},\beta_{t,k,s,j},\gamma_{t,s,j})\}_{t,k,s,j}\bigr)$.

We can now formulate the adaptation of the joint minimization 
problem~\eqref{eq:jointstepsizeopt} to~\eqref{eq:OFW_multiple},
giving us:
\begin{equation}\label{eq:jointstepsizeopt_multiple}     \min_{\{(\eta_{t,k,s},\beta_{t,k,s,j},\gamma_{t,s,j})\}_{t,k,s,j}} 
     \left\{ \! \tilde  W_T(\{(\eta_{t,k,s},\beta_{t,k,s,j},\gamma_{t,s,j})\}_{t,k,s,j}) 
     \text{ s.t.} \sum_{s=1}^{t-1}\sum_{j=1}^{r}\gamma_{t,s,j}\leq 1,\, \gamma_{t,s,j}\geq 0 \! \right\} \! .
\end{equation}

\providecommand{\lex}{\mathrm{lex}}

In the remaining of this section, we will need the following notations.
For integers $t,s\in\mathbb{N}$ and $k,j \in \llbracket 1, r \rrbracket$,
we write $(t,k) \leq_{\lex} (s,j)$ if $t<s$, or $t=s$ and $k<j$.
We write $(t,k) <_{\lex} (s,j)$ if $(t,k) \leq_{\lex} (s,j)$ and $(t,k) \neq (s,j)$.

Then, the method detailed in Appendix~\ref{a:pep} and~\ref{a:stepsizeopt}
can be immediately adapted to reformulate~\eqref{eq:PEP_relaxed_multiple}
as a linear convex semidefinite program and then form its Lagrange dual,
giving us:
\begin{equation}\label{eq:PEP_relaxed_SDP_dual_multiple}
\begin{aligned}
&
\begin{aligned}
\tilde W_T(\{(\eta_{t,k,s},\beta_{t,k,s,j},\gamma_{t,s,j})\}_{t,k,s,j}) = \inf_{\substack{ \lambdaLip\geq 0 \\ \lambdaDiam\geq 0 \\ \lambdaBound \geq 0 }} \,
&  \sum_{t=1}^T \lambdaLip_t\, L^2 
+ \frac{1}{2} \sum_{u,\, v \in 
	\{ \barVect{x}_1, \barVect{v}_{1,1}, \cdots, \barVect{v}_{T-1,r}, \barVect{x}_\star \}}
		\!\!\!\! \lambdaDiam_{\{u,v\}}\, D^2
\\
\text{subject to: } 
& S(\eta,\beta,\gamma; \lambda) \succeq 0, \\
\end{aligned}\\
&
\begin{aligned}
 \text{ where }  S(\eta,\beta,\gamma; \lambda) = 
	& \ \sum_{t=1}^T \lambdaLip_t\, (\barVect{g}_t \odot \barVect{g}_t) 
	+ \frac{1}{2} \sum_{u,\, v \in 
	\{ \barVect{x}_1, \barVect{v}_{1,1}, \cdots, \barVect{v}_{T-1,r}, \barVect{x}_\star \}}
		\!\!\!\! \lambdaDiam_{\{u,v\}}\,  ((u-v) \odot (u-v))   \\
&  
+ \sum_{t=1}^{T-1} \sum_{k=1}^{r} 
	\sum_{u\in\{\barVect{v}_{t,k+1},\cdots, \barVect{v}_{T-1,r}, \barVect{x}_\star \}} 
	\!\!\!\!\!\!
	\lambdaBound_{{\barVect{v}_{t,k}}, u}  \,
		(\barVect{\mathrm{dir}}_{t,k} \odot (\barVect{v}_{t,k} - u) ) \\
&  - \sum_{t=1}^T \sum_{s=1}^{t-1} \sum_{k=1}^r \gamma_{t,s,k} \, (\barVect{g}_t \odot \barVect{v}_{s,k})
			+  \sum_{t=1}^T (\barVect{g}_t \odot \barVect{x}_\star) , \\
\end{aligned}\\
&
\text{ and }
\lambda = (\lambdaLip, \lambdaDiam, \lambdaBound) \text{ and }
\barVect{\mathrm{dir}}_{t,k} = \sum_{s=1}^t \eta_{t,k,s}\, \barVect{g}_s 
			+ \!\!\!\!\!	 \sum_{(s,j) <_{\lex} (t,k)} \!\!\!\!\!\!\! \beta_{t,k,s,j}\, \barVect{v}_{s,j}
	\text{ for } t=1,\dots,T-1 .
\end{aligned}
\end{equation}

Substituting~\eqref{eq:PEP_relaxed_SDP_dual_multiple}
in~\eqref{eq:jointstepsizeopt_multiple},
we get a linear optimization problem
with a bilinear matrix inequality constraint,
which we reformulate as a linear convex semidefinite program
using a change of variables similar to that of~\eqref{eq:SDP_opt_param_change_vars}:
$(\eta, \beta, \lambdaBound)$ are replaced by the variables
(with the convention $\barVect{v}_{T,1} = \barVect{x}_\star$):
\begin{equation}\label{eq:SDP_opt_param_change_vars_multiple}
\begin{aligned}
& B_{t,k,s} 
= \eta_{t,k,s} \sum_{(t,k) <_{\lex} (m,j) \leq_{\lex} (T,1)} 
	\lambdaBound_{\barVect{v}_{t,k}, \barVect{v}_{m,j}}
- \sum_{(s,1) \leq_{\lex} (m,i) <_{\lex} (t,k)}
	\eta_{m,i,s}\, \lambdaBound_{\barVect{v}_{m,i}, \barVect{v}_{t,k}}
\\
& \qquad \forall (1,1) \leq_{\lex} (s,j) \leq_{\lex} (t,k) \leq_{\lex} (T,1),
\\
& C_{t,k,s,j} 
= \beta_{t,k,s,j} \sum_{(t,k) <_{\lex} (m,i) \leq_{\lex} (T,1)} 
	\lambdaBound_{\barVect{v}_{t,k}, \barVect{v}_{m,i}}
- \sum_{(s,j) \leq_{\lex} (m,i) <_{\lex} (t,k)}
	\beta_{m,i,s,j}\,  \lambdaBound_{\barVect{v}_{m,i}, \barVect{v}_{t,k}}
\\
& \qquad \forall (1,1) \leq_{\lex} (s,j) <_{\lex} (t,k) \leq_{\lex} (T,1).
\end{aligned}
\end{equation}

Finally, this change of variables gives us the following linear convex semidefinite program which is the adaptation of~\eqref{eq:PEP_relaxed_joint_opt} for~\eqref{eq:OFW_multiple}:
\begin{equation*}\label{eq:PEP_relaxed_joint_opt_multiple}
\begin{aligned}
&
\begin{aligned}
\inf_{\{(B_{t,s,k},C_{t,s,k},\gamma_{t,s,k})\}_{t,s,k}}
\inf_{\substack{ \lambdaLip\geq 0 \\ \lambdaDiam\geq 0 \\ \lambdaBound \geq 0 }} \,
&  \sum_{t=1}^T \lambdaLip_t\, L^2 
+ \frac{1}{2} \sum_{u,\, v \in 
	\{ \barVect{x}_1, \barVect{v}_{1,1}, \cdots, \barVect{v}_{T-1,r}, \barVect{x}_\star \}}
		\lambdaDiam_{\{u,v\}}\, D^2
\\
\text{subject to: } 
& S(B,C,\gamma; \lambda) \succeq 0, \\
& \sum_{ (s,1) \leq_{\lex} (t,k) \leq_{\lex} (T,1) } B_{t,k,s} = 0 
	\text{ for } s=1,\ldots,T,
\\ & 	\sum_{ (s,j) <_{\lex} (t,k) \leq_{\lex} (T,1) } C_{t,k,s,j} = 0 
\text{ for } 
(1,1) \leq_{\lex} (s,j) <_{\lex} (T,1),
\\
\end{aligned}\\
&
\begin{aligned}
 \text{ where }  S(B,C,\gamma; \lambda) = 
	& \ \frac{1}{2} \sum_{u,\, v \in 
	\{ \barVect{x}_1, \barVect{v}_{1,1}, \cdots, \barVect{v}_{T-1,r}, \barVect{x}_\star \}}
		 \lambdaDiam_{\{u,v\}}\,  ((u-v) \odot (u-v))  + \sum_{t=1}^T \lambdaLip_t\, (\barVect{g}_t \odot \barVect{g}_t)  \\
& + \sum_{\substack{
			(1,1) \leq_{\lex} (t,k) \leq_{\lex} (T,1) \\
			1 \leq s \leq t
		}}
	B_{t,s} \, (\barVect{g}_s \odot \barVect{v}_{t,k})  + \sum_{\substack{
			(1,1) \leq_{\lex} (t,k) \leq_{\lex} (T,1) \\
			(1,1) \leq_{\lex} (s,j) <_{\lex} (t,k)
		}}
		C_{t,s} \, (\barVect{v}_{s,j} \odot \barVect{v}_{t,k}) \\
&  - \sum_{t=1}^T \sum_{s=1}^{t-1} \sum_{k=1}^r \gamma_{t,s,k} \, (\barVect{g}_t \odot \barVect{v}_{s,k})
			+  \sum_{t=1}^T (\barVect{g}_t \odot \barVect{x}_\star), \\
\end{aligned}\\
&
\text{ and }
\lambda = (\lambdaLip, \lambdaDiam, \lambdaBound) .
\end{aligned}
\end{equation*}
Note that, as in Appendix~\ref{a:stepsizeopt}, 
this change of variables is not one-to-one:
for given values of $B$ and $C$ there exist several possible values of
$(\eta,\beta,\lambdaBound)$ satisfying~\eqref{eq:SDP_opt_param_change_vars_multiple}.
Nevertheless, as in Appendix~\ref{a:stepsizeopt},
imposing $\eta_{t,k,s}=1$ for all $t,k,s$,
we get a unique solution by solving the part of~\eqref{eq:SDP_opt_param_change_vars_multiple}
for $\lambdaBound$, and then solving the second part of~\eqref{eq:SDP_opt_param_change_vars_multiple}
for~$\beta$.

\subsection{Extra PEP numerical plots: log-log plots for regret rate}
\label{a:log_log_plots}

In this section, we present the log-log plots version of the numerical experiments of Section~\ref{s:numerics}. 
Recall that those numerical experiments on worst-case regret of different variations of online Frank--Wolfe-type algorithms were obtained by leveraging the semidefinite programming techniques of Section~\ref{sec:PEP_for_OFW},
which allow to compute exact worst-case regret values.
For convenience, we first present again Figure~\ref{fig:Comparison_tau} with linear scales.
In Figure~\ref{fig:Comparison_tau_log_log}, we present the same numerical experiments as in Figure~\ref{fig:Comparison_tau} (same computed values) but as log-log plots.
Note that the slopes of curves in those log-log plots correspond to the exponents of the rates from the regret upper bounds in Figure~\ref{fig:Comparison_tau}.

In Figure~\ref{fig:Comparison_tau_slopes}, we present an interpolation of those regret rate exponent for the different computed exact worst-case regret bounds from the numerical experiments of Figure~\ref{fig:Comparison_tau}.
(The theoretical upper bounds are in closed forms and their exponent of $3/4$ can easily be read.)
To interpolate those regret rate exponent we interpolate the slopes from the log-log plots.
That is, for a sequence of tight regret upper bounds $B_T$ for different horizon values $T\in\{T_1, \cdots, T_n\}$,
we select some value $T_{\text{ref}} = T_{\lfloor n/2 \rfloor}$, and then for all values of
$T\in\{T_1, \cdots, T_n\} \setminus \{ T_{\text{ref}}  \}$ we compute:
\begin{equation}\label{eq:def_regret_rate_interpolation}
( \log( B_T ) - \log(B_{T_{\text{ref}}} ) ) / (\log( T ) - \log( T_{\text{ref}} ) ) . 
\end{equation}
We also include in the top right plot of Figure~\ref{fig:Comparison_tau_slopes} (the plot dedicated to the anytime variants of OFW)
the bound~\eqref{eq:bound_anytime_OFW_with_sums} from the proof of Theorem~\ref{thm_upper_bound_OFW_anytime} 
to compare exponent convergence speed of the bound for anytime tunings of OFW 
to the exponent convergence speed for the sum of terms in $s^{-1/4}$ in the bound~\eqref{eq:bound_anytime_OFW_with_sums}.

\setlength{\figWidth}{6.6cm}
\setlength{\figHeight}{5.8cm}
\renewcommand{\legendDist}{0.25}

\begin{figure}[p]
\begin{subfigure}{\textwidth}
    \centering

        \hfill
    \begin{minipage}[t]{0.48\textwidth}
        \vspace*{0pt}
        \centering
\begin{tikzpicture}

\definecolor{crimson2143940}{RGB}{214,39,40}
\definecolor{forestgreen4416044}{RGB}{44,160,44}
\definecolor{mediumpurple148103189}{RGB}{148,103,189}
\definecolor{darkgray176}{RGB}{176,176,176}
\definecolor{darkorange25512714}{RGB}{255,127,14}
\definecolor{lightgray204}{RGB}{204,204,204}
\definecolor{steelblue31119180}{RGB}{31,119,180}

\begin{axis}[
legend cell align={left},
legend style={
  fill opacity=1,
  draw opacity=1,
  text opacity=1,
  at={(0.5,-\legendDist)},
  anchor=north,
  draw=none
},
tick align=outside,
tick pos=left,
x grid style={darkgray176},
xlabel={Time horizon $T$},
xmin=-5, xmax=105,
xtick style={color=black},
y grid style={darkgray176},
ylabel={Worst-case regret},
ymin=-2, ymax=55,
ytick style={color=black},
  width=\figWidth,
  height=\figHeight,
  scale only axis,
 font=\small
]
\addplot [thick, crimson2143940, mark=asterisk, mark size=1, mark options={solid}]
table {%
1 8
2 13.4543426440594
3 18.2360564556382
4 22.6274169979695
5 26.7496121990569
6 30.6692690038211
7 34.4281365652708
8 38.0546276800871
10 44.9873060152279
15 60.9759297785538
20 75.6593287202541
25 89.4427190999916
30 102.548881535096
35 115.117412732146
40 127.243316602728
45 138.995062234592
50 150.424123723456
55 161.570475875318
60 172.46597374228
65 183.136539862808
70 193.603639399487
75 203.885309381655
80 213.996897592455
90 233.760899142071
100 252.98221281347
};
\addlegendentry{{\legendEntryFontSize \cite[Algo.~27]{hazanIntroductionOnlineConvex2016}: bound~\cite[Theorem 7.3]{hazanIntroductionOnlineConvex2016}}}
\addplot [thick, crimson2143940, dashed, mark=asterisk, mark size=1, mark options={solid}]
table {%
1 0.999999978829868
2 1.86602540835967
3 2.86602530022636
4 3.83521194338494
5 4.84566679047157
6 5.77847121423904
7 6.64296540269502
8 7.47558796478797
10 9.06556940344561
15 12.7147146818818
20 16.0829251259494
25 19.2687001005565
30 22.3084311637338
35 25.2313421308435
40 28.0532078096698
45 30.7885720243218
50 33.4461650859791
55 36.0337573951143
60 38.5568065885238
65 41.02483254
70 43.44430215
75 45.8196833
80 48.15440388
90 52.7148001953334
100 57.147182191147
};
\addlegendentry{{\legendEntryFontSize \cite[Algo.~27]{hazanIntroductionOnlineConvex2016}: tight bound $B_T$ from~(\ref{eq:PEP})}}
\addplot [semithick, forestgreen4416044, mark size=3, mark options={solid}]
table {%
1 1.75476535060332
2 2.95115178586752
3 4
4 4.9632259152112
5 5.86741157862262
6 6.72717132202972
7 7.55166264132207
8 8.34711776039087
10 9.8677707265638
15 13.3748060995284
20 16.5955460610261
25 19.6188730425514
30 22.493653007614
35 25.2505058891838
40 27.9102703837895
45 30.4879648892769
50 32.9948800255984
55 35.4397840933123
60 37.829664360127
65 40.17020682258
70 42.466119771115
75 44.7213595499958
80 46.939292628981
90 51.2744407675481
100 55.4905526705042
};
\addlegendentry{{\legendEntryFontSize Algo.~\ref{OFW_alg_new}: bound from Theorem~\ref{thm_upper_bound_OFW}}}
\addplot [semithick, forestgreen4416044, dashed, mark=x, mark size=2, mark options={solid}]
table {%
1 0.99999997622009
2 1.86602536857043
3 2.79508504207169
4 3.41695016277647
5 4.00578768309647
6 4.56827781991907
7 5.10944325015655
8 5.63752492727876
10 6.66119581166219
15 9.02755847964825
20 11.1903480646082
25 13.2176738631319
30 15.1477451800894
35 17.0005288149411
40 18.7895976211105
45 20.5247895195582
50 22.2138706190613
55 23.8627054385189
60 25.4759388445294
65 27.057093360461
70 28.6091703686733
75 30.1349387581324
80 31.6366960286442
85 33.1162504419378
90 34.5752635225472
};
\addlegendentry{{\legendEntryFontSize Algo.~\ref{OFW_alg_new}: tight bound $B_T$ from~(\ref{eq:PEP})}}
\addplot [semithick, mediumpurple148103189, dashed, mark=o, mark size=2, mark options={solid}]
table {%
1 0.999999999986411
2 1.73205077164806
3 2.34209616885784
4 2.90282823623805
5 3.42168738269308
6 3.91697957391197
7 4.38834045416619
8 4.84470501303823
9 5.28476977360727
10 5.71400805695854
11 6.13109156852881
12 6.53978673154998
13 6.93890009089812
14 7.33121252392302
15 7.71568171523579
20 9.55298861900087
25 11.2764639001485
30 12.91523643836
35 14.4863179235929
40 16.0021461745957
45 17.4709535774897
50 18.8995128351095
55 20.2925817455325
60 21.6543713495292
65 22.9879410641388
70 24.2961479561498
75 25.5810970007411
80 26.8448520074568
85 28.0889320029728
90 29.31489412696
95 30.5238801404065
100 31.7171067431816
};
\addlegendentry{{\legendEntryFontSize Optimized algo.~and bound from~(\ref{eq:jointstepsizeopt})}}
\end{axis}

\end{tikzpicture}
    \end{minipage}
    \hfill
    \begin{minipage}[t]{0.48\textwidth}
        \vspace*{0pt}
        \centering
\begin{tikzpicture}

\definecolor{crimson2143940}{RGB}{214,39,40}
\definecolor{darkgray176}{RGB}{176,176,176}
\definecolor{darkturquoise23190207}{RGB}{23,190,207}
\definecolor{forestgreen4416044}{RGB}{44,160,44}
\definecolor{goldenrod18818934}{RGB}{188,189,34}
\definecolor{lightgray204}{RGB}{204,204,204}
\definecolor{steelblue31119180}{RGB}{31,119,180}

\begin{axis}[
legend cell align={left},
legend style={
  fill opacity=1,
  draw opacity=1,
  text opacity=1,
  at={(0.5,-\legendDist)},
  anchor=north,
  draw=none
},
tick align=outside,
tick pos=left,
x grid style={darkgray176},
xlabel={Time horizon $T$},
xmin=-5, xmax=105,
xtick style={color=black},
y grid style={darkgray176},
ymin=-2, ymax=55,
ytick style={color=black},
  width=\figWidth,
  height=\figHeight,
  scale only axis,
 font=\small
]
\addplot [semithick, forestgreen4416044, mark=Mercedes star, mark size=3, mark options={solid}]
table {%
1 2.193456688254154
2 3.688939732334405
3 5.0
4 6.204032394013997
5 7.334264473278278
6 8.408964152537145
7 9.439578301652581
8 10.433897200488582
10 12.334713408204754
15 16.71850762441055
20 20.74443257628261
25 24.523591303189267
30 28.117066259517458
35 31.56313236147981
40 34.88783797973685
45 38.109956111596105
50 41.24360003199805
55 44.299730116640355
60 47.28708045015879
65 50.212758528225045
70 53.08264971389377
75 55.90169943749474
80 58.674115786226224
85 61.40351851407624
90 64.0930509594351
95 66.74546564983403
100 69.3631908381303
};
\addlegendentry{{\legendEntryFontSize Anytime Algo.~\ref{OFW_alg_new_anytime}: bound from Theorem~\ref{thm_upper_bound_OFW_anytime}}}
\addplot [semithick, crimson2143940, dashed, mark=asterisk, mark size=1, mark options={solid}]
table {%
1 0.999999978829868
2 1.86602540835967
3 2.86602530022636
4 3.83521194338494
5 4.84566679047157
6 5.77847121423904
7 6.64296540269502
8 7.47558796478797
10 9.06556940344561
15 12.7147146818818
20 16.0829251259494
25 19.2687001005565
30 22.3084311637338
35 25.2313421308435
40 28.0532078096698
45 30.7885720243218
50 33.4461650859791
55 36.0337573951143
60 38.5568065885238
65 41.02483254
70 43.44430215
75 45.8196833
80 48.15440388
90 52.7148001953334
100 57.147182191147
};
\addlegendentry{{\legendEntryFontSize \cite[Algo.~27]{hazanIntroductionOnlineConvex2016}: tight bound $B_T$ from~(\ref{eq:PEP})}}
\addplot [thick, crimson2143940, dashed, mark=+, mark size=2, mark options={solid}]
table {%
1 0.999999978829868
2 1.86602540820839
3 2.86602528722214
4 3.81941917172879
5 4.7843884768231
6 5.66753177316362
7 6.50107951868022
8 7.31139133654082
9 8.09325316957615
10 8.85139312035817
11 9.58774550759338
12 10.3054843155196
13 11.0069165395407
14 11.6938619340762
15 12.3679052961251
20 15.5818013204948
25 18.5962663625794
30 21.4601077826225
35 24.2075625818537
40 26.8596151684412
45 29.4313051952187
50 31.9318777341248
55 34.3709642044624
60 36.7548638838124
65 39.0892721371677
70 41.3798427025696
75 43.63062719759103
};
\addlegendentry{{\legendEntryFontSize Anytime \cite[Algo.~27]{hazanIntroductionOnlineConvex2016}: tight bound from~(\ref{eq:PEP})}}
\addplot [semithick, forestgreen4416044, dashed, mark=square, mark size=2, mark options={solid}]
table {%
1 0.99999997622009
2 1.86602539808008
3 2.70563299799477
4 3.58153879868095
5 4.36926755755856
6 5.10221940117672
7 5.79270161250749
8 6.4526213576502
10 7.70024853201964
15 10.5307220687907
20 13.0925557276164
25 15.4754138829171
30 17.7332251672391
35 19.895717259638
40 21.9816644090207
45 24.0038401515299
50 25.9713602247288
55 27.8912843690133
60 29.7692345570552
65 31.6097702772639
70 33.4164604156972
75 35.1922403201557
80 36.9396265985578
85 38.6611721882096
90 40.3588001390253
};
\addlegendentry{{\legendEntryFontSize Anytime Algo.~\ref{OFW_alg_new_anytime}: tight bound $B_T$ from~(\ref{eq:PEP})}}
\addplot [semithick, forestgreen4416044, dashed, mark=x, mark size=2, mark options={solid}]
table {%
1 0.99999997622009
2 1.86602536857043
3 2.79508504207169
4 3.41695016277647
5 4.00578768309647
6 4.56827781991907
7 5.10944325015655
8 5.63752492727876
10 6.66119581166219
15 9.02755847964825
20 11.1903480646082
25 13.2176738631319
30 15.1477451800894
35 17.0005288149411
40 18.7895976211105
45 20.5247895195582
50 22.2138706190613
55 23.8627054385189
60 25.4759388445294
65 27.057093360461
70 28.6091703686733
75 30.1349387581324
80 31.6366960286442
85 33.1162504419378
90 34.5752635225472
};
\addlegendentry{{\legendEntryFontSize Algo.~\ref{OFW_alg_new}: tight bound $B_T$ from~(\ref{eq:PEP})}}
\end{axis}

\end{tikzpicture}
    \end{minipage}
    \hfill
    
\medskip

    \hfill
        \begin{minipage}[t]{0.48\textwidth}
        \vspace*{0pt}
        \centering
\begin{tikzpicture}

\definecolor{darkgray176}{RGB}{176,176,176}
\definecolor{lightgray204}{RGB}{204,204,204}
\definecolor{mediumpurple148103189}{RGB}{148,103,189}
\definecolor{orchid227119194}{RGB}{227,119,194}
\definecolor{sienna1408675}{RGB}{140,86,75}

\begin{axis}[
legend cell align={left},
legend style={
  fill opacity=1,
  draw opacity=1,
  text opacity=1,
  at={(0.5,-\legendDist)},
  anchor=north,
  draw=none
},
tick align=outside,
tick pos=left,
x grid style={darkgray176},
xlabel={Time horizon $T$},
xmin=-5, xmax=105,
xtick style={color=black},
y grid style={darkgray176},
ylabel={Worst-case regret},
ymin=-2, ymax=55,
ytick style={color=black},
  width=\figWidth,
  height=\figHeight,
  scale only axis,
 font=\small
]
\addplot [semithick, mediumpurple148103189, dashed, mark=o, mark size=2, mark options={solid}]
table {%
1 0.999999999986411
2 1.73205077164806
3 2.34209616885784
4 2.90282823623805
5 3.42168738269308
6 3.91697957391197
7 4.38834045416619
8 4.84470501303823
9 5.28476977360727
10 5.71400805695854
11 6.13109156852881
12 6.53978673154998
13 6.93890009089812
14 7.33121252392302
15 7.71568171523579
20 9.55298861900087
25 11.2764639001485
30 12.91523643836
35 14.4863179235929
40 16.0021461745957
45 17.4709535774897
50 18.8995128351095
55 20.2925817455325
60 21.6543713495292
65 22.9879410641388
70 24.2961479561498
75 25.5810970007411
80 26.8448520074568
85 28.0889320029728
90 29.31489412696
95 30.5238801404065
100 31.7171067431816
};
\addlegendentry{{\legendEntryFontSize Optimized algo.~(\ref{eq:OFW_multiple}), $r=1$,  from~(\ref{eq:jointstepsizeopt})}}
\addplot [semithick, sienna1408675, dashed, mark=triangle, mark size=2, mark options={solid}]
table {%
1 0.999999999986411
2 1.63299311979155
3 2.15168457507838
4 2.62509003041537
5 3.06278519425511
6 3.4795301018729
7 3.8762878127482
8 4.25983704757672
9 4.62989237992315
10 4.99045867833446
11 5.3410098087897
12 5.68423150190584
13 6.01957981552844
14 6.34900120418311
15 6.6719871075061
20 8.21488294667517
25 9.66223380277867
30 11.0384070754211
35 12.3578609109691
40 13.6309257863288
45 14.8646107403946
50 16.0645089825127
55 17.2346879049347
60 18.378617200216
65 19.4989121739457
70 20.597904053024
75 21.6774336420247
80 22.7391637690755
};
\addlegendentry{{\legendEntryFontSize Optimized algo.~(\ref{eq:OFW_multiple}), $r=2$, from~(\ref{eq:jointstepsizeopt})}}
\addplot [semithick, orchid227119194, dashed, mark=square, mark size=2, mark options={solid}]
table {%
1 0.999999999986411
2 1.5837637868675
3 2.05505753723153
4 2.48625209927649
5 2.88336510543731
6 3.26149038547669
7 3.62091546948163
8 3.96831411446438
9 4.30323394828769
10 4.62950562985113
11 4.94658359011182
12 5.25698225939425
13 5.560188906985
14 5.85799574095499
15 6.14994478087934
20 7.54419607303208
25 8.85176336291036
30 10.0948492043156
35 11.2866164238039
40 12.4364369256464
45 13.5506663055645
50 14.6343709611955
};
\addlegendentry{{\legendEntryFontSize Optimized algo.~(\ref{eq:OFW_multiple}), $r=3$, from~(\ref{eq:jointstepsizeopt})}}
\end{axis}

\end{tikzpicture}
    \end{minipage}
    \hfill
    \begin{minipage}[t]{0.48\textwidth}
        \vspace*{0pt}
        \centering
\begin{tikzpicture}

\definecolor{darkgray176}{RGB}{176,176,176}
\definecolor{gray127}{RGB}{127,127,127}
\definecolor{lightgray204}{RGB}{204,204,204}
\definecolor{mediumpurple148103189}{RGB}{148,103,189}

\begin{axis}[
legend cell align={left},
legend style={
  fill opacity=1,
  draw opacity=1,
  text opacity=1,
  at={(0.5,-\legendDist)},
  anchor=north,
  draw=none
},
tick align=outside,
tick pos=left,
x grid style={darkgray176},
xlabel={Time horizon $T$},
xmin=-5, xmax=105,
xtick style={color=black},
y grid style={darkgray176},
ymin=-2, ymax=55,
ytick style={color=black},
  width=\figWidth,
  height=\figHeight,
  scale only axis,
 font=\small
]
\addplot [semithick, gray127, dashed, mark=star, mark size=2, mark options={solid}]
table {%
1 0.999999999988875
2 1.73205079427809
3 2.34209622334741
4 2.90670363021072
5 3.44220483725166
6 3.95167756792504
7 4.44511777330512
8 4.92357067946405
9 5.39080426794702
10 5.84791808629133
11 6.29615347280525
12 6.73716873663269
13 7.17116723126338
14 7.59935835351766
15 8.02193620499087
20 10.0681475303238
25 12.0284923064602
30 13.9253024928537
35 15.7722150044208
40 17.5783873084459
45 19.3503333022518
50 21.0929670020047
55 22.8100363786092
60 24.5043424164102
65 26.1782977655244
70 27.8339857818303
75 29.4730527040755
80 31.0969519615577
85 32.7069906330521
90 34.3032966796525
95 35.8879962752747
100 37.4610375681626
};
\addlegendentry{{\legendEntryFontSize Optimized algo.~from~(\ref{eq:jointstepsizeopt}) with $\beta_{t,s}=0$}}
\addplot [semithick, mediumpurple148103189, dashed, mark=o, mark size=2, mark options={solid}]
table {%
1 0.999999999986411
2 1.73205077164806
3 2.34209616885784
4 2.90282823623805
5 3.42168738269308
6 3.91697957391197
7 4.38834045416619
8 4.84470501303823
9 5.28476977360727
10 5.71400805695854
11 6.13109156852881
12 6.53978673154998
13 6.93890009089812
14 7.33121252392302
15 7.71568171523579
20 9.55298861900087
25 11.2764639001485
30 12.91523643836
35 14.4863179235929
40 16.0021461745957
45 17.4709535774897
50 18.8995128351095
55 20.2925817455325
60 21.6543713495292
65 22.9879410641388
70 24.2961479561498
75 25.5810970007411
80 26.8448520074568
85 28.0889320029728
90 29.31489412696
95 30.5238801404065
100 31.7171067431816
};
\addlegendentry{{\legendEntryFontSize Optimized algo.~from~(\ref{eq:jointstepsizeopt})}}
\end{axis}

\end{tikzpicture}
    \end{minipage}
    \hfill
    \end{subfigure}

    \repeatcaption{fig:Comparison_tau}{
        (Top left) Comparison of known upper bounds (respectively from~\cite[Theorem 7.3]{hazanIntroductionOnlineConvex2016} and Theorem~\ref{thm_upper_bound_OFW}) against tight numerical bounds
        (worst-case regrets) obtained from~\eqref{eq:PEP}, for~\textcolor{crimson2143940}{\cite[Algorithm 27]{hazanIntroductionOnlineConvex2016}} and \textcolor{forestgreen4416044}{Algorithm~\ref{OFW_alg_new} (parameters from Theorem~\ref{thm_upper_bound_OFW})}. 
        (Top right) Tight numerical regret bounds for \textcolor{crimson2143940}{\cite[Algorithm 27]{hazanIntroductionOnlineConvex2016}} and \textcolor{forestgreen4416044}{Algorithm~\ref{OFW_alg_new} (parameters from Theorem~\ref{thm_upper_bound_OFW})} against their anytime versions.
        (Bottom left) Tight numerical bounds for optimized online Frank--Wolfe with respectively $r\in\{1, 2,3\}$ linear optimization steps per time round
        (where \eqref{eq:jointstepsizeopt_multiple} is a variant of~\eqref{eq:jointstepsizeopt}
        with~\eqref{eq:OFW_general} replaced by~\eqref{eq:OFW_multiple}, which we detail in
        Appendix~B.5).
        (Bottom right) Tight numerical regret bounds for optimized online Frank--Wolfe with and without regularization (i.e.,~\eqref{eq:jointstepsizeopt} with and without $\beta_{t,s}=0$).
    }
    \end{figure}

\newcommand{\xminPlots}{1}
\newcommand{\xmaxPlots}{100}
\newcommand{\yminPlots}{1}
\newcommand{\ymaxPlots}{100}

\begin{figure}[p]
\begin{subfigure}{\textwidth}
    \centering

        \hfill
    \begin{minipage}[t]{0.48\textwidth}
        \vspace*{0pt}
        \centering
\begin{tikzpicture}

\definecolor{crimson2143940}{RGB}{214,39,40}
\definecolor{forestgreen4416044}{RGB}{44,160,44}
\definecolor{mediumpurple148103189}{RGB}{148,103,189}
\definecolor{darkgray176}{RGB}{176,176,176}
\definecolor{darkorange25512714}{RGB}{255,127,14}
\definecolor{lightgray204}{RGB}{204,204,204}
\definecolor{steelblue31119180}{RGB}{31,119,180}

\begin{axis}[
    xmode=log,
    ymode=log,
    log basis x={10},
    log basis y={10},
    log ticks with fixed point,
    grid=both,
    minor grid style={gray!20},
    major grid style={gray!50},
    legend cell align={left},
    legend style={
        fill opacity=1,
        draw opacity=1,
        text opacity=1,
        at={(0.5,-\legendDist)},
        anchor=north,
        draw=none
    },
    tick align=outside,
    tick pos=left,
    x grid style={darkgray176},
    y grid style={darkgray176},
    xlabel={Time horizon $T$},
    ylabel={Worst-case regret},
    xmin=\xminPlots, xmax=\xmaxPlots,
    ymin=\yminPlots, ymax=\ymaxPlots,
    width=\figWidth,
    height=\figHeight,
    scale only axis,
    font=\small
]
\addplot [thick, crimson2143940, mark=asterisk, mark size=1, mark options={solid}]
table {%
1 8
2 13.4543426440594
3 18.2360564556382
4 22.6274169979695
5 26.7496121990569
6 30.6692690038211
7 34.4281365652708
8 38.0546276800871
10 44.9873060152279
15 60.9759297785538
20 75.6593287202541
25 89.4427190999916
30 102.548881535096
35 115.117412732146
40 127.243316602728
45 138.995062234592
50 150.424123723456
55 161.570475875318
60 172.46597374228
65 183.136539862808
70 193.603639399487
75 203.885309381655
80 213.996897592455
90 233.760899142071
100 252.98221281347
};
\addlegendentry{{\legendEntryFontSize \cite[Algo.~27]{hazanIntroductionOnlineConvex2016}: bound~\cite[Theorem 7.3]{hazanIntroductionOnlineConvex2016}}}
\addplot [thick, crimson2143940, dashed, mark=asterisk, mark size=1, mark options={solid}]
table {%
1 0.999999978829868
2 1.86602540835967
3 2.86602530022636
4 3.83521194338494
5 4.84566679047157
6 5.77847121423904
7 6.64296540269502
8 7.47558796478797
10 9.06556940344561
15 12.7147146818818
20 16.0829251259494
25 19.2687001005565
30 22.3084311637338
35 25.2313421308435
40 28.0532078096698
45 30.7885720243218
50 33.4461650859791
55 36.0337573951143
60 38.5568065885238
65 41.02483254
70 43.44430215
75 45.8196833
80 48.15440388
90 52.7148001953334
100 57.147182191147
};
\addlegendentry{{\legendEntryFontSize \cite[Algo.~27]{hazanIntroductionOnlineConvex2016}: tight bound $B_T$ from~(\ref{eq:PEP})}}
\addplot [semithick, forestgreen4416044, mark size=3, mark options={solid}]
table {%
1 1.75476535060332
2 2.95115178586752
3 4
4 4.9632259152112
5 5.86741157862262
6 6.72717132202972
7 7.55166264132207
8 8.34711776039087
10 9.8677707265638
15 13.3748060995284
20 16.5955460610261
25 19.6188730425514
30 22.493653007614
35 25.2505058891838
40 27.9102703837895
45 30.4879648892769
50 32.9948800255984
55 35.4397840933123
60 37.829664360127
65 40.17020682258
70 42.466119771115
75 44.7213595499958
80 46.939292628981
90 51.2744407675481
100 55.4905526705042
};
\addlegendentry{{\legendEntryFontSize Algo.~\ref{OFW_alg_new}: bound from Theorem~\ref{thm_upper_bound_OFW}}}
\addplot [semithick, forestgreen4416044, dashed, mark=x, mark size=2, mark options={solid}]
table {%
1 0.99999997622009
2 1.86602536857043
3 2.79508504207169
4 3.41695016277647
5 4.00578768309647
6 4.56827781991907
7 5.10944325015655
8 5.63752492727876
10 6.66119581166219
15 9.02755847964825
20 11.1903480646082
25 13.2176738631319
30 15.1477451800894
35 17.0005288149411
40 18.7895976211105
45 20.5247895195582
50 22.2138706190613
55 23.8627054385189
60 25.4759388445294
65 27.057093360461
70 28.6091703686733
75 30.1349387581324
80 31.6366960286442
85 33.1162504419378
90 34.5752635225472
};
\addlegendentry{{\legendEntryFontSize Algo.~\ref{OFW_alg_new}: tight bound $B_T$ from~(\ref{eq:PEP})}}
\addplot [semithick, mediumpurple148103189, dashed, mark=o, mark size=2, mark options={solid}]
table {%
1 0.999999999986411
2 1.73205077164806
3 2.34209616885784
4 2.90282823623805
5 3.42168738269308
6 3.91697957391197
7 4.38834045416619
8 4.84470501303823
9 5.28476977360727
10 5.71400805695854
11 6.13109156852881
12 6.53978673154998
13 6.93890009089812
14 7.33121252392302
15 7.71568171523579
20 9.55298861900087
25 11.2764639001485
30 12.91523643836
35 14.4863179235929
40 16.0021461745957
45 17.4709535774897
50 18.8995128351095
55 20.2925817455325
60 21.6543713495292
65 22.9879410641388
70 24.2961479561498
75 25.5810970007411
80 26.8448520074568
85 28.0889320029728
90 29.31489412696
95 30.5238801404065
100 31.7171067431816
};
\addlegendentry{{\legendEntryFontSize Optimized algo.~and bound from~(\ref{eq:jointstepsizeopt})}}
\end{axis}

\end{tikzpicture}
    \end{minipage}
    \hfill
    \begin{minipage}[t]{0.48\textwidth}
        \vspace*{0pt}
        \centering
\begin{tikzpicture}

\definecolor{crimson2143940}{RGB}{214,39,40}
\definecolor{darkgray176}{RGB}{176,176,176}
\definecolor{darkturquoise23190207}{RGB}{23,190,207}
\definecolor{forestgreen4416044}{RGB}{44,160,44}
\definecolor{goldenrod18818934}{RGB}{188,189,34}
\definecolor{lightgray204}{RGB}{204,204,204}
\definecolor{steelblue31119180}{RGB}{31,119,180}

\begin{axis}[
    xmode=log,
    ymode=log,
    log basis x={10},
    log basis y={10},
    log ticks with fixed point,
    grid=both,
    minor grid style={gray!20},
    major grid style={gray!50},
    legend cell align={left},
    legend style={
        fill opacity=1,
        draw opacity=1,
        text opacity=1,
        at={(0.5,-\legendDist)},
        anchor=north,
        draw=none
    },
    tick align=outside,
    tick pos=left,
    x grid style={darkgray176},
    y grid style={darkgray176},
    xlabel={Time horizon $T$},
    xmin=\xminPlots, xmax=\xmaxPlots,
    ymin=\yminPlots, ymax=\ymaxPlots,
    width=\figWidth,
    height=\figHeight,
    scale only axis,
    font=\small
]
\addplot [semithick, forestgreen4416044, mark=Mercedes star, mark size=3, mark options={solid}]
table {%
1 2.193456688254154
2 3.688939732334405
3 5.0
4 6.204032394013997
5 7.334264473278278
6 8.408964152537145
7 9.439578301652581
8 10.433897200488582
10 12.334713408204754
15 16.71850762441055
20 20.74443257628261
25 24.523591303189267
30 28.117066259517458
35 31.56313236147981
40 34.88783797973685
45 38.109956111596105
50 41.24360003199805
55 44.299730116640355
60 47.28708045015879
65 50.212758528225045
70 53.08264971389377
75 55.90169943749474
80 58.674115786226224
85 61.40351851407624
90 64.0930509594351
95 66.74546564983403
100 69.3631908381303
};
\addlegendentry{{\legendEntryFontSize Anytime Algo.~\ref{OFW_alg_new_anytime}: bound from Theorem~\ref{thm_upper_bound_OFW_anytime}}}
\addplot [semithick, crimson2143940, dashed, mark=asterisk, mark size=1, mark options={solid}]
table {%
1 0.999999978829868
2 1.86602540835967
3 2.86602530022636
4 3.83521194338494
5 4.84566679047157
6 5.77847121423904
7 6.64296540269502
8 7.47558796478797
10 9.06556940344561
15 12.7147146818818
20 16.0829251259494
25 19.2687001005565
30 22.3084311637338
35 25.2313421308435
40 28.0532078096698
45 30.7885720243218
50 33.4461650859791
55 36.0337573951143
60 38.5568065885238
65 41.02483254
70 43.44430215
75 45.8196833
80 48.15440388
90 52.7148001953334
100 57.147182191147
};
\addlegendentry{{\legendEntryFontSize \cite[Algo.~27]{hazanIntroductionOnlineConvex2016}: tight bound $B_T$ from~(\ref{eq:PEP})}}
\addplot [thick, crimson2143940, dashed, mark=+, mark size=2, mark options={solid}]
table {%
1 0.999999978829868
2 1.86602540820839
3 2.86602528722214
4 3.81941917172879
5 4.7843884768231
6 5.66753177316362
7 6.50107951868022
8 7.31139133654082
9 8.09325316957615
10 8.85139312035817
11 9.58774550759338
12 10.3054843155196
13 11.0069165395407
14 11.6938619340762
15 12.3679052961251
20 15.5818013204948
25 18.5962663625794
30 21.4601077826225
35 24.2075625818537
40 26.8596151684412
45 29.4313051952187
50 31.9318777341248
55 34.3709642044624
60 36.7548638838124
65 39.0892721371677
70 41.3798427025696
75 43.63062719759103
};
\addlegendentry{{\legendEntryFontSize Anytime \cite[Algo.~27]{hazanIntroductionOnlineConvex2016}: tight bound from~(\ref{eq:PEP})}}
\addplot [semithick, forestgreen4416044, dashed, mark=square, mark size=2, mark options={solid}]
table {%
1 0.99999997622009
2 1.86602539808008
3 2.70563299799477
4 3.58153879868095
5 4.36926755755856
6 5.10221940117672
7 5.79270161250749
8 6.4526213576502
10 7.70024853201964
15 10.5307220687907
20 13.0925557276164
25 15.4754138829171
30 17.7332251672391
35 19.895717259638
40 21.9816644090207
45 24.0038401515299
50 25.9713602247288
55 27.8912843690133
60 29.7692345570552
65 31.6097702772639
70 33.4164604156972
75 35.1922403201557
80 36.9396265985578
85 38.6611721882096
90 40.3588001390253
};
\addlegendentry{{\legendEntryFontSize Anytime Algo.~\ref{OFW_alg_new_anytime}: tight bound $B_T$ from~(\ref{eq:PEP})}}
\addplot [semithick, forestgreen4416044, dashed, mark=x, mark size=2, mark options={solid}]
table {%
1 0.99999997622009
2 1.86602536857043
3 2.79508504207169
4 3.41695016277647
5 4.00578768309647
6 4.56827781991907
7 5.10944325015655
8 5.63752492727876
10 6.66119581166219
15 9.02755847964825
20 11.1903480646082
25 13.2176738631319
30 15.1477451800894
35 17.0005288149411
40 18.7895976211105
45 20.5247895195582
50 22.2138706190613
55 23.8627054385189
60 25.4759388445294
65 27.057093360461
70 28.6091703686733
75 30.1349387581324
80 31.6366960286442
85 33.1162504419378
90 34.5752635225472
};
\addlegendentry{{\legendEntryFontSize Algo.~\ref{OFW_alg_new}: tight bound $B_T$ from~(\ref{eq:PEP})}}
\end{axis}

\end{tikzpicture}
    \end{minipage}
    \hfill
    
\medskip

    \hfill
        \begin{minipage}[t]{0.48\textwidth}
        \vspace*{0pt}
        \centering
\begin{tikzpicture}

\definecolor{darkgray176}{RGB}{176,176,176}
\definecolor{lightgray204}{RGB}{204,204,204}
\definecolor{mediumpurple148103189}{RGB}{148,103,189}
\definecolor{orchid227119194}{RGB}{227,119,194}
\definecolor{sienna1408675}{RGB}{140,86,75}

\begin{axis}[
    xmode=log,
    ymode=log,
    log basis x={10},
    log basis y={10},
    log ticks with fixed point,
    grid=both,
    minor grid style={gray!20},
    major grid style={gray!50},
    legend cell align={left},
    legend style={
        fill opacity=1,
        draw opacity=1,
        text opacity=1,
        at={(0.5,-\legendDist)},
        anchor=north,
        draw=none
    },
    tick align=outside,
    tick pos=left,
    x grid style={darkgray176},
    y grid style={darkgray176},
    xlabel={Time horizon $T$},
    ylabel={Worst-case regret},
    xmin=\xminPlots, xmax=\xmaxPlots,
    ymin=\yminPlots, ymax=\ymaxPlots,
    width=\figWidth,
    height=\figHeight,
    scale only axis,
    font=\small
]
\addplot [semithick, mediumpurple148103189, dashed, mark=o, mark size=2, mark options={solid}]
table {%
1 0.999999999986411
2 1.73205077164806
3 2.34209616885784
4 2.90282823623805
5 3.42168738269308
6 3.91697957391197
7 4.38834045416619
8 4.84470501303823
9 5.28476977360727
10 5.71400805695854
11 6.13109156852881
12 6.53978673154998
13 6.93890009089812
14 7.33121252392302
15 7.71568171523579
20 9.55298861900087
25 11.2764639001485
30 12.91523643836
35 14.4863179235929
40 16.0021461745957
45 17.4709535774897
50 18.8995128351095
55 20.2925817455325
60 21.6543713495292
65 22.9879410641388
70 24.2961479561498
75 25.5810970007411
80 26.8448520074568
85 28.0889320029728
90 29.31489412696
95 30.5238801404065
100 31.7171067431816
};
\addlegendentry{{\legendEntryFontSize Optimized algo.~(\ref{eq:OFW_multiple}), $r=1$,  from~(\ref{eq:jointstepsizeopt})}}
\addplot [semithick, sienna1408675, dashed, mark=triangle, mark size=2, mark options={solid}]
table {%
1 0.999999999986411
2 1.63299311979155
3 2.15168457507838
4 2.62509003041537
5 3.06278519425511
6 3.4795301018729
7 3.8762878127482
8 4.25983704757672
9 4.62989237992315
10 4.99045867833446
11 5.3410098087897
12 5.68423150190584
13 6.01957981552844
14 6.34900120418311
15 6.6719871075061
20 8.21488294667517
25 9.66223380277867
30 11.0384070754211
35 12.3578609109691
40 13.6309257863288
45 14.8646107403946
50 16.0645089825127
55 17.2346879049347
60 18.378617200216
65 19.4989121739457
70 20.597904053024
75 21.6774336420247
80 22.7391637690755
};
\addlegendentry{{\legendEntryFontSize Optimized algo.~(\ref{eq:OFW_multiple}), $r=2$, from~(\ref{eq:jointstepsizeopt})}}
\addplot [semithick, orchid227119194, dashed, mark=square, mark size=2, mark options={solid}]
table {%
1 0.999999999986411
2 1.5837637868675
3 2.05505753723153
4 2.48625209927649
5 2.88336510543731
6 3.26149038547669
7 3.62091546948163
8 3.96831411446438
9 4.30323394828769
10 4.62950562985113
11 4.94658359011182
12 5.25698225939425
13 5.560188906985
14 5.85799574095499
15 6.14994478087934
20 7.54419607303208
25 8.85176336291036
30 10.0948492043156
35 11.2866164238039
40 12.4364369256464
45 13.5506663055645
50 14.6343709611955
};
\addlegendentry{{\legendEntryFontSize Optimized algo.~(\ref{eq:OFW_multiple}), $r=3$, from~(\ref{eq:jointstepsizeopt})}}
\end{axis}

\end{tikzpicture}
    \end{minipage}
    \hfill
    \begin{minipage}[t]{0.48\textwidth}
        \vspace*{0pt}
        \centering
\begin{tikzpicture}

\definecolor{darkgray176}{RGB}{176,176,176}
\definecolor{gray127}{RGB}{127,127,127}
\definecolor{lightgray204}{RGB}{204,204,204}
\definecolor{mediumpurple148103189}{RGB}{148,103,189}

\begin{axis}[
    xmode=log,
    ymode=log,
    log basis x={10},
    log basis y={10},
    log ticks with fixed point,
    grid=both,
    minor grid style={gray!20},
    major grid style={gray!50},
    legend cell align={left},
    legend style={
        fill opacity=1,
        draw opacity=1,
        text opacity=1,
        at={(0.5,-\legendDist)},
        anchor=north,
        draw=none
    },
    tick align=outside,
    tick pos=left,
    x grid style={darkgray176},
    y grid style={darkgray176},
    xlabel={Time horizon $T$},
    xmin=\xminPlots, xmax=\xmaxPlots,
    ymin=\yminPlots, ymax=\ymaxPlots,
    width=\figWidth,
    height=\figHeight,
    scale only axis,
    font=\small
]
\addplot [semithick, gray127, dashed, mark=star, mark size=2, mark options={solid}]
table {%
1 0.999999999988875
2 1.73205079427809
3 2.34209622334741
4 2.90670363021072
5 3.44220483725166
6 3.95167756792504
7 4.44511777330512
8 4.92357067946405
9 5.39080426794702
10 5.84791808629133
11 6.29615347280525
12 6.73716873663269
13 7.17116723126338
14 7.59935835351766
15 8.02193620499087
20 10.0681475303238
25 12.0284923064602
30 13.9253024928537
35 15.7722150044208
40 17.5783873084459
45 19.3503333022518
50 21.0929670020047
55 22.8100363786092
60 24.5043424164102
65 26.1782977655244
70 27.8339857818303
75 29.4730527040755
80 31.0969519615577
85 32.7069906330521
90 34.3032966796525
95 35.8879962752747
100 37.4610375681626
};
\addlegendentry{{\legendEntryFontSize Optimized algo.~from~(\ref{eq:jointstepsizeopt}) with $\beta_{t,s}=0$}}
\addplot [semithick, mediumpurple148103189, dashed, mark=o, mark size=2, mark options={solid}]
table {%
1 0.999999999986411
2 1.73205077164806
3 2.34209616885784
4 2.90282823623805
5 3.42168738269308
6 3.91697957391197
7 4.38834045416619
8 4.84470501303823
9 5.28476977360727
10 5.71400805695854
11 6.13109156852881
12 6.53978673154998
13 6.93890009089812
14 7.33121252392302
15 7.71568171523579
20 9.55298861900087
25 11.2764639001485
30 12.91523643836
35 14.4863179235929
40 16.0021461745957
45 17.4709535774897
50 18.8995128351095
55 20.2925817455325
60 21.6543713495292
65 22.9879410641388
70 24.2961479561498
75 25.5810970007411
80 26.8448520074568
85 28.0889320029728
90 29.31489412696
95 30.5238801404065
100 31.7171067431816
};
\addlegendentry{{\legendEntryFontSize Optimized algo.~from~(\ref{eq:jointstepsizeopt})}}
\end{axis}

\end{tikzpicture}
    \end{minipage}
    \hfill
    \end{subfigure}

    \caption{Variant of Figure~\ref{fig:Comparison_tau} with log-log plots.
        (Top left) Comparison of known upper bounds (respectively from~\cite[Theorem 7.3]{hazanIntroductionOnlineConvex2016} and Theorem~\ref{thm_upper_bound_OFW}) against tight numerical bounds
        (worst-case regrets) obtained from~\eqref{eq:PEP}, for~\textcolor{crimson2143940}{\cite[Algorithm 27]{hazanIntroductionOnlineConvex2016}} and \textcolor{forestgreen4416044}{Algorithm~\ref{OFW_alg_new} (parameters from Theorem~\ref{thm_upper_bound_OFW})}. 
        (Top right) Tight numerical regret bounds for \textcolor{crimson2143940}{\cite[Algorithm 27]{hazanIntroductionOnlineConvex2016}} and \textcolor{forestgreen4416044}{Algorithm~\ref{OFW_alg_new} (parameters from Theorem~\ref{thm_upper_bound_OFW})} against their anytime versions.
        (Bottom left) Tight numerical bounds for optimized online Frank--Wolfe with respectively $r\in\{1, 2,3\}$ linear optimization steps per time round
        (where \eqref{eq:jointstepsizeopt_multiple} is a variant of~\eqref{eq:jointstepsizeopt}
        with~\eqref{eq:OFW_general} replaced by~\eqref{eq:OFW_multiple}, which we detail in
        Appendix~B.5).
        (Bottom right) Tight numerical regret bounds for optimized online Frank--Wolfe with and without regularization (i.e.,~\eqref{eq:jointstepsizeopt} with and without $\beta_{t,s}=0$).
    }
    \label{fig:Comparison_tau_log_log}
\end{figure}

\renewcommand{\xminPlots}{1}
\renewcommand{\xmaxPlots}{100}
\renewcommand{\yminPlots}{0.6}
\renewcommand{\ymaxPlots}{1}

\begin{figure}[p]
\begin{subfigure}{\textwidth}
    \centering

        \hfill
    \begin{minipage}[t]{0.48\textwidth}
        \vspace*{0pt}
        \centering
\begin{tikzpicture}

\definecolor{crimson2143940}{RGB}{214,39,40}
\definecolor{forestgreen4416044}{RGB}{44,160,44}
\definecolor{mediumpurple148103189}{RGB}{148,103,189}
\definecolor{darkgray176}{RGB}{176,176,176}
\definecolor{darkorange25512714}{RGB}{255,127,14}
\definecolor{lightgray204}{RGB}{204,204,204}
\definecolor{steelblue31119180}{RGB}{31,119,180}

\begin{axis}[
    legend cell align={left},
    legend style={
        fill opacity=1,
        draw opacity=1,
        text opacity=1,
        at={(0.5,-\legendDist)},
        anchor=north,
        draw=none
    },
    tick align=outside,
    tick pos=left,
    x grid style={darkgray176},
    y grid style={darkgray176},
    xlabel={Time horizon $T$},
    ylabel={Estimated regret rate exponent},
    xmin=\xminPlots, xmax=\xmaxPlots,
    ymin=\yminPlots, ymax=\ymaxPlots,
    width=\figWidth,
    height=\figHeight,
    scale only axis,
    font=\small
]
\addplot [thick, crimson2143940, dashed, mark=asterisk, mark size=1, mark options={solid}]
table {%
1 0.912903415724768
2 0.916214169987245
3 0.891189009807578
4 0.873858428574496
5 0.852167856233249
6 0.839315094900244
7 0.832417270196420
8 0.827173884427762
10 0.819653156769017
15 0.811089971565822
20 0.806990593112379
25 0.803430300216490
35 0.798713772166231
40 0.796498362897148
45 0.794590928813371
50 0.792780214809520
55 0.791062853534045
60 0.789396507506225
65 0.787921412623536
70 0.786636025395031
75 0.785502956683877
80 0.784487113611230
90 0.782743445988397
100 0.781301200533998
};
\addlegendentry{{\legendEntryFontSize \cite[Algo.~27]{hazanIntroductionOnlineConvex2016}: tight bound $B_T$ from~(\ref{eq:PEP})}}
\addplot [semithick, forestgreen4416044, dashed, mark=x, mark size=2, mark options={solid}]
table {%
1 0.799086735252323
2 0.773265203700880
3 0.733952964595144
4 0.739044651159048
5 0.742349335729479
6 0.744803753779267
7 0.746767523451280
8 0.747798142042031
10 0.747809465286182
15 0.746695291135244
20 0.746796839028954
25 0.747562862157702
35 0.748571219936961
40 0.748922965395484
45 0.749217889522795
50 0.749502747585108
55 0.749774403225830
60 0.750032256774033
65 0.750265262522091
70 0.750474715610082
75 0.750671749559617
80 0.750860583970982
85 0.751038949335933
90 0.751208428959683
};
\addlegendentry{{\legendEntryFontSize Algo.~\ref{OFW_alg_new}: tight bound $B_T$ from~(\ref{eq:PEP})}}
\addplot [semithick, mediumpurple148103189, dashed, mark=o, mark size=2, mark options={solid}]
table {%
1 0.752659575802967
2 0.741731141039525
3 0.741262281325756
4 0.740503068552256
5 0.740994037439337
6 0.740932488631831
7 0.741392013841200
8 0.741448073567512
9 0.741827173003506
10 0.741900987755124
11 0.742216009082672
12 0.742283431387908
13 0.742552050546178
14 0.742606168721942
15 0.742842275360739
20 0.743304755990148
30 0.744234654089516
35 0.744450425118872
40 0.744686028965672
45 0.744865475213617
50 0.745034312196071
55 0.745173683002126
60 0.745303176735264
65 0.745414476915300
70 0.745518287393159
75 0.745609689497942
80 0.745695528731086
85 0.745772372425718
90 0.745844923478841
95 0.745910729667112
100 0.745973219203652
};
\addlegendentry{{\legendEntryFontSize Optimized algo.~and bound from~(\ref{eq:jointstepsizeopt})}}
\addplot [semithick, black] coordinates {(\xminPlots,0.75) (\xmaxPlots,0.75)};
\addlegendentry{{\legendEntryFontSize $y=3/4$}}
\end{axis}

\end{tikzpicture}
    \end{minipage}
    \hfill
    \begin{minipage}[t]{0.48\textwidth}
        \vspace*{0pt}
        \centering
\begin{tikzpicture}

\definecolor{crimson2143940}{RGB}{214,39,40}
\definecolor{darkgray176}{RGB}{176,176,176}
\definecolor{darkturquoise23190207}{RGB}{23,190,207}
\definecolor{forestgreen4416044}{RGB}{44,160,44}
\definecolor{goldenrod18818934}{RGB}{188,189,34}
\definecolor{lightgray204}{RGB}{204,204,204}
\definecolor{steelblue31119180}{RGB}{31,119,180}

\begin{axis}[
    legend cell align={left},
    legend style={
        fill opacity=1,
        draw opacity=1,
        text opacity=1,
        at={(0.5,-\legendDist)},
        anchor=north,
        draw=none
    },
    tick align=outside,
    tick pos=left,
    x grid style={darkgray176},
    y grid style={darkgray176},
    xlabel={Time horizon $T$},
    xmin=\xminPlots, xmax=\xmaxPlots,
    ymin=\yminPlots, ymax=\ymaxPlots,
    width=\figWidth,
    height=\figHeight,
    scale only axis,
    font=\small
]
\addplot [semithick, crimson2143940, dashed, mark=asterisk, mark size=1, mark options={solid}]
table {%
1 0.912903415724768
2 0.916214169987245
3 0.891189009807578
4 0.873858428574496
5 0.852167856233249
6 0.839315094900244
7 0.832417270196420
8 0.827173884427762
10 0.819653156769017
15 0.811089971565822
20 0.806990593112379
25 0.803430300216490
35 0.798713772166231
40 0.796498362897148
45 0.794590928813371
50 0.792780214809520
55 0.791062853534045
60 0.789396507506225
65 0.787921412623536
70 0.786636025395031
75 0.785502956683877
80 0.784487113611230
90 0.782743445988397
100 0.781301200533998
};
\addlegendentry{{\legendEntryFontSize \cite[Algo.~27]{hazanIntroductionOnlineConvex2016}: tight bound $B_T$ from~(\ref{eq:PEP})}}
\addplot [thick, crimson2143940, dashed, mark=+, mark size=2, mark options={solid}]
table {%
1 0.931796394378238
2 0.943133661787073
3 0.912812786696562
4 0.893198274127307
5 0.867996321270575
6 0.854847400657382
7 0.847000257203278
8 0.839201652778071
9 0.832969734956632
10 0.827673702169855
11 0.823422690651243
12 0.819898244476085
13 0.816919284662145
15 0.812268721614491
20 0.804765159438654
25 0.800072346863818
30 0.796614316392275
35 0.794072207971560
40 0.792095880962948
45 0.790502524978181
50 0.789134183894479
55 0.787961036207844
60 0.786928146347430
65 0.786009537538491
70 0.785199465063475
75 0.784479939177425
};
\addlegendentry{{\legendEntryFontSize Anytime \cite[Algo.~27]{hazanIntroductionOnlineConvex2016}: tight bound from~(\ref{eq:PEP})}}
\addplot [semithick, forestgreen4416044, mark=Mercedes star, mark size=3, mark options={solid}]
table {%
1 0.787189167199614
2 0.788814525614923
3 0.787433617567582
4 0.785858554639797
5 0.784426384145482
6 0.783170250555516
7 0.782071259555838
8 0.781104037969276
10 0.779478660150428
15 0.776577871047360
20 0.774614824930819
25 0.773166132730741
30 0.772035994512196
40 0.770355984356069
45 0.769705275099013
50 0.769141489009257
55 0.768646278068820
60 0.768206354717752
65 0.767811818354422
70 0.767455114155121
75 0.767130359427163
80 0.766832893077508
85 0.766558965620909
90 0.766305520571186
95 0.766070036917877
100 0.765850413438743
};
\addlegendentry{{\legendEntryFontSize Anytime Algo.~\ref{OFW_alg_new_anytime}: bound~\eqref{eq:bound_anytime_OFW_with_sums} from proof of Thm.~\ref{thm_upper_bound_OFW_anytime}}} 
\addplot [semithick, forestgreen4416044, dashed, mark=square, mark size=2, mark options={solid}]
table {%
1 0.845419923980947
2 0.831457738059934
3 0.816518841591583
4 0.793907920959655
5 0.781826268555319
6 0.774036933438049
7 0.768811270034725
8 0.764856501635855
10 0.759310096427199
15 0.751850582899993
20 0.748267109514892
25 0.746962931809545
35 0.746441647517257
40 0.746548646581950
45 0.746732110207052
50 0.746936675164095
55 0.747149798906259
60 0.747367181556198
65 0.747586407325646
70 0.747799123488908
75 0.748000149888328
80 0.748188184716978
85 0.748372831989320
90 0.748552939461185
};
\addlegendentry{{\legendEntryFontSize Anytime Algo.~\ref{OFW_alg_new_anytime}: tight bound $B_T$ from~(\ref{eq:PEP})}}
\addplot [semithick, forestgreen4416044, dashed, mark=x, mark size=2, mark options={solid}]
table {%
1 0.799086735252323
2 0.773265203700880
3 0.733952964595144
4 0.739044651159048
5 0.742349335729479
6 0.744803753779267
7 0.746767523451280
8 0.747798142042031
10 0.747809465286182
15 0.746695291135244
20 0.746796839028954
25 0.747562862157702
35 0.748571219936961
40 0.748922965395484
45 0.749217889522795
50 0.749502747585108
55 0.749774403225830
60 0.750032256774033
65 0.750265262522091
70 0.750474715610082
75 0.750671749559617
80 0.750860583970982
85 0.751038949335933
90 0.751208428959683
};
\addlegendentry{{\legendEntryFontSize Algo.~\ref{OFW_alg_new}: tight bound $B_T$ from~(\ref{eq:PEP})}}
\addplot [semithick, black] coordinates {(\xminPlots,0.75) (\xmaxPlots,0.75)};
\addlegendentry{{\legendEntryFontSize $y=3/4$}}
\end{axis}

\end{tikzpicture}
    \end{minipage}
    \hfill
    
\medskip

    \hfill
        \begin{minipage}[t]{0.48\textwidth}
        \vspace*{0pt}
        \centering
\begin{tikzpicture}

\definecolor{darkgray176}{RGB}{176,176,176}
\definecolor{lightgray204}{RGB}{204,204,204}
\definecolor{mediumpurple148103189}{RGB}{148,103,189}
\definecolor{orchid227119194}{RGB}{227,119,194}
\definecolor{sienna1408675}{RGB}{140,86,75}

\begin{axis}[
    legend cell align={left},
    legend style={
        fill opacity=1,
        draw opacity=1,
        text opacity=1,
        at={(0.5,-\legendDist)},
        anchor=north,
        draw=none
    },
    tick align=outside,
    tick pos=left,
    x grid style={darkgray176},
    y grid style={darkgray176},
    xlabel={Time horizon $T$},
    ylabel={Estimated regret rate exponent},
    xmin=\xminPlots, xmax=\xmaxPlots,
    ymin=\yminPlots, ymax=\ymaxPlots,
    width=\figWidth,
    height=\figHeight,
    scale only axis,
    font=\small
]
\addplot [semithick, mediumpurple148103189, dashed, mark=o, mark size=2, mark options={solid}]
table {%
1 0.752659575802967
2 0.741731141039525
3 0.741262281325756
4 0.740503068552256
5 0.740994037439337
6 0.740932488631831
7 0.741392013841200
8 0.741448073567512
9 0.741827173003506
10 0.741900987755124
11 0.742216009082672
12 0.742283431387908
13 0.742552050546178
14 0.742606168721942
15 0.742842275360739
20 0.743304755990148
30 0.744234654089516
35 0.744450425118872
40 0.744686028965672
45 0.744865475213617
50 0.745034312196071
55 0.745173683002126
60 0.745303176735264
65 0.745414476915300
70 0.745518287393159
75 0.745609689497942
80 0.745695528731086
85 0.745772372425718
90 0.745844923478841
95 0.745910729667112
100 0.745973219203652
};
\addlegendentry{{\legendEntryFontSize Optimized algo.~(\ref{eq:OFW_multiple}), $r=1$,  from~(\ref{eq:jointstepsizeopt})}}
\addplot [semithick, sienna1408675, dashed, mark=triangle, mark size=2, mark options={solid}]
table {%
1 0.700842891305142
2 0.698546340756624
3 0.703144038888163
4 0.705729841891444
5 0.708705923434127
6 0.710495537724107
7 0.712519680472323
8 0.713777671041859
9 0.715281481539026
10 0.716189634753720
11 0.717393065714399
12 0.718021173731679
13 0.719075243085120
14 0.719207967484609
20 0.723123857379991
25 0.724918864796314
30 0.726343589335416
35 0.727459217928846
40 0.728387159310801
45 0.729161274934204
50 0.729829354776278
55 0.730407748316626
60 0.730919908391048
65 0.731374094101164
70 0.731783064588181
75 0.732152544244962
80 0.732489990158449
};
\addlegendentry{{\legendEntryFontSize Optimized algo.~(\ref{eq:OFW_multiple}), $r=2$, from~(\ref{eq:jointstepsizeopt})}}
\addplot [semithick, orchid227119194, dashed, mark=square, mark size=2, mark options={solid}]
table {%
1 0.667854926189417
2 0.669594892591982
3 0.677528039716150
4 0.681569612274872
5 0.686031459606080
6 0.688703490673877
7 0.691711839823645
8 0.693563498667992
9 0.695873484352289
10 0.697158703645201
11 0.699448678959690
13 0.700562714873368
14 0.702237410245523
15 0.703071878919015
20 0.707132476548424
25 0.709920197933992
30 0.712075484480631
35 0.713780764629417
40 0.715193462223356
45 0.716379337354479
50 0.717401763489467
};
\addlegendentry{{\legendEntryFontSize Optimized algo.~(\ref{eq:OFW_multiple}), $r=3$, from~(\ref{eq:jointstepsizeopt})}}
\addplot [semithick, black] coordinates {(\xminPlots,0.75) (\xmaxPlots,0.75)};
\addlegendentry{{\legendEntryFontSize $y=3/4$}}
\end{axis}

\end{tikzpicture}
    \end{minipage}
    \hfill
    \begin{minipage}[t]{0.48\textwidth}
        \vspace*{0pt}
        \centering
\begin{tikzpicture}

\definecolor{darkgray176}{RGB}{176,176,176}
\definecolor{gray127}{RGB}{127,127,127}
\definecolor{lightgray204}{RGB}{204,204,204}
\definecolor{mediumpurple148103189}{RGB}{148,103,189}

\begin{axis}[
    legend cell align={left},
    legend style={
        fill opacity=1,
        draw opacity=1,
        text opacity=1,
        at={(0.5,-\legendDist)},
        anchor=north,
        draw=none
    },
    tick align=outside,
    tick pos=left,
    x grid style={darkgray176},
    y grid style={darkgray176},
    xlabel={Time horizon $T$},
    xmin=\xminPlots, xmax=\xmaxPlots,
    ymin=\yminPlots, ymax=\ymaxPlots,
    width=\figWidth,
    height=\figHeight,
    scale only axis,
    font=\small
]
\addplot [semithick, gray127, dashed, mark=star, mark size=2, mark options={solid}]
table {%
1 0.772716417054260
2 0.767292267136544
3 0.771711542496316
4 0.775004307547188
5 0.777393136775009
6 0.779991064099735
7 0.782009954346464
8 0.783936548485266
9 0.785574928470424
10 0.787078220938947
11 0.788495300691147
12 0.789731294013139
13 0.790929538686677
14 0.791996774031665
15 0.793026740764458
20 0.797251198705091
30 0.803137691638591
35 0.805331388861883
40 0.807210701519023
45 0.808850350501880
50 0.810306221979502
55 0.811613245045121
60 0.812789901737111
65 0.813859998931234
70 0.814844374390911
75 0.815754741815583
80 0.816602287841113
85 0.817397224777674
90 0.818124426996536
95 0.818819330491156
100 0.819467518730970
};
\addlegendentry{{\legendEntryFontSize Optimized algo.~from~(\ref{eq:jointstepsizeopt}) with $\beta_{t,s}=0$}}
\addplot [semithick, mediumpurple148103189, dashed, mark=o, mark size=2, mark options={solid}]
table {%
1 0.752659575802967
2 0.741731141039525
3 0.741262281325756
4 0.740503068552256
5 0.740994037439337
6 0.740932488631831
7 0.741392013841200
8 0.741448073567512
9 0.741827173003506
10 0.741900987755124
11 0.742216009082672
12 0.742283431387908
13 0.742552050546178
14 0.742606168721942
15 0.742842275360739
20 0.743304755990148
30 0.744234654089516
35 0.744450425118872
40 0.744686028965672
45 0.744865475213617
50 0.745034312196071
55 0.745173683002126
60 0.745303176735264
65 0.745414476915300
70 0.745518287393159
75 0.745609689497942
80 0.745695528731086
85 0.745772372425718
90 0.745844923478841
95 0.745910729667112
100 0.745973219203652
};
\addlegendentry{{\legendEntryFontSize Optimized algo.~from~(\ref{eq:jointstepsizeopt})}}
\addplot [semithick, black] coordinates {(\xminPlots,0.75) (\xmaxPlots,0.75)};
\addlegendentry{{\legendEntryFontSize $y=3/4$}}
\end{axis}

\end{tikzpicture}
    \end{minipage}
    \hfill
    \end{subfigure}

    \caption{
    Estimation of regret rate exponents for the tight worst-case regret bounds from Figure~\ref{fig:Comparison_tau}
    by interpolating the slopes of the log-log plots of Figure~\ref{fig:Comparison_tau_log_log}
    using~\eqref{eq:def_regret_rate_interpolation}.
        (Top left)         Tight numerical regret bounds
        (worst-case regrets) obtained from~\eqref{eq:PEP}, for~\textcolor{crimson2143940}{\cite[Algorithm 27]{hazanIntroductionOnlineConvex2016}} and \textcolor{forestgreen4416044}{Algorithm~\ref{OFW_alg_new} (parameters from Theorem~\ref{thm_upper_bound_OFW})}. 
        (Top right) Tight numerical regret bounds for \textcolor{crimson2143940}{\cite[Algorithm 27]{hazanIntroductionOnlineConvex2016}} and \textcolor{forestgreen4416044}{Algorithm~\ref{OFW_alg_new} (parameters from Theorem~\ref{thm_upper_bound_OFW})} against their anytime versions.
        The bound~\eqref{eq:bound_anytime_OFW_with_sums} from the proof of Theorem~\ref{thm_upper_bound_OFW_anytime} is also included
        to compare to the speed of convergence for the sum of terms in $s^{-1/4}$ in~\eqref{eq:bound_anytime_OFW_with_sums}.
        (Bottom left) Tight numerical bounds for optimized online Frank--Wolfe with respectively $r\in\{1, 2,3\}$ linear optimization steps per time round
        (where \eqref{eq:jointstepsizeopt_multiple} is a variant of~\eqref{eq:jointstepsizeopt}
        with~\eqref{eq:OFW_general} replaced by~\eqref{eq:OFW_multiple}, which we detail in
        Appendix~B.5).
        (Bottom right) Tight numerical regret bounds for optimized online Frank--Wolfe with and without regularization (i.e.,~\eqref{eq:jointstepsizeopt} with and without $\beta_{t,s}=0$).
    }
    \label{fig:Comparison_tau_slopes}
\end{figure}

\section{The PEP methodology for OGD and FTRL}
\label{sec:PEP_for_OGD_and_FTRL}

This section presents how to leverage the semidefinite programming method presented in Section~\ref{sec:PEP_for_OFW} to compute tight worst-case regret bounds for OGD and FTRL.
We use the same assumptions as for OFW: the cost functions $\ell_t$ are convex and $L$-Lipschitz
and the domain of feasible points $\cK$ is convex closed with diameter bounded by $D$.

\subsection{Deriving proof for FTRL from the PEP methodology}
\label{s:PEP_for_FTRL}

We start by studying FTRL with the same regularization parameter for all time steps using the PEP methodology.
For convenience, we restate the definition of FTRL with single parameter
(note that in this section we will use notations with $x$ for FTRL iterates,
whereas we used notations with $y$ for FTRL iterates in the previous sections to distinguish them from OFW iterates).
\begin{algorithm}[H]
\caption{Follow The Regularized Leader (FTRL)}
\label{FTRL_alg_appendix}
\begin{algorithmic}[1]
\Require $T\geq 1$,~ $x_1 \in \cK$, $\eta \geq 0$
\For{$t=1$ to $T$} 
    \State Play $x_t$, pay cost $\ell_t(x_t)$,  observe $g_t = \nabla\ell_t(x_t)$.
    \State $x_{t+1} \gets \argmin_{x\in \cK} \eta \langle  \sum_{s=1}^t g_s , x \rangle + \frac{1}{2} \Vert x - x_1 \Vert^2 $
\EndFor
\end{algorithmic}
\end{algorithm}

Adapting the PEP formulation~\eqref{eq:PEP} of OFW to the case of FTRL, we get:
\begin{equation}\label{eq:PEP_FTRL}
\begin{aligned}
B_T(\eta)\triangleq
\sup_{\substack{\cK, \{\ell_t\}_{t\in\llbracket 1,T \rrbracket}\\ 
			x_\star, \{x_t\}_{t\in\llbracket 1,T \rrbracket}\\
			d\in\mathbb{N}}} \,
&
R_T(x_1,\ldots,x_T; x_\star)\\
\text{subject to: } 
&  \ell_t \text{ is convex and $L$-Lipschitz for }t\in\llbracket 1,T \rrbracket,\\
& \cK \text{ is a non-empty closed convex set of $\mathbb{R}^d$,}\\ 
&\mathrm{Diam}(\cK)\leq D,\\
& \{x_t\}_{t=1,\ldots,T} \text{ is generated by Algorithm~\ref{FTRL_alg_appendix}}.\\
\end{aligned}
\end{equation}

We now show how to use~\eqref{eq:PEP_FTRL}
to derive a proof that $B_T(\eta)$ is an upper bound of the regret for FTRL.

\begin{lemma}\label{lemma_bound_FTRL}
Let $T\geq 1$.
Assume that the cost functions $\ell_t$ are convex and $L$-Lipschitz for all $t\in \llbracket 1, T \rrbracket$,
and that the convex closed domain $\cK$ of feasible points has a diameter bounded by $D$.
Then, for any $x_\star \in\cK$, the following upper bound 
on the regret of the FTRL Algorithm~\ref{FTRL_alg} holds:
\begin{equation*}
R_T(x_1,\cdots, x_T; x_\star)
\leq \frac{\eta}{2} \sum_{t=1}^T \Vert g_t \Vert^2
	+ \frac{1}{2\eta} \Vert x_\star - x_1 \Vert^2.
\end{equation*}
In particular, for $\eta = D / (L \sqrt{T})$, we get that the regret is upper bounded by
$D L \sqrt{T}$.
\end{lemma}

Note that the tight regret upper bound from Lemma~\ref{lemma_bound_FTRL} is not new,
it can be found, e.g., in \cite[Corollary~7.9]{orabona2019modern}.
Nevertheless, the proof of Lemma~\ref{lemma_bound_FTRL}
will allow us to see how to easily find
simple clean proofs for regret upper bounds using the PEP methodology.

\begin{remark}
Also note that it is important to use the regularization in FTRL 
with respect to $x_1$ (or any point in $\cK$),
as if we were to regularize with respect to an arbitrary point 
we would have the trivial upper bound $L D T$ for the regret
(as FTRL can be equivalently defined as online mirror descent with greedy/lazy updates,
this can be seen as having the feasible set $\cK$ far from 
the reference regularization point, say the origin, 
and thus projection of gradient descent iterates
on $\cK$ would always land on the same point
opposite of the optimum).
\end{remark}

We start by explaining how to use~\eqref{eq:PEP_FTRL} to infer
the proof of Lemma~\ref{lemma_bound_FTRL},
and then we will do the proof itself.

We first rewrite~\eqref{eq:PEP_FTRL} as an SDP using a similar argument to that of Section~\ref{s:wc_regret_construction}.
We sample function values and gradients for the convex loss functions $\ell_t$ 
and the indicator function $i_{\cK}$ of the closed convex set $\cK$ of feasible points,
and then use the interpolation / extension theorems \cite[Theorem 3.3 and Equation (7)]{taylor2017exact}
and \cite[Theorem 3.6]{taylor2017exact}, respectively.
For all $t\in\llbracket 1,T \rrbracket$, we denote $(f_t,g_t)$ and $(f_t^\star,g_t^\star)$ by the function values and gradients for the loss function $\ell_t$ at points $x_t$ and $x_\star$, respectively.
We denote by $s_t \in \partial i_{\cK}(x_t)$ (resp. $s_\star \in \partial i_{\cK}(x_\star)$) a sub-gradient of the indicator function $i_{\cK}$ at point $x_t$ for $t=1,\dots,T$ (resp. at point $x_\star$).
Hence, the upper bound on the regret for FTRL given by~\eqref{eq:PEP_FTRL}
can be reformulated as the following finite dimensional program:
\begin{equation}\label{eq:PEP_FTRL_1}\begin{aligned}
\sup_{\substack{
				x_\star, \{x_t\}_{t=1,\ldots,T}\\
				\{(f_t,f_t^\star)\}_{t=1,\ldots,T}\\
				\{(g_t,g_t^\star)\}_{t=1,\ldots,T}\\
				s_\star, \{s_t\}_{t=1,\ldots,T}\\
				d\in\mathbb{N}}} \,
	& \sum_{t=1}^T f_t - f_t^\star
	\\
\text{subject to: } 
	&  f_t \geq f_t^\star + \langle g_t^\star, x_t - x_\star \rangle \text{ for }t=1,\ldots,T,\\
	&  f_t^\star \geq f_t + \langle g_t, x_\star - x_t \rangle \text{ for }t=1,\ldots,T,\\
	& \Vert g_t \Vert \leq L \text{ and } \Vert g_t^\star \Vert \leq L
		\text{ for }t=1,\ldots,T,\\
	&\Vert x_t-x_\star\Vert\leq D, \text{ for } t=1,\ldots,T,\\
	&\Vert x_i-x_j\Vert\leq D, \text{ for } i,j=1,\ldots,T,\\
	& \langle s_t, x_\star - x_t \rangle \leq 0 
		\text{ and } \langle s_\star, x_t - x_\star \rangle \leq 0
		\text{ for }t=1,\ldots,T, \\
	& \langle s_i, x_j - x_i \rangle \leq 0 \text{ for } i,j=1,\ldots,T,\\
	& (x_t - x_1) + \eta \sum_{i=1}^{t-1} g_i + s_t = 0 \text{ for } t=1,\ldots,T. \\
\end{aligned}
\end{equation}
Note that the last constraint corresponds to the optimality condition of FTRL at each time $t$.

Moreover, as the values of the gradient $g_t^\star$ are never used, we can impose $g_t^\star=g_t$ for all $t$,
which gives a tight relaxation of the problem.
Indeed, this corresponds to consider only linear cost functions.
In particular, we get that among the worst-case instances, there is a worst-case instance involving only linear cost functions.
Hence, we get the following simpler reformulation of~\eqref{eq:PEP_FTRL}:
\begin{equation}\label{eq:PEP_FTRL_1bis}\begin{aligned}
\sup_{\substack{
				x_\star, \{x_t\}_{t=1,\ldots,T}\\
				\{g_t\}_{t=1,\ldots,T}\\
				s_\star, \{s_t\}_{t=1,\ldots,T}\\
				d\in\mathbb{N}}} \,
	& \sum_{t=1}^T \langle g_t, x_t - x_\star \rangle
	\\
\text{subject to: } 
	& \Vert g_t \Vert \leq L 
		\text{ for }t=1,\ldots,T,\\
	&\Vert x_t-x_\star\Vert\leq D, \text{ for } t=1,\ldots,T,\\
	&\Vert x_i-x_j\Vert\leq D, \text{ for } i,j=1,\ldots,T,\\
	& \langle s_t, x_\star - x_t \rangle \leq 0 
		\text{ and } \langle s_\star, x_t - x_\star \rangle \leq 0
		\text{ for }t=1,\ldots,T, \\
	& \langle s_i, x_j - x_i \rangle \leq 0 \text{ for } i,j=1,\ldots,T,\\
	& (x_t - x_1) + \eta \sum_{i=1}^{t-1} g_i + s_t = 0 \text{ for } t=1,\ldots,T. \\
\end{aligned}
\end{equation}
We now reformulate this program as an SDP.
Define the matrix:
\begin{equation*}
P= [ x_\star \vert x_1 \vert \cdots \vert x_T \vert g_1 \vert \cdots \vert g_T \vert s_\star  ],
\end{equation*}
and let $G = P^\transp P \succeq 0$ denote the Gram matrix containing all dots products of those vectors, which is of dimension $(2T+2) \times (2T+2)$.
Let $\{\bar{x}_t\}_{t=1,\dots,T}$ be the vector in $\R^{2T+2}$
such that $\bar{x}_i^\transp G \bar{x}_j = \langle x_i, x_j \rangle$,
and similarly for the other pair of variables.
Remark that we do not include the variables $\{s_t\}_{t=1,\dots,T}$ in $P$
as they are redundant due to the last constraint in~\eqref{eq:PEP_FTRL_1bis}.
Instead, we just define vectors $\{\bar{s}_t\}_{t=1,\dots,T}$
such that $\bar{s}_t = - (\bar{x}_t - \bar{x}_1) - \eta \sum_{i=1}^{t-1} \bar{g}_t$ for all $t$.
This allows us to redefine~\eqref{eq:PEP_FTRL_1bis} as
the following equivalent SDP:
\begin{equation}\label{eq:SDP_FTRL_1}\begin{aligned}
\sup_{\substack{
				G \succeq 0
			}} \,
	& \sum_{t=1}^T  \bar{g}_t^\transp G (\bar{x}_t - \bar{x}_\star)
	\\
\text{subject to: } 
	&  \bar{g}_t^\transp G \bar{g}_t \leq L 
		\text{ for }t=1,\ldots,T,\\
	& (\bar{x}_t-\bar{x}_\star)^\transp G (\bar{x}_t-\bar{x}_\star) \leq D, \text{ for } t=1,\ldots,T,\\
	& (\bar{x}_i-\bar{x}_j)^\transp G (\bar{x}_i-\bar{x}_j) \leq D, \text{ for } i,j=1,\ldots,T,\\
	&  \bar{s}_t^\transp G (\bar{x}_\star - \bar{x}_t ) \leq 0 
		\text{ and }  \bar{s}_\star^\transp G (\bar{x}_t - \bar{x}_\star) \leq 0
		\text{ for }t=1,\ldots,T, \\
	& \bar{s}_i^\transp G (\bar{x}_j - \bar{x}_i) \leq 0 \text{ for } i,j=1,\ldots,T.\\
\end{aligned}
\end{equation}
Note that the variable $d\in\mathbb{N}$ in~\eqref{eq:PEP_FTRL_1bis}
imposes that the rank of $G$ is upper bounded by $d$ in~\eqref{eq:SDP_FTRL_1},
but this condition disappear when taking the supremum other $d\in\mathbb{N}$.

The dual variables associated to each constraints in~\eqref{eq:SDP_FTRL_1}
indicate which of those inequalities are used in deriving a proof
of the upper bound on the regret of FTRL.
Choosing $\eta = D / (L \sqrt{T})$,
and running numerical simulations for moderate values of $T$
indicates that some dual variables have constantly small values
(several orders smaller than other dual variables).
Hence, relaxing the constraints associated to those dual variables with small values
gives us a new upper bound on the regret (indeed, with same value)
via the following SDP:
\begin{equation}\label{eq:SDP_FTRL_2}\begin{aligned}
\sup_{\substack{
				G \succeq 0
			}} \,
	& \sum_{t=1}^T  \bar{g}_t^\transp G (\bar{x}_t - \bar{x}_\star)
	\\
\text{subject to: } 
	&  \bar{g}_t^\transp G \bar{g}_t \leq L 
		\text{ for }t=1,\ldots,T,\\
	& (\bar{x}_1-\bar{x}_\star)^\transp G (\bar{x}_1-\bar{x}_\star) \leq D, \\
	&  \bar{s}_t^\transp G (\bar{x}_\star - \bar{x}_T ) \leq 0,  \\
	& \bar{s}_t^\transp G (\bar{x}_{t+1} - \bar{x}_t) \leq 0 \text{ for } t=1,\ldots,T-1.\\
\end{aligned}
\end{equation}
Then, running numerical simulations for~\eqref{eq:SDP_FTRL_2} 
with small values of $T$,
we observe that most dual variables take the same values, 
thus we group them accordingly and relax those constraints by summing them,
which gives us a new upper bound on the regret (indeed, with same value)
via the following SDP:
\begin{equation}\label{eq:SDP_FTRL_3}\begin{aligned}
\sup_{\substack{
				G \succeq 0
			}} \,
	& \sum_{t=1}^T  \bar{g}_t^\transp G (\bar{x}_t - \bar{x}_\star)
	\\
\text{subject to: } 
	&  \sum_{t=1}^T \bar{g}_t^\transp G \bar{g}_t \leq L T ,\\
	& (\bar{x}_1-\bar{x}_\star)^\transp G (\bar{x}_1-\bar{x}_\star) \leq D, \\
	&   \bar{s}_t^\transp G (\bar{x}_\star - \bar{x}_T ) + \sum_{t=1}^{T-1} \bar{s}_t^\transp G (\bar{x}_{t+1} - \bar{x}_t) \leq 0 .\\
\end{aligned}
\end{equation}
We now observe that the values of the dual variables for 
those three remaining constraints are
respectively $\eta/2$, $1/(2\eta)$ and $1/\eta$.
Looking at the two constraints involving constants $L$ and $D$,
we can already infer that the upper bound on the regret will be:
\begin{equation*}
\frac{\eta}{2} \sum_{t=1}^T \Vert g_t \Vert^2
	+ \frac{1}{2\eta} \Vert x_\star - x_1 \Vert^2,
\end{equation*}
which is upper bounded by $L D \sqrt{T}$ for $\eta = D / (L \sqrt{T})$.
From~\eqref{eq:SDP_FTRL_3}, 
we now that to prove this upper bound,
we only need to use the third constraint inequality (times $1/\eta$)
and some scalar product / euclidean norm inequalities
(which do not include any information on the relation between vectors).
In general, finding which scalar product inequalities to use
can be done by studying the Cholesky decomposition of the matrix dual variable for the SDP constraint
(see matrix $S$ in~\eqref{eq:dual_SDP_generic}, 
and see Appendix~\ref{a:proof_design} for an explanation on how we used the Cholesky decomposition in the case of OFW).
In this case, the scalar product inequalities to use are rather simple to guess,
and thus we will not need to look at this dual matrix variable.

Hence, we have all the ingredients and we are now ready 
to do the proof of Lemma~\ref{lemma_bound_FTRL}.

\begin{proof}[Proof of Lemma~\ref{lemma_bound_FTRL}]
For simplicity, we write $x_{T+1} = x_\star$.
We start by summing the boundary inequalities $\langle s_{t}, x_{t+1} - x_t \rangle \leq 0$ for $t\in\llbracket 1,T\rrbracket$.
As we have:
\begin{equation*}
\begin{aligned}
\bar{s}_t^\transp G (\bar{x}_\star - \bar{x}_T ) + \sum_{t=1}^{T-1} \bar{s}_t^\transp G (\bar{x}_{t+1} - \bar{x}_t)
& = \sum_{t=1}^{T} \langle s_t , x_{t+1} - x_t \rangle \\
& = 	- \sum_{t=1}^{T}  \langle x_t - x_1, x_{t+1} - x_t \rangle 
	- \eta \sum_{t=1}^{T} \sum_{i = 1}^{t-1} \langle g_i , x_{t+1} - x_t \rangle \\
& = - \sum_{t=1}^{T}  \langle x_t - x_1, x_{t+1} - x_t \rangle 
	-\eta \sum_{i = 1}^{T-1}  \langle g_i , x_{T+1} - x_{i+1} \rangle .	
\end{aligned}
\end{equation*}
this sum of inequalities gives the following inequality (which corresponds to the third inequality in~\eqref{eq:SDP_FTRL_3}):
\begin{equation*}
\eta \sum_{t = 1}^{T-1}  \langle g_t , x_{t+1} - x_\star \rangle
\leq \sum_{t=1}^{T-1}  \langle x_t - x_1, x_{t+1} - x_t \rangle 
	+ \langle x_T - x_1, x_\star - x_T \rangle .
\end{equation*}

Now, we use this inequality and
the inequality 
$\langle u, v \rangle \leq \frac{\eta}{2} \Vert u \Vert^2 + \frac{1}{2\eta} \Vert v \Vert^2$,
as well as the convexity of the cost functions $\ell_t$,
to upper bound the regret for FTRL:\allowdisplaybreaks[1]
\begin{align*}
R_T(x_1,\ldots,x_T; x_\star) 
& \leq \sum_{t=1}^T \langle g_t, x_t - x_\star \rangle \\
& = \sum_{t=1}^T \langle g_t, x_t - x_{t+1} \rangle
	+ \sum_{t=1}^T \langle g_t, x_{t+1} - x_\star \rangle \\
& \leq \sum_{t=1}^T \langle g_t, x_t - x_{t+1} \rangle
	+ \frac{1}{\eta} \sum_{t=1}^T \langle x_t - x_1, x_{t+1} - x_t \rangle \\
& \leq \frac{\eta}{2} \sum_{t=1}^T \Vert g_t \Vert^2
	+ \frac{1}{2\eta} \sum_{t=1}^T \Vert x_t - x_{t+1} \Vert^2
	+ \frac{1}{\eta} \sum_{t=1}^T \langle x_t - x_1, x_{t+1} - x_t \rangle \\
& = \frac{\eta}{2} \sum_{t=1}^T \Vert g_t \Vert^2
	+ \frac{1}{2\eta} \sum_{t=1}^T \Bigl\{ 
		\Vert x_{t+1} - x_1 \Vert^2 - \Vert x_{t} - x_1 \Vert^2
	\Bigr\} \\
& = \frac{\eta}{2} \sum_{t=1}^T \Vert g_t \Vert^2
	+ \frac{1}{2\eta} 	\Vert x_\star - x_1 \Vert^2,
\end{align*}
\allowdisplaybreaks[0]
which concludes the proof.
\end{proof}

\subsection{Deriving proof for OGD from the PEP methodology}

Now, we turn to OGD, and again use~\eqref{eq:PEP}
to derive a proof that $B_T(\eta)$ is an upper bound of the regret for OGD
when using the same step-size for all times (i.e.\ $\eta_t = \eta$ for all $t$)
and with static regret.

We now turn to study OGD using the PEP methodology.
We consider the following version of OGD with a single horizon-dependent step-size
(where $\Pi_{\cK}$ denotes the euclidean projection on the convex closed set $\cK$).
\begin{algorithm}[H]
\caption{Online Gradient Descent (OGD)}
\label{OGD_alg_appendix}
\begin{algorithmic}[1]
\Require $T\geq 1$,~ $x_1 \in \cK$, $\eta \geq 0$
\For{$t=1$ to $T$} 
    \State Play $x_t$, pay cost $\ell_t(x_t)$,  observe $g_t = \nabla\ell_t(x_t)$.
    \State $x_{t+1} \gets \Pi_{\cK}(x_t - \eta g_t) =  \argmin_{x\in \cK}  \Vert x - (x_t - \eta g_t) \Vert^2 $
\EndFor
\end{algorithmic}
\end{algorithm}

Adapting the PEP formulation~\eqref{eq:PEP} of OFW to the case of OGD, we get:
\begin{equation}\label{eq:PEP_OGD}
\begin{aligned}
B_T(\eta)\triangleq
\sup_{\substack{\cK, \{\ell_t\}_{t\in\llbracket 1,T \rrbracket}\\ 
			x_\star, \{x_t\}_{t\in\llbracket 1,T \rrbracket}\\
			d\in\mathbb{N}}} \,
&
R_T(x_1,\ldots,x_T; x_\star)\\
\text{subject to: } 
&  \ell_t \text{ is convex and $L$-Lipschitz for }t\in\llbracket 1,T \rrbracket,\\
& \cK \text{ is a non-empty closed convex set of $\mathbb{R}^d$,}\\ 
&\mathrm{Diam}(\cK)\leq D,\\
& \{x_t\}_{t=1,\ldots,T} \text{ is generated by Algorithm~\ref{OGD_alg_appendix}}.\\
\end{aligned}
\end{equation}

We now show how to use~\eqref{eq:PEP_OGD}
to derive a proof that $B_T(\eta)$ is an upper bound of the regret for OGD.

\begin{lemma}\label{lemma_bound_OGD}
Let $T\geq 1$.
Assume that the cost functions $\ell_t$ are convex and $L$-Lipschitz for all $t\in \llbracket 1, T \rrbracket$,
and that the convex closed domain $\cK$ of feasible points has a diameter bounded by $D$.
Then, for any $x_\star \in\cK$, the following upper bound 
on the regret of the OGD Algorithm~\ref{OGD_alg_appendix} holds:
\begin{equation*}
R_T(x_1,\cdots, x_T; x_\star)
\leq \frac{\eta}{2} \sum_{t=1}^T \Vert g_t \Vert^2
	+ \frac{1}{2\eta} \Vert x_\star - x_1 \Vert^2.
\end{equation*}
In particular, for $\eta = D / (L \sqrt{T})$, we get that the regret is upper bounded by
$D L \sqrt{T}$.
\end{lemma}

We start by explaining how to use~\eqref{eq:PEP_OGD} to infer
the proof of Lemma~\ref{lemma_bound_OGD},
and then we will do the proof itself.
As the idea is similar to the case of FTRL, 
we just give the general idea.

Using a similar argument to that of Section~\ref{s:wc_regret_construction} for OFW
and Appendix~\ref{s:PEP_for_FTRL}, the PEP program~\eqref{eq:PEP_OGD} for OGD can be 
reformulated as the following finite dimensional program:
\begin{equation*}\label{eq:PEP_OGD_bis}\begin{aligned}
\sup_{\substack{
				x_\star, \{x_t\}_{t=1,\ldots,T}\\
				\{g_t\}_{t=1,\ldots,T}\\
				s_\star, \{s_t\}_{t=1,\ldots,T}\\
				d\in\mathbb{N}}} \,
	& \sum_{t=1}^T \langle g_t, x_t - x_\star \rangle
	\\
\text{subject to: } 
	& \Vert g_t \Vert \leq L 
		\text{ for }t=1,\ldots,T,\\
	&\Vert x_t-x_\star\Vert\leq D, \text{ for } t=1,\ldots,T,\\
	&\Vert x_i-x_j\Vert\leq D, \text{ for } i,j=1,\ldots,T,\\
	& \langle s_t, x_\star - x_t \rangle \leq 0 
		\text{ and } \langle s_\star, x_t - x_\star \rangle \leq 0
		\text{ for }t=1,\ldots,T, \\
	& \langle s_i, x_j - x_i \rangle \leq 0 \text{ for } i,j=1,\ldots,T,\\
	& x_t - ( x_{t-1} - \eta  g_{t-1}) + s_t = 0 \text{ for } t=2,\ldots,T. \\
\end{aligned}
\end{equation*}

We then turn this PEP into an SDP using the Gram matrix $G$
associated to the matrix $P$ containing all vectors (without redundancy) as before.
We define the vectors $\bar{x}_\star$, $\bar{s}_\star$,
$\bar{x}_t$ and $\bar{g}_t$ for $t=1,\dots,T$ as before.
As the optimality constraint is different between OGD and FTRL,
here we have $\bar{s}_t = - \bar{x}_t + (\bar{x}_{t-1} - \eta \bar{g}_{t-1})$
for $t=2,\dots,T$.
(Note that the value of $s_1$ can be set freely to zero in the PEP above, 
so we can simply set $\bar{s}_1=0$.)
This allows us to redefine~\eqref{eq:PEP_OGD} as
the following equivalent SDP:
\begin{equation}\label{eq:SDP_OGD_1}\begin{aligned}
\sup_{\substack{
				G \succeq 0
			}} \,
	& \sum_{t=1}^T  \bar{g}_t^\transp G (\bar{x}_t - \bar{x}_\star)
	\\
\text{subject to: } 
	&  \bar{g}_t^\transp G \bar{g}_t \leq L 
		\text{ for }t=1,\ldots,T,\\
	& (\bar{x}_t-\bar{x}_\star)^\transp G (\bar{x}_t-\bar{x}_\star) \leq D, \text{ for } t=1,\ldots,T,\\
	& (\bar{x}_i-\bar{x}_j)^\transp G (\bar{x}_i-\bar{x}_j) \leq D, \text{ for } i,j=1,\ldots,T,\\
	&  \bar{s}_t^\transp G (\bar{x}_\star - \bar{x}_t ) \leq 0 
		\text{ and }  \bar{s}_\star^\transp G (\bar{x}_t - \bar{x}_\star) \leq 0
		\text{ for }t=1,\ldots,T, \\
	& \bar{s}_i^\transp G (\bar{x}_j - \bar{x}_i) \leq 0 \text{ for } i,j=1,\ldots,T.\\
\end{aligned}
\end{equation}
Note that~\eqref{eq:SDP_OGD_1} for OGD looks identical to~\eqref{eq:SDP_FTRL_1} for FTRL,
the only (implicit) difference is the definition of the vectors $\bar{s}_t$
for $t=1,\dots,T$.

Choosing $\eta = D / (L \sqrt{T})$,
running numerical simulations of~\ref{eq:SDP_OGD_1} for small values of $T$,
relaxing constraints with small dual variable values
and further relaxing by grouping constraints with the same dual variables values,
we get a new upper bound on the regret of OGD
(indeed, with same value) via the following SDP:
\begin{equation}\label{eq:SDP_OGD_2}\begin{aligned}
\sup_{\substack{
				G \succeq 0
			}} \,
	& \sum_{t=1}^T  \bar{g}_t^\transp G (\bar{x}_t - \bar{x}_\star)
	\\
\text{subject to: } 
	&  \sum_{t=1}^T \bar{g}_t^\transp G \bar{g}_t \leq L T ,\\
	& (\bar{x}_1-\bar{x}_\star)^\transp G (\bar{x}_1-\bar{x}_\star) \leq D, \\
	&   \sum_{t=2}^{T} \bar{s}_t^\transp G (\bar{x}_\star - \bar{x}_t) \leq 0 .\\
\end{aligned}
\end{equation}
We now observe that the values of the dual variables for 
those three remaining constraints are
respectively $\eta/2$, $1/(2\eta)$ and $1/\eta$.
Looking at the two constraints involving constants $L$ and $D$,
we can already infer that the upper bound on the regret will be:
\begin{equation*}
\frac{\eta}{2} \sum_{t=1}^T \Vert g_t \Vert^2
	+ \frac{1}{2\eta} \Vert x_\star - x_1 \Vert^2,
\end{equation*}
which is upper bounded by $L D \sqrt{T}$ for $\eta = D / (L \sqrt{T})$.
From~\eqref{eq:SDP_OGD_2}, 
we now that to prove this upper bound,
we only need to use the third constraint inequality (times $1/\eta$)
and some scalar product / euclidean norm inequalities
(which we will easily infer as for FTRL).

Hence, we have all the ingredients and we are now ready 
to do the proof of Lemma~\ref{lemma_bound_OGD}.

\begin{proof}
We start by summing the boundary inequalities $\langle s_{t}, x_\star - x_t \rangle \leq 0$ 
where $s_t = - x_t + (x_{t-1} - \eta g_{t-1})$ for $t\in\llbracket 2,T\rrbracket$.
This gives us the following inequality (which corresponds to the third constraint in~\eqref{eq:SDP_OGD_2}):
\begin{align*}
0 \geq 
\sum_{t=2}^{T} \langle s_t, x_\star - x_t \rangle
= \sum_{t=2}^{T} \langle x_{t-1} - x_t , x_\star - x_t \rangle
	- \eta \sum_{t=2}^{T} \langle g_{t-1}, x_\star - x_t \rangle ,
\end{align*}
which we more conveniently rewrite as:
\begin{equation*}
\sum_{t=2}^{T} \langle g_{t-1}, x_t - x_\star  \rangle 
\leq \frac{1}{\eta} \sum_{t=2}^{T} \langle x_{t-1} - x_t , x_t - x_\star \rangle .
\end{equation*}

Now, we use this inequality and
the inequality 
$\langle u, v \rangle \leq \frac{\eta}{2} \Vert u \Vert^2 + \frac{1}{2\eta} \Vert v \Vert^2$,
as well as the convexity of the cost functions $\ell_t$,
to upper bound the regret for OGD:
\allowdisplaybreaks[1]
\begin{align*}
R_T(x_1,\ldots,x_T; x_\star) 
& \leq \sum_{t=1}^T \langle g_t, x_t - x_\star \rangle \\
& = \sum_{t=1}^T \langle g_t, x_t - x_{t+1} \rangle
	+ \sum_{t=1}^{T-1} \langle g_t, x_{t+1} - x_\star \rangle \\
& = \sum_{t=1}^T \langle g_t, x_t - x_{t+1} \rangle
	+ \sum_{t=2}^{T} \langle g_{t-1}, x_{t} - x_\star \rangle \\
& \leq \frac{\eta}{2} \sum_{t=1}^T \Vert g_t \Vert^2
	+ \frac{1}{2\eta} \sum_{t=1}^T \Vert x_t - x_{t+1} \Vert^2
	+ \frac{1}{\eta} \sum_{t=2}^{T} \langle x_{t-1} - x_t , x_t - x_\star \rangle \\
& \leq \frac{\eta}{2} \sum_{t=1}^T \Vert g_t \Vert^2
	+ \frac{1}{2\eta} \sum_{t=1}^T \Vert x_t - x_{t+1} \Vert^2
	+ \frac{1}{\eta} \sum_{t=1}^{T-1} \langle x_{t} - x_{t+1} , x_{t+1} - x_\star \rangle \\
& = \frac{\eta}{2} \sum_{t=1}^T \Vert g_t \Vert^2
	+ \frac{1}{2\eta} \sum_{t=1}^{T-1} \Bigl\{
		\Vert x_t - x^* \Vert^2
		- \Vert x_{t+1} - x^* \Vert^2
	\Bigr\}
	+ \frac{1}{2\eta} \Vert x_T - x^* \Vert^2 \\
& = \frac{\eta}{2} \sum_{t=1}^T \Vert g_t \Vert^2
	+ \frac{1}{2\eta} \Vert x_1 - x^* \Vert^2 ,
\end{align*}
\allowdisplaybreaks[0]
which concludes the proof.
\end{proof}

\subsection{Adapting the proof for FTRL to Bregman divergences}
We now adapt the proof for FTRL of Lemma~\ref{lemma_bound_FTRL}
to the case of general norm and Bregman divergences.
This illustrates the fact that the proof we derived from the PEP~\eqref{eq:PEP_FTRL}
is simple and clean enough to easily be adapted to a more general setting.

Let $\Vert \cdot \Vert$ denote a general norm,
and denote by $\Vert \cdot \Vert_*$ its associated dual norm defined by 
$\Vert g \Vert_* = \sup_{x \ : \ \Vert x \Vert \leq 1} \langle g, x \rangle$.
In particular, this definition implies the generalised Cauchy-Schwarz inequality
$ \langle g, x \rangle \leq \Vert g \Vert_* \Vert x \Vert$.
Let $\psi$ be a $1$-strongly convex function w.r.t.\ the norm $\Vert \cdot \Vert$, that is:
\[ \psi(x) \geq \psi(y) + \langle \nabla\psi(y), x-y \rangle + \frac{1}{2} \Vert x-y \Vert^2 .\]
The Bregman divergence associated to $\psi$ is then defined as:
\[ B_\psi(x; y) = \psi(x) - \psi(y) - \langle \nabla\psi(y), x-y \rangle .\]
The Bregman divergence $B_\psi(x;y)$ 
will be used as a regularization function to replace 
the term $\frac{1}{2} \Vert x-y \Vert^2$ used before.

Then, we define the iterates of FTRL (on linearized losses) 
with Bregman divergence regularization for all $t$ as:
\begin{equation}\label{eq:def_FTRL_bregman}
x_{t+1}=\argmin_{x\in\cK}
	B_\psi(x;x_1)
	+\eta \sum_{\tau=1}^t \langle \nabla \ell_\tau(x_\tau),\,x\rangle .
\end{equation}
In particular, the sub-gradients $s_t \in \partial i_\cK(x_t)$ is defined
by the optimality condition:
\begin{equation}\label{eq:def_s_t_Bregman}
\nabla\psi(x_t) - \nabla\psi(x_1) + \eta \sum_{\tau=1}^{t-1} g_\tau + s_t = 0 .
\end{equation}

We can now state the following lemma upper bounding the regret of FTRL with Bregman divergence,
and whose proof is an adaptation of the proof of Lemma~\ref{lemma_bound_FTRL}.

\begin{lemma}\label{lemma_bound_FTRL_Bregman}
Let $T\geq 1$.
Assume that the cost functions $\ell_t$ are convex and $L$-Lipschitz for all $t\in \llbracket 1, T \rrbracket$,
and that the convex closed domain $\cK$ of feasible points is such that $\sup_{x,y\in\cK} B_\psi(x;y) \leq \frac{1}{2}D^2$ for some $D$.
Then, for any $x_\star \in\cK$, the following upper bound 
on the regret of the FTRL Algorithm defined by~\eqref{eq:def_FTRL_bregman} holds:
\begin{equation*}
R_T(x_1,\cdots, x_T; x_\star)
\leq \frac{\eta}{2} \sum_{t=1}^T \Vert g_t \Vert_*^2
	+ \frac{1}{\eta} B_\psi(x_\star ; x_1).
\end{equation*}
In particular, for $\eta = D / (L \sqrt{T})$, we get that the regret is upper bounded by
$D L \sqrt{T}$.
\end{lemma}

\begin{proof}
For simplicity, we write $x_{T+1} = x_\star$.
As in the proof of Lemma~\ref{lemma_bound_FTRL},
we start by summing the boundary inequalities $\langle s_{t}, x_{t+1} - x_t \rangle \leq 0$ for $t\in\llbracket 1,T\rrbracket$ where $s_t$ is defined in~\eqref{eq:def_s_t_Bregman}.
This gives us:
\begin{align*}
0 & \geq \sum_{t=1}^{T} \langle s_t , x_{t+1} - x_t \rangle \\
& = - \sum_{t=1}^{T}  \langle \nabla\psi(x_t) - \nabla\psi(x_1), x_{t+1} - x_t \rangle 
	- \eta \sum_{t=1}^{T} \sum_{i = 1}^{t-1} \langle g_i , x_{t+1} - x_t \rangle \\
& = - \sum_{t=1}^{T}  \langle \nabla\psi(x_t) - \nabla\psi(x_1), x_{t+1} - x_t \rangle 
	-\eta \sum_{i = 1}^{T-1}  \langle g_i , x_{T+1} - x_{i+1} \rangle .	
\end{align*}
We then rewrite this inequality as:
\begin{equation*}
 \sum_{t = 1}^{T-1}  \langle g_t , x_{t+1} - x_\star \rangle
\leq \frac{1}{\eta} \sum_{t=1}^{T}  \langle \nabla\psi(x_t) - \nabla\psi(x_1), x_{t+1} - x_t \rangle .
\end{equation*}
We will need the following inequality, which comes directly from rewriting
the definition of $1$-strong convexity of~$\psi$:
\[ \frac{1}{2} \Vert x-y \Vert^2 
\leq \psi(x) - \psi(y) - \langle \nabla\psi(y), x-y \rangle .
\]

Now, we use those inequalities and
the inequality:
\[ \langle g, x \rangle 
\leq \Vert g \Vert_* \Vert x \Vert
\leq \frac{\eta}{2} \Vert g \Vert_*^2 + \frac{1}{2\eta} \Vert x \Vert^2 ,
\]
as well as the convexity of the cost functions $\ell_t$,
to upper bound the regret for FTRL (writing $R_T$ for $R_T(x_1,\ldots,x_T; x_\star)$):
\allowdisplaybreaks[1]
\begin{align*}
R_T & \leq \sum_{t=1}^T \langle g_t, x_t - x_\star \rangle \\
& = \sum_{t=1}^T \langle g_t, x_t - x_{t+1} \rangle
	+ \sum_{t=1}^T \langle g_t, x_{t+1} - x_\star \rangle \\
& \leq \sum_{t=1}^T \langle g_t, x_t - x_{t+1} \rangle
	+ \frac{1}{\eta} \sum_{t=1}^T \langle \nabla\psi(x_t) - \nabla\psi(x_1), x_{t+1} - x_t \rangle \\
& \leq \frac{\eta}{2} \sum_{t=1}^T \Vert g_t \Vert_*^2
	+ \frac{1}{2\eta} \sum_{t=1}^T \Vert x_t - x_{t+1} \Vert^2
	+ \frac{1}{\eta} \sum_{t=1}^T \langle \nabla\psi(x_t) - \nabla\psi(x_1), x_{t+1} - x_t \rangle \\
& \leq \frac{\eta}{2} \sum_{t=1}^T \Vert g_t \Vert_*^2
	+ \frac{1}{\eta} \sum_{t=1}^T \Bigl\{ 
		\psi(x_{t+1}) - \psi(x_t)
		- \langle \nabla\psi(x_t), x_{t+1} - x_t \rangle
	\Bigr\} 
	+ \frac{1}{\eta} \sum_{t=1}^T \langle \nabla\psi(x_t) - \nabla\psi(x_1), x_{t+1} - x_t \rangle \\
& = \frac{\eta}{2} \sum_{t=1}^T \Vert g_t \Vert_*^2
	+ \frac{1}{\eta}  \Bigl\{ 
		\psi(x_\star) - \psi(x_1)
		- \langle \nabla\psi(x_1), x_\star - x_1 \rangle
	\Bigr\} \\
& = \frac{\eta}{2} \sum_{t=1}^T \Vert g_t \Vert_*^2
	+ \frac{1}{\eta} 	B_\psi( x_\star ; x_1 ),
\end{align*}
\allowdisplaybreaks[0]
which concludes the proof.
\end{proof}

\subsection{Generic SDP for PEP}

In this section, we present the following generic form of an SDP arising from a PEP.
\begin{equation}\label{eq:SDP_generic}\tag{SDP-generic}
\begin{aligned}
\sup_{\substack{
				G \succeq 0 \\ F
			}} \,
	&  \tr(A_{\text{obj}}G) +  \langle a_{\text{obj}}, F \rangle
	\\
\text{subject to: } 
	&  \tr(A_i G) + \langle a_{i}, F \rangle \leq b_i , \text{ for } i=1,\ldots,N. \\
\end{aligned}
\end{equation}
The semidefinite positive matrix variable $G$ corresponds to the Gram matrix of the gradients and points,
and the vector variable $F$ corresponds to the sampled function values. 
Note that in our case, the variable $F$ can be removed as we get a tight relaxation of the PEP by considering only linear cost functions $\ell_t$.

We then form the (generic form) Lagrangian function associated to the SDP~\eqref{eq:SDP_generic}
by dualizing all the constraints except the semidefinite positive constraint $G\succeq 0$.
\begin{align*}
L(G,F; \lambda_{1:N}) 
& = \tr(A_{\text{obj}}G) +  \langle a_{\text{obj}}, F \rangle
	- \sum_{i=1}^N \lambda_i \Bigl[
			\tr(A_i G) + \langle a_i, F \rangle - b_i
		\Bigr] \\
& = \sum_{i=1}^N b_i \lambda_i 
	+ \tr\left( \left(A_{\text{obj}} - \sum_{i=1}^N \lambda_i A_i\right) G \right) 
	+  \left\langle a_{\text{obj}} - \sum_{i=1}^N  \lambda_i a_i, F \right\rangle
\end{align*}

By maximizing this Lagrangian function over the primal variables $G$ and $F$,
we can form the Lagrange dual problem to the generic form SDP~\eqref{eq:SDP_generic}
which is also an SDP.
\begin{equation}\label{eq:dual_SDP_generic}\tag{dual-SDP-generic}
\begin{aligned}
\inf_{\substack{
				\lambda_i \geq 0, i=1,\dots, N
			}} \,
	&  \sum_{i=1}^N b_i \lambda_i 
	\\
\text{subject to: } 
	& a_{\text{obj}} = \sum_{i=1}^N  \lambda_i a_i , 
	\ \text{   and   } \  S:= A_{\text{obj}} - \sum_{i=1}^N \lambda_i A_i \preceq 0 . \\
\end{aligned}
\end{equation}
Note that in this dual SDP, there is a semidefinite positive constraint involving a matrix $S$
which corresponds to the dual variable of the primal semidefinite positive constraint $G\succeq 0$.
By studying the numerical values of the Cholesky decomposition of this matrix $S$ at optimum,
we can get numerical insights on the scalar product / euclidean norm inequalities
that need to be used in proofs of regret upper bounds.

\end{document}